\documentclass[12pt]{amsart}
\usepackage{epsfig,color}

\makeatother

\theoremstyle{definition}

\def\fnum{equation} 
\newtheorem{Thm}[\fnum]{Theorem}
\newtheorem{Cor}[\fnum]{Corollary}

\newtheorem{Lem}[\fnum]{Lemma}

\newtheorem{Rem}[\fnum]{Remark}
\newtheorem{Pro}[\fnum]{Proposition}

\numberwithin{equation}{section}

\newcommand{\Vol}{{\text{Vol}}}
\newcommand{\V}{{\text{V}}}

\newcommand{\nn}{{\bf{n}}}
\newcommand{\Id}{{\bf{Id}}}
\newcommand{\Ric}{{\text{Ric}}}

\newcommand{\Tr}{{\text{Tr}}}

\newcommand{\II}{{\text{II}}}

\newcommand{\Hess}{{\text {Hess}}}

\def\RR{{\bold R}}

\def\SS{{\bold S}}

\newcommand{\dv}{{\text {div}}}
\newcommand{\e}{{\text {e}}}

\newcommand{\cT}{{\mathcal{T}}}
\newcommand{\cTDiff}{{\mathcal{T}_{\mathcal{D}}}}
\newcommand{\cTperp}{{\mathcal{T}_{\perp}}}

\newcommand{\cA}{{\mathcal{A}}}
\newcommand{\ca}{{\mathcal{T}}}
\newcommand{\cB}{{\mathcal{B}}}

\newcommand{\cE}{{\mathcal{E}}}
\newcommand{\cF}{{\mathcal{R}}}

\newcommand{\cL}{{\mathcal{L}}}
\newcommand{\cO}{{\mathcal{O}}}

\newcommand{\cN}{{\mathcal{N}}}

\newcommand{\cD}{{\mathcal{D}}}
\newcommand{\cR}{{\mathcal{R}}}
\newcommand{\cU}{{\mathcal{U}}}

\newcommand{\eqr}[1]{(\ref{#1})}

\title[Uniqueness of tangent cones]{On uniqueness of tangent cones for Einstein manifolds}

 \author{Tobias Holck Colding}%
\address{MIT, Dept. of Math.\\
77 Massachusetts Avenue, Cambridge, MA 02139-4307.}

\author{William P. Minicozzi II}%
\address{Johns Hopkins University\\
Dept. of Math.\\
3400 N. Charles St.\\
Baltimore, MD 21218.}

\thanks{The authors
were partially supported by NSF Grants DMS  11040934, DMS
0906233,  and NSF FRG grants DMS 
 0854774 and DMS 0853501}

\email{colding@math.mit.edu and minicozz@math.jhu.edu}

\begin{document}

\maketitle

\begin{abstract}
We show that for any Ricci-flat manifold with Euclidean volume growth the tangent cone at infinity is unique if one tangent cone has a smooth cross-section.   Similarly, for any noncollapsing limit of Einstein manifolds with uniformly bounded Einstein constants, we show that local tangent cones are unique if one tangent cone has a smooth cross-section. 
\end{abstract}

\setcounter{section}{-1}

\section{Introduction}

By Gromov's compactness theorem, \cite{GLP}, \cite{G}, if $M$ is an $n$-dimensional manifold with nonnegative Ricci curvature, then any sequence of rescalings $(M,r_i^{-2}g)$, where $r_i\to \infty$, has a subsequence that converges in the Gromov-Hausdorff topology to a length space.   Any such limit is said to be a tangent cone at infinity of $M$.  Compactness follows from that
\begin{equation}
r^{-n}\,\Vol (B_r(x))
\end{equation}
is monotone nonincreasing in the radius $r$ of the ball $B_r(x)$ for any fixed $x\in M$ by the Bishop-Gromov volume comparison.    As $r$ tends to $0$, this quantity on a smooth manifold converges to the volume of the unit ball in $\RR^n$ and, as $r$ tends to infinity, it converges to a nonnegative number $\V_M$.  If $\V_M>0$, then $M$ is said to have Euclidean volume growth and, by \cite{ChC1}, any tangent cone at infinity  is a metric cone.\footnote{A metric cone $C(X)$ with cross-section $X$ is a warped product metric $dr^2+r^2\,d^2_X$ on the space $(0,\infty)\times X$.  For tangent cones at infinity of manifolds with $\Ric\geq 0$ and $\V_M>0$, by \cite{ChC1} any cross-secton is a length space with diameter $\leq \pi$.}

An important well-known question is whether the cross-section of the tangent cone at infinity of a Ricci-flat manifold with $\V_M>0$ depends on the convergent sequence of blow-downs or is unique and independent of the sequence.  Our main theorem is the following:

\begin{Thm}[Uniqueness at $\infty$]    \label{t:main}
Let $M^n$ be a Ricci-flat manifold with Euclidean volume growth.  If one tangent cone at infinity has a smooth cross-section, then the tangent cone 
at infinity is unique.\footnote{In fact, we prove that the scale invariant distance to the tangent cone converges to zero like $(\log r)^{-\beta}$ for some $\beta > 0$, where $r$ is the distance to a fixed point.} 
\end{Thm}

In fact, we prove an effective version of uniqueness that is considerably stronger.   Theorem \ref{t:main} settles in the affirmative a very strong form of conjecture 1.12 in \cite{CN1}.

\vskip2mm
The results of this paper were announced in \cite{C2} and again in \cite{CM3}.

\vskip2mm
Theorem \ref{t:main} describes the asymptotic structure of Einstein manifolds with Euclidean volume growth and vanishing Ricci curvature. 
These arise in a number of different fields, including string theory, general relativity, and complex and algebraic geometry, amongst others, and there   is a extensive literature of examples; see, e.g., 
\cite{BGS}, \cite{DS}, \cite{K1}, \cite{K2}, \cite{MS1}, \cite{MS2}, \cite{MSY1}, \cite{MSY2}, \cite{TY1} and \cite{TY2}.
Most  examples fall into several different classes, including ALE spaces (like the
 Eguchi-Hanson metric and, more generally, non-collapsing gravitational instantons, etc.), K\"ahler-Einstein metrics constructed by blowing up divisors, or cones over Sasaki-Einstein manifolds.

\vskip2mm
Our arguments will also show that local tangent cones of limits of noncollapsing Einstein metrics are unique:

\begin{Thm}[Local uniqueness]
Let $(M_i,x_i)$ be a sequence of pointed $n$-dimensional Einstein metrics with uniformly bounded Einstein constants and $\Vol (B_1(x_i))\geq v>0$.  

If $(M_{\infty}, x_{\infty})$ is a Gromov-Hausdorff limit of $(M_i,x_i)$ and one tangent cone at $y\in M_{\infty}$ has a smooth cross-section, then the tangent cone at $y$ is unique.
\end{Thm}

Similar to the case of tangent cones at infinity, the above statement follows from a stronger effective version of uniqueness of local tangent cones.

\vskip2mm
It is well-known that uniqueness may fail without the two-sided bound on the Ricci curvature.  Namely, there exist a large number of examples of manifolds with nonnegative Ricci curvature and Euclidean volume growth and nonunique tangent cones at infinity; see \cite{P2}, \cite{ChC1}, \cite{CN2}.  In fact, by \cite{CN2}, it is known that any smooth family of metrics on a fixed closed manifold can occur as cross-sections of tangent cones at infinity of a single manifold with nonnegative Ricci curvature and Euclidean volume growth provided the following two necessary assumptions are satisfied for any element in the family:
\begin{enumerate}
\item The Ricci curvature is $\geq$ than that of the round unit $(n-1)$-dimensional sphere.\label{e:exun1}\footnote{Strictly speaking, for the construction in \cite{CN2}, one must assume strict inequality for the Ricci curvature.}
\item The volume is equal to a fixed constant. \label{e:exun2}
\end{enumerate}
Since the space of cross-sections of tangent cones at infinity of a given manifold with nonnegative Ricci curvature and Euclidean volume growth is connected and closed under the Gromov-Hausdorff topology, it follows that if a smooth family of closed manifolds occurs as cross-sections, then so does any metric space in the closure.  

  \vskip2mm
   There is a rich history of uniqueness results for geometric problems and equations. Ê In perhaps its simplest form, the issue of uniqueness or not comes up already in a 1904 paper entitled ``On a continuous curve without tangents constructible from elementary geometry" by the Swedish mathematician Helge von Koch.  In that paper, Koch described what is now known as the Koch curve or Koch snowflake.  It is one of the earliest fractal curves to be described and, as suggested by the title, shows that there are continuous curves that do not have a tangent in any point.  On the other hand, when a set or a curve has a well-defined tangent or well-defined blow-up at every point, then much regularity is known to follow.  Tangents at every point, or uniqueness of blow-ups, is a `hard' analytical fact that most often is connected with a PDE,  as opposed to say Rademacher's theorem, where tangents are shown to exist almost everywhere for any Lipschitz functions. 
  
 Uniqueness is a key question for the regularity of Geometric PDE's; for instance, as explained in \cite{W}: ``Whether nonuniqueness of tangent cones ever happens remains perhaps the most fundamental open question about singularities of minimal varieties''.  Two of the most prominent early works on uniqueness of tangent cones are Leon Simon's hugely influential paper \cite{S1} from 1983, where he proves uniqueness for tangent cones of minimal surfaces with smooth cross-section. The other is Allard-Almgren's 1981 \cite{AA} paper where uniqueness of tangent cones with smooth cross-section is proven under an additional integrability assumption on the cross-section; see also \cite{S2} and \cite{H} for more references about uniqueness. ÊEarlier work on uniqueness for Ricci-flat metrics includes Cheeger-Tian's 1994 paper \cite{ChT}, where uniqueness is shown if all tangent cones have smooth cross-sections and all are integrable.\footnote{In addition to integrability of all cross-sections,  \cite{ChT} assumed that the sectional curvatures  decay at least quadratically at infinity.   This can be seen (by \cite{C1}) to be equivalent to that all tangent cones at infinity have smooth cross-sections.}
 
 In each of these geometric problems, existence of tangent cones comes from  monotonicity, while the approaches to uniqueness rely on showing that the monotone quantity approaches its limit at a definite rate.  However, estimating the rate of convergence seems to require either integrability and/or a great deal of regularity (such as analyticity).   For instance, for minimal surfaces or harmonic maps,
  the classical monotone quantities are highly regular and are well-suited to this type of argument.    This is not at all the case in the current setting where the Bishop-Gromov is of very low regularity and ill suited: the distance function is Lipschitz,  but is not even $C^1$, let alone analytic.  This is a major point (cf. page 496 of \cite{ChT}).  In contrast, the functional $A$ (that we describe below) is defined on the level sets of an analytic function (the Green's function) and does depend analytically and, furthermore, its derivative has the right properties.  In a sense, the scale invariant volume is already a regularization of the quantity that, if one could, one would most of all like to work with.  Namely, one would like to work directly with the scale invariant Gromov-Hausdorff distance between
  the manifold and the cone that best approximates it on the given scale and try to prove directly some kind of decay (in the scale)  for this quantity.  However, not only is it not clear that it is monotone, but as a purely metric quantity it is even less regular than the scale invariant volume.

  \vskip4mm
Throughout, $C$ will denote a constant which will be allowed to change from line to line.  When the dependence is important, we will be more explicit.    $M^n$ will always be an open $n$-dimensional Ricci-flat manifold with Euclidean volume growth where $n \geq 3$.  
Moreover, $d_{GH}(X,Y)$ will denote the Gromov-Hausdorff distance between metric spaces $X$ and $Y$.

\vskip4mm
\subsection{Proving uniqueness}
  Next we will try to explain  the key points in the proof of uniqueness; a much more detailed discussion can be found in Section \ref{s:s1}.
  
  Let $p\in M$ be a fixed point in a Ricci flat manifold with Euclidean volume growth.  We would like to show that the tangent cone at infinity is unique; that is, does not depend on the sequence of blow-downs.  To show this,
   let $\Theta_r$ be the scale invariant Gromov-Hausdorff distance between the annulus $B_{4r}(p)\setminus B_{r}(p)$ and the corresponding annulus centered at the vertex of the cone that best approximates the annulus.  (By scale invariant distance, we mean the distance 
   between the annuli after the metrics are rescaled so that the annuli have unit size; see \eqr{e:scaleinvd}.)  The first key point is to find a positive quantity $A=A(r)$ that is a function of the distance to $p$, is monotone $A\downarrow$ and so for some positive constant $C$ 
 \begin{align}  \label{e:crucial}
 -A' (r)\geq C\,\frac{\Theta_r^2}{r} \, .
\end{align}
(The quantity $A$ with this property was found in \cite{C2}.  Perelman's monotone $W$ functional is also potentially a candidate, but it comes from integrating over the entire space which introduces so many other serious difficulties that it cannot be used.)    In fact, we shall use that for $Q$ roughly equal to $-r\,A'(r)$, $Q$ is monotone nonincreasing and 
\begin{align}  \label{e:crucial2}
 [Q(r/2)-Q(8r)] \geq C\,\Theta_r^2 \, .
\end{align}
We claim that uniqueness of tangent cones is implied by showing that $A$ converges to its limit at infinity at a sufficiently fast rate or, equivalently, that $Q$ decays sufficiently fast to zero.  Namely, by the triangle inequality, uniqueness is implied by proving that
 \begin{align}
 \sum_k \Theta_{2^k}<\infty \, .
\end{align}
This, in turn, is implied by the Cauchy-Schwarz inequality by showing that for some $\epsilon>0$
\begin{align} \label{e:summa}
 \sum_k\Theta_{2^k}^2\,k^{1+\epsilon}<\infty \, ,
\end{align}
as
\begin{align} 
 \sum_k k^{-1-\epsilon}<\infty \, .
\end{align}
Equation \eqr{e:summa} follows, by \eqr{e:crucial2}, from showing that
\begin{align}  
 \sum [Q(2^{k-1})-Q(2^{k+3})]\,k^{1+\epsilon}<\infty \, .
\end{align}
This is implied by proving that for a slightly larger $\epsilon$
\begin{align}  
 Q(r)\leq  \frac{C}{(\log r)^{1+\epsilon}} \, .
\end{align}
All the work in this paper is then to establish  this crucial decay for $Q$.  This decay follows easily
from showing that for some $\alpha < 1 $
 \begin{align}			\label{e:o3}
	    Q(2\,r)^{2-\alpha} &  \leq   C \, \left( Q(r/2) - Q(2\,r) \right) 
 \, .
\end{align}
The proof of this comes from an infinite dimensional Lojasiewicz inequality that essentially gives
\begin{align}			\label{e:o4}
	    |A(r)-A(\infty)|^{2-\alpha} &  \leq   C \, |\nabla A|^2=-C \, r\, A'
 \, .
\end{align}
(Here the middle equation can be ignored as we won't    explain the meaning of $\nabla A$ until later.)
The left-hand side of \eqr{e:o3} is easily seen (using that $Q$ is monotone) 
to be bounded from above by the left-hand side of \eqr{e:o4}.  To get that the right-hand side of \eqr{e:o4} is bounded from above by the right-hand side of \eqr{e:o3} is more subtle and uses that the quantity $Q(r)$ is defined slightly differently.

The proof of uniqueness has three parts.  The first is to find the right quantities and set up the general scheme described above.  The second will be to find a way to actually implement this general scheme.  The third will be to prove the 
 infinite dimensional Lojasiewicz inequality for a functional $\cF$ that approximates $A$ to first order.  
 $\cF$ will actually be defined on the space of metrics and weights.  To explain
 how $\cF$ is chosen, 
recall that a  
 Lojasiewicz inequality  describes analytic functions in a  neighborhood of a critical point.  The inequality  asserts that the difference in values of such a function at a critical point versus a nearby point is bounded in terms of the norm of the gradient.  In particular, any other nearby critical point must have the same value.  
  In our case, the analytic function will be a linear combination of a weighted Einstein-Hilbert functional on the level sets plus the $A$ functional.  The Einstein-Hilbert functional enters into this picture since in a Ricci-flat cone the cross-section is a Einstein manifold and, thus, a critical point for the Einstein-Hilbert functional.

Finally, note that although $Q\geq 0$ and $Q\downarrow$, the rate of decay on $Q$ implies only that 
\begin{align}
\Theta_{2^k}\leq \left(\sum_{j\geq k} \Theta_{2^j}^2\right)^{\frac{1}{2}}
\end{align}
decays like $k^{-\frac{1}{2}-\epsilon}$ which in itself is of course not summable.   Uniqueness comes from the decay of $Q$ together with that 
\begin{align}  \label{e:crucial2a}
\Theta_r^2\leq C\, [Q(r/2)-Q(8r)] \, ,
\end{align}
which gives that 
\begin{align}
	\sum_{j\geq k} \Theta_{2^j} \leq C \, k^{- \bar{\beta}}
\end{align}
for a power $\bar{\beta} > 0$.

  \subsection{Effective uniqueness}
  
  In this subsection, we will describe how our main uniqueness will follow from a stronger effective version.
  
  Let $M^n$ be a Ricci-flat $n$-manifold and $N$ a smooth closed Einstein $(n-1)$-manifold with $\Ric=(n-2)$.

\begin{Thm}[Effective uniqueness]	\label{t:effective}
There exist $\epsilon$, $\delta$, $\beta > 0$ and $C>1$ such that if
$A(r_1/C ) - A(C r_2) < \delta$ for some $0<r_1 < r_2$ and 
every $r \in [r_1/C , C r_1]$ satisfies
\begin{align}
d_{GH}(   B_{  2r}(x) \setminus B_r(x) , 
 B_{  2r}(v) \setminus B_{r}(v) ) < \epsilon\, r\, ,
\end{align}
where $x\in M$ and $v$ is the vertex of the cone $C(N)$, then:
\begin{enumerate}
\item[(E1)] Every $r \in [r_1 , r_2]$ satisfies
\begin{align}
d_{GH}(   B_{  2r}(x) \setminus B_{r}(x) , 
 B_{  2r}(v) \setminus B_{r}(v) ) < 4 \, \epsilon\, r \, .
\end{align}
\item[(E2)]
There exists a cone $C(N_0)$ with vertex $\tilde v$ such that for $r$ between $r_1$ and $r_2$  
\begin{align}	\label{e:effv1}
d_{GH}(   B_{  4r}(x) \setminus B_{r}(x) , 
 B_{  4r}( \tilde v) \setminus B_{r}(\tilde v) ) < C \, r\, \left(\log  \frac{r}{r_1} \right)^{-\beta}    \, .
\end{align}
\end{enumerate}
\end{Thm}

Note that the cone $C(N_0)$ in this theorem is independent of $r$.   Moreover, the Gromov-Hausdorff distance could be replaced by the $C^k$ norm in \eqr{e:effv1} by appealing to \cite{C1}.
The key in the above theorem is that the constants do not depend on $r_1$ and $r_2$.  As a consequence, we get the uniqueness theorem stated above.

\vskip1mm
Remarks:
\begin{itemize}
\item It seems very likely that, by arguing similarly, one could also replace the right-hand side of \eqr{e:effv1} by
$  C\, r \, \left[ A(r_1)-A(r_2)\right]^{\beta}$.
\item There is also a local version of this that we will not state here.
\end{itemize}

\subsection{Key technical difficulties for the Lojasiewicz-Simon inequality}

The classical Lojasiewicz-Simon inequality
  is proven by using Lyapunov-Schmidt reduction to reduce it to a finite dimensional Lojasiewicz inequality on the kernel of the second variation operator.  It is critical that the kernel is
   finite dimensional.  In \cite{S1},  the finite dimensionality came from the functional being strictly convex in the first derivative (which was the highest order), so that there are only finitely many eigenvalues (counting multiplicity) below any fixed level.

There are two key difficulties for proving a Lojasiewicz-Simon inequality for the $\cF$ functional:
\begin{enumerate}
\item There is an infinite dimensional kernel for the second variation operator.  \label{e:item1} 
\item The second variation operator   has infinitely many positive and negative eigenvalues.  \label{e:item2} 
\end{enumerate}
The reason for \eqr{e:item1} is that the infinite dimensional gauge group of diffeomorphisms preserves the functional.  \eqr{e:item2} is similar to the situation for the Einstein-Hilbert functional, where the highest order part of the second variation operator has opposite signs depending on whether the variation is conformal or orthogonal to the conformal variations.  \eqr{e:item1}  is far more serious.   

Geometric functionals are invariant under changes of coordinates, so \eqr{e:item1} could potentially arise in any geometric problem, including the original ones considered in \cite{S1}, such as uniqueness for minimal surfaces. 
This is overcome in \cite{S1} by working in canonical coordinates, such as writing the surfaces as normal graphs.  Similarly,   in \cite{Ya}, the author makes a canonical choice of frames to ``gauge away''  \eqr{e:item1}
for the Yang-Mills functional and then directly apply \cite{S1}.  In our setting, the action of the diffeomorphism group is more complicated and even \eqr{e:item2} already makes it impossible to appeal directly to \cite{S1}.

We will deal with \eqr{e:item1}  by using the Ebin-Palais slice theorem to mod out by the diffeomorphism group.{\footnote{The diffeomorphism group also created difficulties in \cite{ChT}, where they use a different version of the slice theorem.}}
This will allow us to restrict to variations that are transverse to the action of the group.  We will then analyze the second variation operator separately, depending on whether the variation is in the conformal direction (up to a diffeomorphism) or it is orthogonal to both the conformal variations and to the action of the group. We will show that, if we write the operator in block form, then the off-diagonal blocks vanish and the kernel is finite dimensional in each diagonal block.  This will be enough to carry through the Lyapunov-Schmidt reduction
and prove the Lojasiewicz-Simon inequality.

\subsection{Normalizations}

Our normalization is that the Ricci curvature of the
$(n-1)$-dimensional unit sphere $\SS^{n-1}$ is $(n-2)$ and the
  scalar curvature   is $(n-1) \, (n-2)$.

\section{The proof of uniqueness}  \label{s:s1}

 As mentioned in the introduction, the starting point for uniqueness is a monotonicity formula from \cite{C2}, where the monotone quantity $A(r)$ is non-increasing in $r$, is constant  on   cones, and where the derivative   $A'(r)$   measures   distance  to being a cone on a given scale.  We will  show that $A(r)$ goes to its limit $A(\infty)$ fast enough to ensure  uniqueness of the tangent cone.
The key   is to show that
\begin{enumerate}
\item[($\star$)] $A'(r)$ controls   $A(r) - A(\infty)$.
\end{enumerate}
   Iterating ($\star$) will show that $A'(r)$, and thus the distance to being a cone, converges to zero 
at a rate    that implies uniqueness.

In order to prove  ($\star$), we will need to introduce an
auxiliary  functional $\cF$.  To explain this, recall that 
the Lojasiewicz inequality, \cite{L}, for an analytic function $f$ on $\RR^n$ with a critical point $x$ gives some $\alpha < 1$ so that
\begin{align}
	\left| f (x) - f(y) \right|^{2-\alpha} \leq \left| \nabla f  (y) \right|^2 
\end{align}
for all $y$ close to $x$.  Leon Simon proved an infinite dimensional version of this for certain analytic functionals on Banach spaces
in \cite{S1}.   We will construct an analytic functional $\cF$ that approximates $A$ to first order  and satisfies a 
Lojasiewicz-Simon inequality (these properties are  (1)--(5)  in
 subsection \ref{ss:5key}).  Using $\cF$, we can prove ($\star$).

In this section, we will prove the uniqueness of the tangent cones assuming properties  (1)--(5).  The rest of the paper will be devoted to proving these properties.

\subsection{Monotonicity}

We will next define the monotone quantity $A(r)$.
Let $G$ be a Green's 
function\footnote{Our Green's functions will be normalized so that on Euclidean space of dimension $n\geq 3$ the Green's function is $r^{2-n}$.}  on $M$ with a pole at a fixed point $x \in M$ and define
\begin{align}
	b = G^{ \frac{1}{2-n}} \, . 
\end{align}
With this normalization, Stokes' theorem implies that 
\begin{align}	\label{e:stok}
	  r^{1-n} \, \int_{b=r} |\nabla b| =  \Vol (\partial B_1 (0)) \, .
\end{align}
Following \cite{C2}, define a scale-invariant quantity $A(r)$ by
\begin{align}
	A(r) = r^{1-n} \, \int_{b=r} |\nabla b|^3 \, .
\end{align}
Since $M$ is Ricci-flat the third monotonicity formula of \cite{C2} gives that
\begin{align}
	  A'(r) &= - \frac{1}{2}  \, r^{n-3} \, \int_{ r \leq b}  b^{2-2n}\,     \left| \Hess_{b^2} - \frac{\Delta b^2}{n} \, g
	   \right|^2 \, . 
\end{align}
 In particular, $A$ is monotone non-increasing and, thus, has a
  limit\footnote{In fact, an easy calculation
   shows (see \cite{C1}) that $A_{\infty}=b_{\infty}^2\, \Vol (\partial B_1(0))$; where $b_{\infty}$ is defined below.} 
 \begin{align}
 	A_{\infty} = \lim_{r\to \infty} \, A(r) \, .
\end{align}
As a consequence, we have that
\begin{align}	\label{e:intde}
	A(R) - A_{\infty}  = \frac{1}{2} \, \int_{R}^{\infty}  r^{n-3} \,
	\int_{ r \leq b}  b^{2-2n}\,     \left| \Hess_{b^2} - \frac{\Delta b^2}{n} \, g
	   \right|^2
	 \, dr \, .
\end{align}

\subsection{A brief introduction to the $\cF$ functional}  
We will next briefly explain what the functional $\cF$ is that will appear in our Lojasiewicz-Simon inequality.   This discussion can safely be ignored as we will later return to the precise definition, including the weighted space that $\cF$ is defined on.  At any rate,  when restricted to the level set $b=r$ the functional $\cF$ will be given by
\begin{align}
	\cF(r)=\cF &= \frac{1}{2-n} \, \left( A -  \frac{r^{3-n}}{n-2} \, \int_{b=r}  R_{b=r}\,  |\nabla b|    \right)\\
	&=\frac{r^{1-n}}{2-n} \, \int_{b=r}\left(|\nabla b|^2-\frac{r^2\,R_{b=r}}{n-2}\right)\, |\nabla b|\, .\notag
\end{align}
Here $R_{b=r}$ is the intrinsic scalar curvature of the level set $b=r$.  The idea behind this functional is that $\cF$ defined this way is a weighted analog of the classical Einstein-Hilbert functional.    In particular, when $\cF$ is restricted to an appropriate weighted space, then the critical points will precisely be weighted Einstein metrics.  

It may be helpful to illustrate this with an example.  Suppose that $M$ is $n$-dimensional Euclidean space $\RR^n$ so that $b$ is the distance function $|x|$. Since the scalar curvature of the sphere of radius $r$ is $(n-1) (n-2) r^{-2}$, we get
\begin{align}
	\cR(r) = \frac{r^{1-n}}{2-n} \, \int_{|x|=r} \left( 1 - \frac{ r^2 \, (n-1) (n-2) r^{-2} }{n-2} \right) = 
	r^{1-n} \, \int_{|x|=r} 1 = A(r) \, .
\end{align}
This is a special case of that $\cR$ and $A$ agree on cones with a constant weight (see (1) below in the subsection after the next one).

\subsection{Asymptotic convergence}

By \cite{ChC1},  every tangent cone at infinity of $M$ is a metric cone.  Below, $C(N)$ will always be a fixed cone with vertex $v$ over a smooth $(n-1)$-dimensional Einstein metric $g_0$ on the cross-section $N$ with 
\begin{align}
	\Ric_{g_0} = (n-2) \, g_0 \, .
\end{align}
 Moreover, 
 $\delta=\delta (N)> 0$ will be a fixed small constant and   we will work on scales $R> 0$ so that
\begin{align}  \label{e:conedel}
	 d_{GH}(B_{2r}(x) \setminus B_{r}(x) , B_{2r}(v) \setminus B_{r} (v))
	 < \delta\,  r   {\text{ for all }}  r \in \left[  \frac{R}{4} , 2R   \right] \, ,
\end{align}
where $d_{GH}$ is the Gromov-Hausdorff distance.
In particular,  by \cite{C1},  the annulus $B_{2R}(x)\setminus B_{\frac{R}{2}}(x)$ in $M$ is $C^k$ close to one in the cone $C(N)$.

We claim that as long as annuli in $M$ are close to annuli in the cone (in the sense explained above around \eqr{e:conedel}), then 
\begin{align}	
|\nabla b|\text{ is close to }b_{\infty}\, .\label{e:closetob}
\end{align}
Here the positive constant $b_{\infty}$ is defined by
\begin{align}
b_{\infty}=\left(\frac{\V_M}{\Vol (B_1(0))}\right)^{\frac{1}{n-2}}\, ,
\end{align}
where $\V_M > 0$ is the asymptotic volume ratio
\begin{align}
\V_M=\lim_{r\to \infty}r^{-n} \,  \Vol (B_r(x))\, .
\end{align}
To see \eqr{e:closetob}, note that by page 1374 of \cite{CM2} for $\epsilon> 0$ fixed, there exists $r_0> 0$ so that for $r\geq r_0$
\begin{align}
\sup_{\partial B_r(x)}\left|\frac{b}{r}-\left(\frac{\V_M}{\Vol (B_1(0))}\right)^{\frac{1}{n-2}}\right|&<\epsilon\, ,\\
\int_{B_r(x)}\left||\nabla b|^2-\left(\frac{\V_M}{\Vol (B_1(0))}\right)^{\frac{2}{n-2}}\right|^2&<\epsilon\, \Vol (B_r(x))\, .
\end{align}
Since the annulus in $M$ is $C^k$ close to one in the cone $C(N)$ (by \cite{C1}) 
and $b$ satisfies an elliptic equation, we get
estimates for higher derivatives of $b$.  Namely, the integral bound on $\left||\nabla b|^2-\left(\frac{\V_M}{\Vol (B_1(0))}\right)^{\frac{2}{n-2}}\right|$
gives the following pointwise bound (for a slightly larger $\epsilon$)
\begin{align}
\sup_{B_{2R}(x)\setminus B_{\frac{R}{2}}(x)}\left||\nabla b|^2-\left(\frac{\V_M}{\Vol (B_1(0))}\right)^{\frac{2}{n-2}}\right|^2&<\epsilon \, .
\end{align}

\subsection{The functional $\cF$ and the Lojasiewicz-Simon inequality}	\label{ss:5key}

We will next bring in the auxiliary  functional $\cF$ and list its five key properties.

Given $R > 0$, we let $g_R$ denote the induced metric on the level set $\{ b = R \}$ in $M$.  It follows from the previous subsection that
if we are in an annulus that is close to one in $C(N)$, then $\{ b = R \}$ is diffeomorphic to $N$.  Moreover, the metric
$R^{-2} \, g_R$ is close to the metric $b_{\infty}^{-2} \, g_0$ and, in fact, \eqr{e:stok} implies that
\begin{align}
	\int_{b=R} |\nabla b| \, d\mu_{R^{-2} \, g_R} = R^{1-n} \, \int_{b=R} |\nabla b|  = \Vol (\partial B_1 (0)) \, .
\end{align}
Define $\cA$ to be the set of $C^{2,\beta}$ metrics $g$ and positive $C^{2,\beta}$ functions $w$ on $N$. 
 Let $\cA_1$ be 
\begin{align}
	\cA_1 = \left\{ (g,w) \in \cA \, | \, \int_N w \, d\mu_g = \Vol (\partial B_1 (0)) \right\} \, .
\end{align}
The set $\cA_1$ includes $(R^{-2} \, g_R , |\nabla b|)$ as well as $(b_{\infty}^{-2} \, g_0 , b_{\infty})$.

\vskip2mm
We will construct a functional $\cF: \cA_1 \to \RR$ that satisfies:
\begin{enumerate}
\item $\cF ( b_{\infty}^{-2} \, g_0 , b_{\infty}) = A_{\infty}$.
\item $(b_{\infty}^{-2} \, g_0 , b_{\infty})$ is a critical point for $\cF$ on $\cA_1$.
\item $\cF$ satisfies the Lojasiewicz-Simon inequality for some $\alpha < 1$
\begin{align}
	\left| \cF (g, w) - \cF (b_{\infty}^{-2} \, g_0 , b_{\infty}) \right|^{2-\alpha} \leq \left| \nabla_1 \cF \right|^2 (g, w)\, , 
\end{align}
where $\nabla_1 \cF$ is the restriction of $\nabla \cF$ to $\cA_1$ and $(g,w)$ is near  $(b_{\infty}^{-2} \, g_0 , b_{\infty})$.
\item We have 
\begin{align}
	\left| \nabla_1 \cF (   R^{-2} \, g_R, |\nabla b|) \right|^2 \leq C \, \int_{\frac{R}{2} \leq b \leq \frac{3\,R}{2}}  b^{-n} \,   \left| \Hess_{b^2} - \frac{\Delta b^2}{n} \, g
	   \right|^2 \, .
\end{align}
\item   We  have 
\begin{align}
A(R) \leq \cF ( R^{-2} \, g_R , |\nabla b| )  +  C \,
 \int_{\frac{R}{2} \leq b \leq \frac{3\,R}{2}}  b^{-n}\,  \left| \Hess_{b^2} - \frac{\Delta b^2}{n} \, g
	   \right|^2    
\, .
\end{align}
\end{enumerate}

\vskip2mm
Roughly speaking, (1) and (2) show that $\cF$ agrees with $A$ to first order at infinity, while (4) and (5) show that
they are equivalent to first order on $( R^{-2} \, g_R , |\nabla b| ) $.  At first, this may appear surprising since $\cR$ will contain the scalar curvature and, thus, depends on more derivatives of the metric.  However, we will see that the trace-free Hessian satisfies an elliptic equation and, thus, elliptic  estimates  will allow us to bound these higher derivatives by lower order ones (see Theorem \ref{t:meanval} below).

  We will construct $\cF$ to satisfy (1) and (2) in Section \ref{s:functs}.  Properties (4) and (5) are proven in Section \ref{s:4and5}.  The remainder of the paper proves  the Lojasiewicz-Simon inequality (3)  for $\cF$.

\begin{Rem}
Roughly speaking, one can think of (4) and (5) as effective forms of (2) and (1), respectively.  Namely, when the manifold is conical, then (4) and (5) imply (1) and (2), but with inequalities instead of equalities.  The precise dependence in the error terms will be critical for our arguments.
\end{Rem}
 
  \subsection{Decay}
  
  We will show next that (1)--(5) above implies that the tangent cone at infinity is unique.  We will first show decay of the following natural monotone non-increasing scale-invariant integral 
 \begin{align}
 	Q(r) = \int_{r \leq b} b^{-n} \,   \left| \Hess_{b^2} - \frac{\Delta b^2}{n} \, g
	   \right|^2  
 \end{align}
   that roughly measures $- r \, A'(r)$.  One important reason why we work with $Q$ instead of $r\, A'$ is that $Q(r)$ is obviously monotone.
 
 \vskip2mm
 Precisely, we will show that (1)--(5) implies the following crucial decay estimate:

 \begin{Pro}  \label{p:decay}
Set $\beta = \frac{1}{1-\alpha} - 1> 0$.  There exists  $C$ so that 
  if every $R \in (r , s)$ satisfies \eqr{e:conedel}, then
     \begin{align}	 \label{e:fu}
  	 Q(s) \leq \frac{C}{|\log (s/r)|^{\beta +1}}  \, .
  \end{align}
 \end{Pro}
 
 \subsection{Proving decay} 
 
 As described in the overview, the key for
 proving the decay in  Proposition \ref{p:decay} is to establish  the   inequality \eqr{e:o3} bounding $Q(2r)$ in terms of the decay of $Q$ from $r/2$ to $2r$.  This will be done in a series of lemmas culminating in Corollary
 \ref{c:o3}.
  
  \begin{Lem}	\label{l:fp}
  If $R$ satisfies \eqr{e:conedel}, then 
  \begin{align}	 	\label{e:mydealq}
	  \left(   \int_{R}^{\infty}  r^{n-3} \,
	\int_{ r \leq b}  b^{2-2n}\,     \left| \Hess_{b^2} - \frac{\Delta b^2}{n} \, g
	   \right|^2
	 \, dr  \right)^{2-\alpha}   &\leq   C \, \int_{\frac{R}{2} \leq b \leq \frac{3\,R}{2}}  b^{-n}\,  \left| \Hess_{b^2} - \frac{\Delta b^2}{n} \, g
	   \right|^2  \, .
\end{align}
  \end{Lem}
  
  \begin{proof}
  Using \eqr{e:intde}, then  (1) and then (5)  gives
\begin{align}	 	\label{e:mydeal1}
	  &   \frac{1}{2} \, \int_{R}^{\infty}  r^{n-3} \,
	\int_{ r \leq b}  b^{2-2n}\,     \left| \Hess_{b^2} - \frac{\Delta b^2}{n} \, g
	   \right|^2
	 \, dr      =   A(R) - A_{\infty}   =   A(R) -   \cF (  b_{\infty}^{-2} \, g_0 , b_{\infty})   \notag \\
	 &\qquad  \leq
	  \cF (  R^{-2} g_R , |\nabla b|) - \cF ( b_{\infty}^{-2} \, g_0 , b_{\infty}) +  C \, \int_{\frac{R}{2} \leq b \leq \frac{3\,R}{2}}  b^{-n}\,  \left| \Hess_{b^2} - \frac{\Delta b^2}{n} \, g
	   \right|^2    \, .
\end{align}
On the other hand, 
(3) and (4) give that
\begin{align}	\label{e:3and4}
	\left|  \cF ( R^{-2} g_R , |\nabla b|) - \cF ( b_{\infty}^{-2} \, g_0 , b_{\infty}) \right|^{2-\alpha} &\leq 
	\left| \nabla_1 \cF ( R^{-2} g_R ,|\nabla b|) \right|^2  \notag 
	  \\
	& \leq   C \, \int_{\frac{R}{2} \leq b \leq \frac{3\,R}{2}}  b^{-n}\,  \left| \Hess_{b^2} - \frac{\Delta b^2}{n} \, g
	   \right|^2  
\end{align}
Raising \eqr{e:mydeal1} to the power $2-\alpha$, using the convexity of $t \to t^p$ for $p \geq 1$ so that
\begin{align}
	(a + b)^{p} \leq 2^{p-1} \, (a^{p} + b^{p}) {\text{ for $a,b \geq 0$  and $p \geq 1$ }}
\end{align}
with $p= 2-\alpha$,  and then using \eqr{e:3and4} gives
\begin{align}	 	\label{e:mydealq}
	  \left(   \int_{R}^{\infty}  r^{n-3} \,
	\int_{ r \leq b}  b^{2-2n}\,     \left| \Hess_{b^2} - \frac{\Delta b^2}{n} \, g
	   \right|^2
	 \, dr  \right)^{2-\alpha}   &\leq   C \, \int_{\frac{R}{2} \leq b \leq \frac{3\,R}{2}}  b^{-n}\,  \left| \Hess_{b^2} - \frac{\Delta b^2}{n} \, g
	   \right|^2   \notag \\
	   &+ C\, \left(     \int_{\frac{R}{2} \leq b \leq \frac{3\,R}{2}}  b^{-n}\,  \left| \Hess_{b^2} - \frac{\Delta b^2}{n} \, g
	   \right|^2   
	   \right)^{2-\alpha} \, .
\end{align}
 Since $2-\alpha > 1$ and we always work on annuli where $ \int_{\frac{R}{2} \leq b \leq \frac{3\,R}{2}}  b^{-n}\,  \left| \Hess_{b^2} - \frac{\Delta b^2}{n} \, g
	   \right|^2   $ is bounded, we conclude that
\begin{align}	 	
	  \left(   \int_{R}^{\infty}  r^{n-3} \,
	\int_{ r \leq b}  b^{2-2n}\,     \left| \Hess_{b^2} - \frac{\Delta b^2}{n} \, g
	   \right|^2
	 \, dr  \right)^{2-\alpha}   &\leq   C \, \int_{\frac{R}{2} \leq b \leq \frac{3\,R}{2}}  b^{-n}\,  \left| \Hess_{b^2} - \frac{\Delta b^2}{n} \, g
	   \right|^2  \, .
\end{align}
\end{proof}

 \begin{Lem}	\label{l:sums}
Given $R> 0$, we have
  \begin{align}
 	 \int_{R}^{\infty}  r^{n-3} \,
	\int_{ r \leq b}  b^{2-2n}\,    \left| \Hess_{b^2} - \frac{\Delta b^2}{n} \, g
	   \right|^2 \, dr &\geq   4^{2-n} \, Q(2\,R) \, .
 \end{align}
 \end{Lem}
 
 \begin{proof}
 Within this proof, set $f =  \left| \Hess_{b^2} - \frac{\Delta b^2}{n} \, g
	   \right|^2$ to simplify notation.  We have
 \begin{align}
 	 \int_{R}^{\infty}  r^{n-3} \,
	\int_{ r \leq b}  b^{2-2n}\,     f \, dr &= \sum_{j=0}^{\infty} \int_{2^j\,R}^{2^{j+1}\,R} \, 
	r^{n-3} \,
	\int_{ r \leq b}  b^{2-2n}\,     f 
	 \, dr \notag \\
	 &\geq \sum_{j=0}^{\infty} \int_{2^j\,R}^{2^{j+1}\,R} \, 
	\left( 2^j \right)^{n-3} \,
	\int_{ 2^{j+1}\, R \leq b \leq 2^{j+2}\, R} b^{2-2n}\,     f 
	 \, dr \\
	 &=  \sum_{j=0}^{\infty}  	  \left( 2^j\, R \right)^{n-2} \,
	\int_{ 2^{j+1}\, R \leq b \leq 2^{j+2}\, R}  b^{2-2n}    f 
	     \notag  \, .
 \end{align}
 On the interval $2^{j+1} \, R\leq b \leq 2^{j+2}\, R$, we have that
 \begin{align}
 	\left( 2^j \, R\right)^{n-2} \,
	   b^{2-2n}  =  b^{-n} \left( \frac{2^j\, R}{b} \right)^{n-2} \geq 4^{2-n} \, b^{-n} \, .
 \end{align}
 We conclude that
 \begin{align}
 	 \int_{R}^{\infty}  r^{n-3} \,
	\int_{ r \leq b}  b^{2-2n}\,     f \, dr &\geq  4^{2-n} \, 
	\sum_{j=0}^{\infty}  	 	\int_{ 2^{j+1}\, R \leq b \leq 2^{j+2}\, R}  b^{ - n}    f  = 4^{2-n} \, Q(2\, R) \, .
 \end{align}

 \end{proof}
 
Combining Lemmas \ref{l:fp} and  \ref{l:sums}     gives the inequality \eqr{e:o3}:

\begin{Cor}		\label{c:o3}
 If $r$ satisfies \eqr{e:conedel}, then 
 \begin{align}		\label{e:mydeal2}
	    Q(2\,r)^{2-\alpha}    \leq     C \, \left( Q(r/2) - Q(2\,r) \right) 
 \, .
\end{align}
\end{Cor}

\begin{proof}
Combining Lemmas \ref{l:fp} and  \ref{l:sums}   gives
\begin{align}		 
	    Q(2\,r)^{2-\alpha} &  \leq   C \, \int_{\frac{r}{2} \leq b \leq 2r}  b^{-n}\,  \left| \Hess_{b^2} - \frac{\Delta b^2}{n} \, g
	   \right|^2 = C \, \left( Q(r/2) - Q(2\,r) \right) 
 \, .
\end{align}
\end{proof}
 
 The decay estimate for $Q(r)$, i.e., Proposition \ref{p:decay}, will follow easily 
  from Corollary \ref{c:o3} and 
 the following elementary algebraic fact:
 
  \begin{Lem}	\label{l:alg}
  If $0 < a <  b \leq 1$,  $\alpha \in (0,1)$, and   
  $a^{2-\alpha} \leq C' \, (b-a)$, then 
  \begin{align}
  	a^{\alpha - 1} - b^{\alpha -1} \geq C \, ,
  \end{align}
  where $C$ depends on $\alpha$ and $C'$.
  \end{Lem}
  
  \begin{proof}
  Since $\alpha < 1$ and $0 < a < b \leq 1$, the fundamental theorem of calculus gives
  \begin{align}
  	a^{\alpha - 1} - b^{\alpha -1} &=  (1-\alpha) \, \int_a^b t^{\alpha -2} \, dt \leq (1-\alpha)
	\, (b-a) \, \max\,  \left\{ t^{\alpha -2} \, | \, t \in (a,b) \right\} \notag \\
	&=  (1-\alpha)
	\, (b-a) \, a^{\alpha - 2} \geq \frac{(1-\alpha)}{C'} \, , 
  \end{align}
  where the last inequality used the hypothesis that $a^{2-\alpha} \leq C' \, (b-a)$. 
  \end{proof}

  \begin{proof}[Proof of 
  Proposition \ref{p:decay}]
  Given $j$ so that $r=2\, (4^j)$ satisfies \eqr{e:conedel}, then 
   \eqr{e:mydeal2} gives 
    \begin{align}		
	    Q(4^{j+1})^{2-\alpha}    \leq     C' \, \left( Q(4^{j}) - Q(4^{j+1}) \right) 
 \, ,
\end{align}
where $C'$ is independent of $j$.
Applying  Lemma \ref{l:alg} with $a = Q(4^{j+1})$ and $b= Q(4^{j})$ gives
    \begin{align}
  	Q(4^{j+1})^{\alpha - 1} - Q(4^{j})^{\alpha -1} \geq C \, .
  \end{align}
  Therefore, if  $r=2\, (4^j)$ satisfies \eqr{e:conedel} for $j_1 \leq j \leq j_2$, then iterating this gives
    \begin{align}	\label{e:fu0}
  	Q(4^{j_2 +1})^{\alpha - 1} \geq Q(4^{j_1 +1})^{\alpha -1}  + C\, (j_2 - j_1)  
	 \, .
  \end{align}
  If we set $\beta = \frac{1}{1-\alpha} - 1$, then $\beta > 0$ and \eqr{e:fu0} 
  gives
    \begin{align}	
  	Q(4^{j_2 +1})  \leq   C \, (j_2 - j_1)^{\frac{1}{\alpha -1}} =  C \, (j_2 - j_1)^{-\beta -1}
	 \, .
  \end{align}
Using the monotonicity of $Q$, we conclude that if every $R \in (r , s)$ satisfies \eqr{e:conedel}, then
     \begin{align}	 
  	 Q(s) \leq \frac{C}{|\log (s/r)|^{\beta +1}}  \, ,
  \end{align}
  completing the proof.

  \end{proof}

 \subsection{Distance to cones}
 
Let the point $y \in M$ be the pole for the Green's function.  Following definition $4.2$ in \cite{C2}, define the quantity $\Theta_r$ to be the scale invariant Gromov-Hausdorff distance from the annulus 
\begin{align}	\label{e:dbinf}
	B_{\frac{4r}{b_{\infty} }}(y) \setminus 
B_{ \frac{r}{b_{\infty} } }(x) \subset M
\end{align}
 to the corresponding annulus centered at the vertex in the closest metric cone.  
Here, we have divided by $b_{\infty}$ since  the function $b$ is not asymptotic to the distance function $r$, but rather to $b_{\infty} \, r$.  
 Thus, if $\Theta_r < \epsilon$, then there is a cone $C_r$ so that
\begin{align}	\label{e:scaleinvd}  
	d_{GH} \, \left( B_{    \frac{4r}{b_{\infty}}     }(y) \setminus 
B_{ \frac{r}{b_{\infty} } } (x) \subset M , B_{ \frac{4r}{b_{\infty} } } \setminus 
B_{ \frac{r}{b_{\infty} }  } \subset C_r \right) < \epsilon \,   \frac{r}{b_{\infty} }   \, ,
\end{align}
where the balls in $C_r$ are centered at the vertex of the cone $C_r$.

We need the following fact which follows from \cite{ChC1}:  Given $\mu > 0$, there exists $C_{\mu}$ so that
\begin{align}	\label{e:chc1}
	\Theta_{r}^{2+\mu} \leq C_{\mu} \, \left[ Q(r/2) - Q(8r) \right] \, .
\end{align}
(In the current case, where we already know that we are close to a fixed Ricci-flat cone with smooth cross-section, this can also be proven directly using the estimates from Section \ref{s:4and5}.)

\vskip2mm
The last properties of $\Theta_r$ that we will need  are the following criteria for uniqueness (cf. Theorem $4.6$ in \cite{C2}) and an effective version of it that follows afterwards:

\begin{Lem}	\label{l:uni}
If $\sum_{j=1}^{\infty} \, \Theta_{2^j} < \infty$, then $M$ has a unique tangent cone at infinity.
\end{Lem}

  \begin{proof}
  To keep notation simple within this proof, we will argue as if $b_{\infty} =1$.
  For each $j$, we get a cone
  $C_j$ so that
  \begin{align}	\label{e:GH0}
  	d_{GH} \, (B_{4\, 2^j} (x) \setminus B_{2^j}(x) \subset M , B_{4 \, 2^j} \setminus B_{2^j}   \subset C_j) \leq 2  \, \Theta_{2^j} 
	\, 2^j \, .
\end{align}
Let $A_j$ denote the annulus $B_{  2^{j+1}} (x) \setminus B_{2^j}(x) \subset M$ and define the rescaled annuli $\bar{A}_j$ by
\begin{align}
	\bar{A}_j = 2^{-j} \, A_j \, .
\end{align}
Since two cones that agree on an annulus   must be   equal, it suffices to prove that the sequence
$\bar{A}_j$ is Cauchy with respect to Gromov-Hausdorff distance.  This will follow from  the triangle inequality once we  show that the sequence $d_{GH}(\bar{A}_j , \bar{A}_{j+1})$ is summable. 

The bound \eqr{e:GH0}    implies that
  \begin{align}	\label{e:GH}
  	d_{GH} \, (\bar{A}_j , B_{2} \setminus B_{1}   \subset C_j)  
	&= 2^{-j} \,  d_{GH} \, (A_j , B_{2^{j+1}} \setminus B_{2^j}   \subset C_j) \leq 2  \, \Theta_{2^j}  \, , \\
	d_{GH} \, (\bar{A}_{j+1} , B_{2} \setminus B_{1}   \subset C_j)  
	&= 2^{-j-1} \,  d_{GH} \, (A_{j+1} , B_{2^{j+1}} \setminus B_{2^j}   \subset C_j) \leq  \Theta_{2^j}
	 \, .
\end{align}
Combining these bounds with the triangle inequality gives 
\begin{align}
	d_{GH} (\bar{A}_j , \bar{A}_{j+1}) \leq d_{GH} \, (\bar{A}_j , B_{2} \setminus B_{1}   \subset C_j) +
	d_{GH} \, (\bar{A}_{j+1} , B_{2} \setminus B_{1}   \subset C_j) \leq 3 \, \Theta_{2^j} \, .
\end{align}
It follows that the sequence $d_{GH}(\bar{A}_j , \bar{A}_{j+1})$ is summable, completing the proof.
  \end{proof}
  
  We will also use the following effective version of Lemma \ref{l:uni}:
  
  \begin{Lem}	\label{l:uni2}
  Fix $R> 0$.
Let $A_j$ denote the annulus $B_{  2^{j+1}\, R} (x) \setminus B_{2^j \, R}(x) \subset M$ and define the rescaled annuli $\bar{A}_j$ by
\begin{align}
	\bar{A}_j = \frac{1}{2^{j}\, R}  \, A_j \, .
\end{align}
Given integers $j_1 < j_2$, then
\begin{align}
	\sup \, \, \left\{ d_{GH} (\bar{A}_i  , \bar{A}_j) \, | \, j_1 \leq i,j \leq j_2 \right\} \leq
	3\, \sum_{j=j_1}^{j_2} \, \Theta_{2^j \, R \, b_{\infty}}   \, .
\end{align}
\end{Lem}
 
 \begin{proof}
 This follows as in the proof of Lemma \ref{l:uni}.
 \end{proof}
 \subsection{Uniqueness}

 Uniqueness will follow by combining Lemma \ref{l:uni2} with the following modification of Theorem $4.6$ in \cite{C2}.
 
 \begin{Pro}	\label{p:pbar}
 There exist   $\bar{C}$, $\bar{\beta} > 0$ so that
  if every $r \in (R , 2^m \, R)$ satisfies \eqr{e:conedel}, 
then 
\begin{align}
	\sum_{j=j_1}^{m} \, \Theta_{2^j \, R} \leq \bar{C} \, j_1^{-\bar{\beta}} \, .
\end{align}
 \end{Pro}
 
 \begin{proof}
 By scaling, we may assume that $R=1$.
 
 Given any $\mu > 0$, $\gamma > 0$, and $j_1 < j_2$,
H\"older's inequality for series  gives 
 \begin{align}
  	  \sum_{j=j_1}^{j_2} \, \Theta_{2^j}  \leq  \left( 
	\sum_{j=j_1}^{j_2} \, \Theta_{2^j}^{2+\mu }  
		j^{\gamma \, (2+\mu )} \right)^{ \frac{1}{2+\mu } }  \, \left(  \sum_{j=1}^{\infty} \,  
		\left(j^{-\gamma} \right)^{ \frac{2+\mu }{1+\mu }} \right)^{ \frac{1+\mu }{2+\mu }} \, .
\end{align}
 The series in the last term is  summable whenever we have
 \begin{align}	\label{e:c1}
 	\left( \frac{2+\mu }{1+\mu } \right) \, \gamma > 1 \, .
 \end{align}
   To bound the remaining term, we bring in \eqr{e:chc1} 
  	  to get 
  \begin{align}	\label{e:sectg}
  	 	\sum_{j=j_1}^{j_2} \, \Theta_{2^j}^{2+\mu}  
		j^{\gamma \, (2+\mu)}  \leq  C_{\mu} \, \sum_{j_1=1}^{\infty} \,\left[  Q (2^{j-1})
		- Q(2^{j+3}) \right]  \, 
		j^{\gamma \, (2+\mu )} \, .
\end{align}
By assumption,  every $r \in (1 , 2^{j_2})$ satisfies \eqr{e:conedel}, so
Proposition \ref{p:decay} gives for $j \leq j_2$ 
  \begin{align}
  	Q(2^j) \leq C \, j^{-1 - \beta} \, ,  
\end{align}
so
Lemma \ref{l:elser} below applies as long as
\begin{align}	\label{e:c2}
	\gamma \, (2+\mu )  <  1+ \beta  \, .
\end{align}
Since $\beta > 0$, we can choose $\mu > 0$ and $\gamma > 0$ so that both \eqr{e:c1}
and \eqr{e:c2} are satisfied.  
Therefore,  we get that \eqr{e:sectg} is bounded by
  \begin{align}	\label{e:sectg2}
  	 	\sum_{j=j_1}^{j_2} \, \Theta_{2^j}^{2+\mu}  
		j^{\gamma \, (2+\mu)}  \leq  C_{\mu} \sum_{j=j_1}^{\infty} \,\left[  Q (2^{j-1})
		- Q(2^{j+3}) \right]  \, 
		j^{\gamma \, (2+\mu )} \leq C \,  j_1^{\gamma \, (2+\mu )-1-\beta}
		 \, .
\end{align}
 \end{proof}
 
 The preceding proposition used the following elementary lemma for sequences:
 
 \begin{Lem}	\label{l:elser}
 Suppose that $\beta > 0$ and $\{ a_j\} $ is a monotone non-increasing sequence with
 \begin{align}	\label{e:aj}
 		0 \leq a_j \leq C \, j^{-1- \beta} \, .
\end{align}
For any positive integers $k$ and $m$ and constant $\nu \in [1, 1+\beta) $, then we have
\begin{align}
	\sum_{j=m}^{\infty} \, \left[ a_j - a_{j+k} \right] \, j^{\nu}
	\leq  C \, k \,   \frac{\beta + 1}{  \beta + 1 - \nu }  \, m^{\nu - 1 - \beta}
	 < \infty \, .
\end{align}
 \end{Lem}
 
 \begin{proof}
Given $N > m$, we have
 \begin{align}	\label{e:hwg}
	\sum_{j=m}^{N} \, \left[ a_j - a_{j+k} \right] \, j^{\nu} &=
	\sum_{j=m}^N a_j \, j^{\nu} - \sum_{j=m+k}^{N+k} a_j (j-k)^{\nu} \notag \\
		&= \sum_{j=m}^{m+k-1}a_j \, j^{\nu} - \sum_{j=N+1}^{N+k} a_j (j-k)^{\nu}
	+ \sum_{j=m+k}^N a_j \left( j^{\nu} - (j-k)^{\nu} \right)
	 \, .
\end{align}
Using   \eqr{e:aj} and noting that $j^{\nu- 1 - \beta}$ is decreasing in $j$, 
the first sum is bounded by
 \begin{align}	 
	 \sum_{j=m}^{m+k-1}a_j \, j^{\nu} \leq C \,  \sum_{j=m}^{m+k-1}  j^{\nu- 1 - \beta}
	 \leq C \, k \,  m^{\nu- 1 - \beta}
	 \, .
\end{align}
To prove the lemma, we have to handle the last sum in \eqr{e:hwg}.
Since $\nu \geq 1$, the fundamental theorem of calculus gives
\begin{align}
	j^{\nu} - (j-k)^{\nu} = \nu \,  \int_{j-k}^j   t^{\nu - 1} \, dt \leq k \, \nu \, j^{\nu - 1} \, .
\end{align}
Putting this in, then using \eqr{e:aj}, and then noting that $\nu - 2 - \beta < 0$  gives
\begin{align}
\sum_{j=m+k}^{N} a_j \left( j^{\nu} - (j-k)^{\nu} \right) &\leq k \, \nu \, \sum_{j=m+k}^{N} a_j \, j^{\nu-1}  
\leq C \, k \, \nu \,  \sum_{j=m+k}^{\infty}  j^{\nu-2-\beta}  \notag \\
&\leq  C \, k \, \nu \,  \int_{m}^{\infty}  t^{\nu-2-\beta} \, dt =  \frac{C \, k \, \nu \, m^{\nu - 1 - \beta}}{  \beta + 1 - \nu } 
\, ,
\end{align}
where we used that  $\nu - 2 - \beta < -1$.
 \end{proof}

\vskip2mm
We are now ready to prove uniqueness assuming that we have a functional $\cF$ that satisfies (1)--(5).  The rest of the paper will then be devoted to constructing $\cF$ and proving (1)--(5).

\begin{proof}[Proof of Theorem \ref{t:main} assuming (1)--(5)]
We   start by choosing constants $\delta > 0$, $j_1$ and $\epsilon > 0$:
\begin{itemize}
\item Fix  $\delta > 0$, so that (1)--(5) hold on any scale $r$ that satisfies \eqr{e:conedel}.

\item Proposition \ref{p:pbar}
 gives $\bar{C}$, $\bar{\beta} > 0$ so that
  if every $r \in (R , 2^m \, R)$ satisfies \eqr{e:conedel}, 
then 
\begin{align}
	\sum_{j=j_1}^{m} \, \Theta_{2^j \, R} \leq \bar{C} \, j_1^{-\bar{\beta}} \, .
\end{align}
Fix an integer $j_1 = j_1 (\bar{C} , \bar{\beta})$ so that  $\bar{C} \, j_1^{-\bar{\beta}}  < \delta / 100$.
\item Using \cite{ChC1}, fix $\epsilon > 0$ so that if $A (r/2) - A(8r) < \epsilon$, then $\Theta_r < \delta / 100$.
\end{itemize}
Suppose now that $R > 0$ and an integer $m \geq j_1$ satisfy:
\begin{enumerate}
\item[(A)] Every $r \in (R , 2^{j_1} \, R)$ satisfies  \eqr{e:conedel} with $\delta/100$ in place of $\delta$.
\item[(B)] $A(R/2) - A( 2^{m+3} \, R) < \epsilon$.

\end{enumerate}

Suppose that $k \in [j_1 , m-1]$ .
If    $r \in (    R  , 2^{k} \, R)$ satisfies
\eqr{e:conedel} with $\delta_k \leq \delta/2$ in place of $\delta$, then (B) and the triangle inequality give that
$r \in (    R  , 2^{k+1} \, R)$ satisfies
\eqr{e:conedel} with 
\begin{align}
	\delta_k + 3\delta/100 < \delta 
\end{align}
 in place of $\delta$.  In particular, we can apply 
Proposition \ref{p:pbar} on this stretch  to get that
\begin{align}	\label{e:pbar}
	  \sum_{j=j_1}^{m} \, \Theta_{2^j \, R} \leq \bar{C} \, j_1^{-\bar{\beta}} <    \delta / 100 \, .
\end{align}
Consequently, Lemma \ref{l:uni2} and the triangle inequality give that
$r \in (    R  , 2^{k+1} \, R)$ satisfies
\eqr{e:conedel} with $4\delta/100 < \delta $ in place of $\delta$.  Since this bound is independent of $k$, we conclude that it holds on the entire interval  $(R , 2^{m} \, R)$.

We can use this to prove both the global uniqueness theorem (Theorem \ref{t:main}) and the effective version.  To prove 
Theorem \ref{t:main}, use the monotonicity of $A$ to pick some large $R$ so that (B) holds for every $m$.  It follows that
\eqr{e:conedel}  holds on the entire interval  $(R ,\infty)$ and \eqr{e:pbar} gives for $\bar{j} \geq j_1$ that
\begin{align}	\label{e:cfhere}
	 \sum_{j=\bar{j}}^{\infty} \, \Theta_{2^j \, R} \leq \bar{C} \, \bar{j}^{-\bar{\beta}} < \infty \, .
\end{align}
This implies uniqueness by Lemma \ref{l:uni}; combining it with Lemma \ref{l:uni2} gives the rate of convergence.

\end{proof}

We will next describe the modifications   needed for the effective version of uniqueness.

\begin{proof}[Proof of Theorem \ref{t:effective}]
The first claim (E1) follows as in the proof of the uniqueness theorem, with (A) and (B) in the proof now given by the assumptions instead of by taking $R$ sufficiently large.
Furthermore, arguing as there (see \eqr{e:cfhere} and Lemma \ref{l:uni2}) gives an ``effective Cauchy bound'' for
$r_1 < r < s < r_2$:
\begin{align}
	d_{GH} \, \left( \frac{1}{r} \, \left( B_{2r} (x) \setminus B_r(x) \right) , \frac{1}{s} \, \left( B_{2s} (x) \setminus B_s(x) \right) \right) \leq C \, \left( \log \frac{r}{r_1} \right)^{ - \bar{\beta} } \, .
\end{align}
Thus, we get that the maximal scale-invariant 
distance between any of these annuli    decays as claimed. 
Finally,   \eqr{e:pbar}  gives that $\Theta_r $ also 
decays like a power of $\log \frac{r}{r_1} $ so these annuli are close to an annulus in a fixed cone.
\end{proof}

 \section{Functionals on the space of metrics and measures}	\label{s:functs}

In this section, we will define the functional $\cF$   and verify properties (1) and (2) of $\cF$.
Recall that    $g_0$ is a fixed Einstein metric on an $(n-1)$-dimensional manifold
$N$ with $\Ric_{g_0} = (n-2) \, g_0$,    $\cA$ is the set of $C^{2,\beta}$ metrics $g$ and positive $C^{2,\beta}$  functions $w$, 
and  $\cA_1 \subset \cA$ are the ones satisfying the weighted volume constraint 
\begin{align}
	\cA_1 = \left\{ (g,w) \in \cA \, | \, \int_N w \, d\mu_g = \Vol (\partial B_1 (0)) \right\} \, .
\end{align}
As we saw, $(b_{\infty}^{-2} \, g_0 , b_{\infty}) \in \cA_1$.
The tangent space $\ca$ to $\cA$ is given by the set of symmetric $2$-tensors $h$ and functions $v$, with $(h,v)$ being tangent to the path{\footnote{This normalization simplifies some later computations.}}
\begin{align}
	(g+t\, h , w \, \e^{ t \, v} ) \, .
\end{align}
  The linear space $\ca$  comes with a natural inner product
\begin{align}
	\langle (h_1 , v_1) , (h_2 , v_2) \rangle_{(g,w)} = \int_N 
	\left\{ \langle h_1 , h_2 \rangle_g + v_1 \, v_2 
	\right\} \, w \, d\mu_g
	\, .
\end{align}

\begin{Lem}	\label{l:obvious}
The variation $(h,v)$ is tangent to $\cA_1$ at $(g,w)$ if and only if
\begin{align}
	\int_N \left( \frac{1}{2} \, \Tr (h) +v \right) \, w \, d\mu_g = 0  \, .
\end{align}
\end{Lem}

\begin{proof}
This follows immediately from integrating
\begin{align}	\label{e:wtd}
	\left( (w\e^{tv}) \, d\mu_{g+th} \right)' &= \left( \frac{1}{2} \, \Tr (h) + v \right) \, w \, d\mu_g \, .
\end{align}
\end{proof}

The functional $\cF$ will be a linear combination of 
 two natural functionals on $\cA$   given by
\begin{align}
	A (g,w) &= \int_N  w^3 \, d\mu_g \, , \\
	B(g,w) &= \int_N  R_g \, w \, d\mu_g \,   .
\end{align}
where $R_g$ is the scalar curvature of the metric $g$.  The coefficients of $A$ and $B$   will be chosen
so that $\cF$ satisfies (1) and (2).   
 
\vskip2mm
The next proposition computes the first derivatives of $A$ and $B$  at $(g,w)$.

\begin{Pro}		\label{p:fvar1}
Given one parameter families $g+th$ and $w\,\e^{tv}$, we get  
\begin{align}
	A' & = \int_N \left\{ w^2 \, \left( \frac{1}{2} \, \Tr (h) +v  \right)  + 2 w^2 \, v
	 \right\} \, w \, d\mu_g \, , \\
	 B' & = \int_N \left\{   - \langle \Ric_g , h \rangle + \langle h , \frac{ \Hess_w}{w} \rangle - \Tr (h)\frac{\Delta w}{w}  + 
	R_g \, \left( \frac{1}{2} \, \Tr (h) + v \right) \right\} \, w \, d\mu_g    	\, .
\end{align}
\end{Pro}

\begin{proof}
Since $\left[ \left(w\,\e^{tv} \right)^2 \right]' = 2 \, w^2 \, v$, the first claim follows from the formula \eqr{e:wtd}
for the derivative of the weighted volume form.
 Using Lemma \ref{l:topping} and \eqr{e:wtd}, the variation of $B$ is  
\begin{align}
	B' &= \int_N \left\{ R_g' + R_g \, \left( \frac{1}{2} \, \Tr (h) + v \right) \right\} \, w \, d\mu_g \notag  \\
	&= \int_N \left\{ \left( - \langle \Ric_g , h \rangle + \delta^2 \, h - \Delta \Tr (h) \right) + 
	R_g \, \left( \frac{1}{2} \, \Tr (h) + v \right) \right\} \, w \, d\mu_g \, .
\end{align}
This almost gives what we want, except that two of the terms have derivatives applied to $h$.  We will integrate by parts to take these off.
Namely, Stokes' theorem gives that
\begin{align}
	\int_N w \, \Delta \, \Tr (h)   \, d\mu_g &= \int_N \Tr (h) \, \Delta w \, d\mu_g \, , \\
	\int_N w \, \delta^2 \, h \, d\mu_g &= - \int_N \langle \nabla w , \delta \, h \rangle \, d\mu_g = \int_N \langle h , \Hess_w \rangle \, d\mu_g \, .
\end{align}
\end{proof}

The next corollary uses the first variation formulas to choose a linear combination $\cF$ of $A$ and $B$   so that $\cF (b_{\infty}^{-2} \, g_0 , b_{\infty}) = A_{\infty}$ and  $(g_0 , b_{\infty})$ is a critical point, i.e., (1) and (2) hold.

\begin{Cor}	\label{c:wtd}
Given $b_{\infty} > 0$, 
the pair $(b_{\infty}^{-2} \, g_0,b_{\infty})$ is a critical point for the functional 
\begin{align}
	\cF \equiv \frac{1}{2-n} \, \left( A - \frac{B }{(n-2)}    \right)
\end{align}
 restricted to the subset $\cA_1$ and, moreover, $\cF (b_{\infty}^{-2} \, g_0 , b_{\infty}) = A_{\infty}$.
\end{Cor}

\begin{proof}
To simplify notation, set $\bar{g} = b_{\infty}^{-2} \, g_0$.
Since $g_0$ is Einstein with $\Ric_{g_0} = (n-2) \, g_0$,  
\begin{align}
	R_{\bar{g}} &= b_{\infty}^2 \, R_{g_0} = b_{\infty}^2 \, (n-1) \, (n-2) \, , \\
	\Ric_{\bar{g}} &= b_{\infty}^2 \,   (n-2) \, \bar{g} \, .
\end{align}
   Hence, 
at $(\bar{g} , b_{\infty})$, Proposition \ref{p:fvar1} gives that 
\begin{align}
	A' &    = 2\, b_{\infty}^3 \,  \int_N v
	  \,   \, d\mu_{\bar{g}} 
	=
	- b_{\infty}^3 \,  \int_N  \Tr (h)  \, d\mu_{\bar{g}} 
	   \,, \\
	B' & = - b_{\infty} \, \int_N \langle \Ric_{\bar{g}} , h \rangle \, d\mu_{\bar{g} }
	=
	(2-n) \, b_{\infty}^3 \int_N  \Tr (h) \,  d\mu_{\bar{g}}  \, ,  
\end{align}
where the first two equations used that the integral of $\Tr (h) + 2v$ is zero because of the weighted volume constraint.
This gives the first claim.  

 For the second claim, observe that  
 \begin{align}
 	 \left( A - \frac{B }{(n-2)}    \right) (\bar{g} , b_{\infty}) &=
	 \int_N \left\{ b_{\infty}^2   - \frac{b_{\infty}^2  \, (n-1)(n-2) }{(n-2)} \right\} b_{\infty} \, d\mu_{\bar{g}} \notag \\
	 &=(2-n) \,  b_{\infty}^2   \int_M b_{\infty} \, d\mu_{\bar{g}} = 
	 (2-n) \, b_{\infty}^2 \, \Vol (\partial B_1 (0)) = (2-n) \, A_{\infty} \, .
 \end{align}
\end{proof}

\subsection{The gradient of $\cF$}

We will next compute the gradient of $\cF$ as a functional on the full space of metrics $g$ and weights $w$.  The starting point is the following lemma that computes the directional derivative of $\cF$.

\begin{Lem}	\label{l:gradF}
Given one parameter families $g+th$ and $w\,\e^{tv}$, we have
\begin{align}
	(2-n) \, \cF' &= \int_N 
	\left\{ \left( 3\, w^2  - \frac{R_g}{n-2} \right) \, \left( \frac{1}{2} \, \langle g , h \rangle_g +v  \right)    +    \langle \left( \frac{\Ric_g}{n-2} - w^2 \, g\right) , h \rangle_g
	  \right\} \, w \, d\mu_g \notag \\
	  &\qquad + \frac{1}{n-2} \, \int_N 
	    \langle \left((\Delta w) \, g  -  \Hess_w \right) , h \rangle_g
	  \,   d\mu_g  
	  \, .  
\end{align}
\end{Lem}

\begin{proof}
It is convenient to set $\phi =  \left( \frac{1}{2} \, \Tr (h) +v  \right) $.
Proposition \ref{p:fvar1} gives  
\begin{align}
	A' & = \int_N \left\{ w^2 \, \phi  + 2 w^2 \, v
	 \right\} \, w \, d\mu_g = \int_N \left\{ 3\, w^2 \, \phi  -   w^2 \, \langle g , h \rangle 
	 \right\} \, w \, d\mu_g \, , \\
	 B' & = \int_N \left\{   - \langle \Ric_g , h \rangle + \langle h , \frac{ \Hess_w}{w} \rangle - \Tr (h)\frac{\Delta w}{w}  + 
	R_g \, \phi \right\} \, w \, d\mu_g  	\, .
\end{align}
Using the  equations for $A'$ and $B'$ gives
\begin{align}
	\left( A - \frac{B}{n-2} \right)' &= \int_N 
	\left\{ \left( 3\, w^2  - \frac{R_g}{n-2} \right) \, \phi  +    \langle \left( \frac{\Ric_g}{n-2} - w^2 \, g\right) , h \rangle 
	  \right\} \, w \, d\mu_g \notag \\
	  &\qquad + \frac{1}{n-2} \, \int_N 
	    \langle \left((\Delta w) \, g  -  \Hess_w \right) , h \rangle 
	  \,   d\mu_g \, .
\end{align}
 
\end{proof}

The previous lemma computed the directional derivative of $\cF$.  To get the gradient, we need to write it in terms of inner products for a fixed background  metric $\bar{g}$.

\begin{Lem}	\label{l:gradadjust}
If $h$ and $J$ are symmetric $2$-tensors, while $g$ and $\bar{g}$ are metrics, then
\begin{align}
	\langle h , J \rangle_g = \langle h ,  \Psi (J) \rangle_{\bar{g}} \, ,
\end{align}
where $\Psi$ is the mapping  defined by
$
	[\Psi (J)]_{ij} = \bar{g}_{ik} g^{kn} J_{nm} g^{m\ell} \bar{g}_{\ell j} \, .
$  If $g = \bar{g} + t \, h$, then
\begin{align}
	\frac{d}{dt} \, \big|_{t=0} \, \Psi (J)_{ij}  = J_{ij}' - h_{ip} \,  \bar{g}^{pn}  \, J_{nj}  
	-    J_{im}   \bar{g}^{mp} \, h_{pj}  \, .
\end{align}
\end{Lem}

\begin{proof}
Expanding the first expression out, we have
\begin{align}
	\langle h , J \rangle_g =h_{ij} J_{kn} g^{ik} g^{jn}   \, .
\end{align}
On the other hand, we get
\begin{align}
	 \langle h ,  \Psi (J) \rangle_{\bar{g}} &= h_{pq} \bar{g}^{pi} \bar{g}^{qj} \, [\Psi (J)]_{ij} = h_{pq} \bar{g}^{pi} \bar{g}^{qj} \, \bar{g}_{ik} g^{kn} J_{nm} g^{m\ell} \bar{g}_{\ell j} \notag \\
	 &= h_{pq}   \, \delta_{pk} g^{kn} J_{nm} g^{m\ell} \delta_{\ell q} =   h_{k \ell}    g^{kn} J_{nm} g^{m\ell}
	 \, .
\end{align}
Suppose now that we have a one-parameter family of metrics $g = \bar{g} + t \, h$ and  both $\Psi$ and $J$ depend on $t$.    Differentiating at $t=0$ and using that $\Psi$ is the identity at $t=0$ gives
\begin{align}
	\left[ \Psi (J)_{ij}  \right]' & = J_{ij}' +  \bar{g}_{ik} \left(  g^{kn} \right)' J_{nj}  
	+    J_{im} \left( g^{m\ell} \right)' \bar{g}_{\ell j} \notag \\
	&=  J_{ij}' - h_{ip} \,  \bar{g}^{pn}  \, J_{nj}  
	-    J_{im}   \bar{g}^{mp} \, h_{pj} \, , 
\end{align}
where the last equality used that  $\left( g^{m\ell} \right)'  = - g^{mp} h_{pq} g^{q\ell}$ (and the corresponding equation for the derivative of $g^{kn}$).
\end{proof}

 We will apply Lemma \ref{l:gradadjust} with $\bar{g}$ equal to the background metric $\bar{g} = b_{\infty}^{-2} \, g_0$.  The next corollary uses the lemma to calculate the gradient of $\cF$ on the space of all variations; later, we will project this onto $\cA_1$.

\begin{Cor}	\label{c:gradF}
The gradient of $\cF$ at $(g,w)$ is given by
\begin{align}
	(2-n) \, \nabla \cF = \left(   \frac{1}{2} \phi_1   \Psi (g) + \Psi (J) , \, \phi_1   \right) \, \nu 
	\, ,
\end{align}
where we define functions $\nu$ and  $\phi_1$   by
\begin{align}
	\nu &= \frac{ w \, \sqrt{ \det (g)}}{b_{\infty} \, \sqrt{\det (b_{\infty}^{-2} \, g_0)}} \, , \\
	\phi_1 & = 3\, w^2  - \frac{R_g}{n-2} \, ,  
\end{align}
and we define the $2$-tensor $J = J_1 + J_2  $ by
\begin{align}
	J_1 &= \frac{\Ric_g}{n-2} - w^2 \, g \, , \\
	J_2 &= \frac{1}{n-2} \, \left(  \frac{\Delta w}{w} \, g - \frac{ \Hess_w}{w} \right)   \, .
\end{align}
\end{Cor}

\begin{proof}
Given one parameter families $g+th$ and $w\,\e^{tv}$, 
Lemma \ref{l:gradF} gives that
\begin{align}
	(2-n) \, \cF' &= \int_N 
	\left\{ \phi_1  \, \left( \frac{1}{2} \langle g , h \rangle_g +v  \right)    +    \langle J  , h \rangle_g
	  \right\} \, w \, d\mu_g   \notag \\
	  &=\int_N 
	\left\{   \frac{1}{2}  \phi_1  \, \langle   g , h \rangle_g    +    \langle J  , h \rangle_g
	+ \phi_1  \, v
	  \right\} \, \nu \,  b_{\infty} \,  d\mu_{b_{\infty}^{-2} g_0} 
	  \, .  
\end{align}
Lemma \ref{l:gradadjust} gives the corollary.

\end{proof}

For the next corollary, it is useful to define the functional $A_1$ by
\begin{align}
	A_1 (g,w) = \int_N w \, d\mu_g \, .
\end{align}
The next corollary computes the gradient of $A_1$.

 \begin{Cor}	\label{c:gA1}
The gradient of $A_1$ at $(g,w)$ is given by $\nabla A_1 = \left( \frac{1}{2} \, \Psi (g) , 1 \right) \, \nu $ where
 \begin{align}
	\nu &= \frac{ w \, \sqrt{ \det (g)}}{b_{\infty} \, \sqrt{\det (b_{\infty}^{-2} \, g_0)}}  \, .
\end{align}
 \end{Cor}
 
 \begin{proof}
 Given one parameter families $g+th$ and $w\,\e^{tv}$,  
differentiating $A_1$ gives
\begin{align}
	A_1'   &= \int_N \left( \frac{1}{2} \, \langle g , h \rangle_g +v \right) \, w \, d\mu_g   = \int_N \left( \frac{1}{2} \, \langle g , h \rangle_g +v \right) \, \nu \, b_{\infty}    \, d\mu_{b_{\infty}^{-2} g_0}
	  \, .
\end{align}
 Lemma \ref{l:gradadjust} gives the corollary.

\end{proof}

\section{Proving properties (4) and (5)}	\label{s:4and5}
	
In this section, we will  show that when $\cF$ is applied to the level sets of $b$, then it satisfies properties (4) and (5).  A key for both of these will be to show in the next subsection that an $L^2$ bound on the trace-free Hessian of $b^2$ implies scale-invariant $C^1$ bounds.

As in section \ref{s:s1},  will assume throughout this section that we are working on a scale $R$ where the Hessian of $b^2$ is almost diagonal and $|\nabla b|$ is almost constant.

  \subsection{$C^1$ bounds on the trace free Hessian}

  \begin{Thm}	 \label{t:meanval}
  There exists a constant $C$ so that
  \begin{align}
  	\| \Hess_{b^2} - \frac{\Delta b^2}{n} \, g \|^2_{C^1 (b=R)} \leq 
  	 C \,
 \int_{\frac{R}{2} \leq b \leq \frac{3\,R}{2}}  b^{-n}\,  \left| \Hess_{b^2} - \frac{\Delta b^2}{n} \, g
	   \right|^2    
 \, ,
  \end{align}
  where $\| \cdot \|_{C^1 (b=R)}$ is the scale-invariant $C^1$-norm on $M$ at $b=R$.
  \end{Thm}
  
  Here, ``scale-invariant'' means measured with respect to the rescaled metric $R^{-2} \, g_R$, where $g_R$ is the induced metric on the level set $b=R$.  Namely,  at $b=R$
  \begin{align}
  	 \left| \Hess_{b^2} - \frac{\Delta b^2}{n} \, g
	   \right|^2 + R^2 \,  \left| \nabla \left\{  \Hess_{b^2} - \frac{\Delta b^2}{n} \, g \right\}
	   \right|^2   \leq C \,  \int_{\frac{R}{2} \leq b \leq \frac{3\,R}{2}}  b^{-n}\,  \left| \Hess_{b^2} - \frac{\Delta b^2}{n} \, g
	   \right|^2  \, . \notag 
  \end{align}
       
     We will need the following Bochner type formula for the Hessian in the proof.
     
     \begin{Lem}	    \label{l:hessb}
     We have
       \begin{align}	 
 	\left(  \Delta \Hess_w \right)_{jk}  &
	= (\Delta w)_{jk}  + 2\, R_{ij\ell k}w_{i\ell}  \, .
 \end{align}
     \end{Lem}
     
     \begin{proof}
     
Let $w$ be a function and  $e_i$ an orthonormal frame.
  The definition of the curvature tensor gives
 \begin{align}
 	w_{ijk} - w_{ikj} =  \nabla_{[e_k , e_j]} \nabla w   + R(e_k,e_j) \nabla w   \, .
 \end{align}
  To simplify notation, we will assume that the $e_i$'s are coordinate vector fields (so that the brackets all vanish) and that we are working at a point where
 $\nabla_{e_i} e_j = 0$ for every $i,j$.  
 
 Since $\nabla_{e_i} e_i = 0$ at this point, the Laplacian of the Hessian is 
 \begin{align}
 	 \Delta \Hess_w = \nabla_{e_i} \nabla_{e_i} \nabla \nabla w \, ,
 \end{align}
and combining this with $\nabla_{e_i} e_j = 0$ at the point gives  
  \begin{align}	\label{e:kla}
 	\left(  \Delta \Hess_w \right)_{jk}  =   \langle  \nabla_{e_i} \nabla_{e_i} \nabla_{e_j} \nabla w , e_k \rangle
	- \langle   \nabla_{ \nabla_{e_i} \nabla_{e_i} e_j} \nabla w , e_k \rangle
	   \, .
 \end{align}
 Using   the definition of the curvature and the properties of the $e_i$'s, we get at this point
   \begin{align}
 	\langle  \nabla_{e_i} \nabla_{e_i} \nabla_{e_j} \nabla w , e_k \rangle  &=   \langle  \nabla_{e_i} \left( \nabla_{e_j} \nabla_{e_i} \nabla w +
	\nabla_{[e_i , e_j]} \nabla w + R(e_i , e_j) \nabla w ) \right)  , e_k \rangle \notag \\
	&=  \langle  \nabla_{e_j}   \nabla_{e_i} \nabla_{e_i} \nabla w 
	+\nabla_{e_i}
	\nabla_{[e_i , e_j]} \nabla w
	+ R(e_i , e_j) (\nabla_{e_i} \nabla w)
	+ \nabla_{e_i} (R(e_i , e_j) \nabla w)    , e_k \rangle \\
	&=  \langle  \nabla_{e_j}   \nabla_{e_i} \nabla_{e_i} \nabla w 
	+\nabla_{e_i}
	\nabla_{[e_i , e_j]} \nabla w , e_k \rangle  + R_{ij\ell k}w_{i\ell}  + R_{ijnk} w_{in} \, , \notag 
 \end{align}
 where the last equality   used that   $\Ric = 0$ and, by the second Bianchi identity and $\Ric =0$,  
 \begin{equation}
 	(\nabla R)_{i i j n k} = 0 \, .
 \end{equation}
 Since $[e_i , e_j]$ vanishes at the point, we have
 $\nabla_{e_i}
	\nabla_{[e_i , e_j]} \nabla w = 
	\nabla_{ \nabla_{e_i} [e_i , e_j]} \nabla w$ and we get
   \begin{align}	\label{e:lefthalf}
 	\langle  \nabla_{e_i} \nabla_{e_i} \nabla_{e_j} \nabla w , e_k \rangle  &
	=  \langle  \nabla_{e_j}   \nabla_{e_i} \nabla_{e_i} \nabla w 
	+ 
	\nabla_{\nabla_{e_i} [e_i , e_j]} \nabla w , e_k \rangle  + 2\, R_{ij\ell k}w_{i\ell}   \, .
 \end{align}

 On the other hand, $\Ric = 0$ implies that $\nabla \Delta w = \Delta \nabla w$, so we have
 \begin{align}
 	(\Delta w)_{jk} &= \langle \nabla_{e_j} \nabla \Delta w , e_k \rangle =
	\langle  \nabla_{e_j}   \Delta \nabla w , e_k \rangle   = \langle  \nabla_{e_j}  \left( \nabla_{e_i} \nabla_{e_i} \nabla w - \nabla_{ \nabla_{e_i} e_i } \nabla w 
	\right) , e_k \rangle
	\notag \\
	&= \langle  \nabla_{e_j}   \nabla_{e_i} \nabla_{e_i} \nabla w 
	- 
	\nabla_{\nabla_{e_j} \nabla_{e_i} e_i  } \nabla w , e_k \rangle
	\, .
 \end{align}
 Combining this with \eqr{e:kla} and \eqr{e:lefthalf} gives
 \begin{align}
 	\left(  \Delta \Hess_w \right)_{jk}  - (\Delta w)_{jk}  &= 2\, R_{ij\ell k}w_{i\ell}
	+ \langle 
	\nabla_{\nabla_{e_i} [e_i , e_j]} \nabla w-
	   \nabla_{ \nabla_{e_i} \nabla_{e_i} e_j} \nabla w 
	+ \nabla_{\nabla_{e_j} \nabla_{e_i} e_i  } \nabla w , e_k \rangle \, . \notag
 \end{align}
 To complete the proof, we observe that
 \begin{align}
 	\nabla_{e_i} [e_i , e_j] - \nabla_{e_i} \nabla_{e_i} e_j + \nabla_{e_j} \nabla_{e_i} e_i =
	- \nabla_{e_i} \nabla_{e_j} e_i +
	\nabla_{e_j} \nabla_{e_i} e_i = 0 
 \end{align}
 since $M$ is Ricci flat.
     \end{proof}

    \begin{proof}[Proof of Theorem \ref{t:meanval}]
    Set $B_b = \Hess_{b^2} - \frac{\Delta b^2}{n} \, g$, so that $B_b$ is trace free.
    Since  $\Delta b^2 = 2n \, |\nabla b|^2$, we have
    \begin{align}
    	B_b \equiv \Hess_{b^2} -  2 \, |\nabla b|^2 \, g \, .
    \end{align}
     Since $M$ is Ricci flat, a computation from \cite{C2} (see
   Lemma \ref{l:deltanu}) gives
    \begin{align}
    	b^2 \, \Delta |\nabla b|^2 = \frac{1}{2} \, |B_b|^2 + (2n-4) \, B_b(\nabla b , \nabla b) \, .
    \end{align}
    Lemma \ref{l:k1} gives
    \begin{align}
    	b \, \nabla |\nabla b|^2 = B_b(\nabla b) \, ,
    \end{align}
    so we know that
    \begin{align}
    	  \nabla b \otimes \nabla |\nabla b|^2 + b \, \Hess_{|\nabla b|^2} = \nabla (B_b(\nabla b)) \, .
    \end{align}
    We rewrite this as
     \begin{align}
    	  b^2 \, \Hess_{|\nabla b|^2} = b \, \nabla (B_b(\nabla b)) - \nabla b \otimes B_b(\nabla b) \, .
    \end{align}
   Thus, using Lemma \ref{l:hessb}, we compute
    \begin{align}
    	b^2 \, \Delta \Hess_{b^2} &= b^2 \, \Hess_{\Delta b^2} + 2\, b^2 \,  R_{ij\ell k}(b^2)_{i\ell}= 2n \, b^2 \, \Hess_{|\nabla b|^2} + 2\, b^2 \,  R_{ij\ell k} (b^2)_{i\ell} \notag \\
	&= 
	2n \, \left\{ b \, \nabla (B_b(\nabla b)) - \nabla b \otimes B_b(\nabla b) \right\} + 2\, b^2 \,  R_{ij\ell k} (B_b)_{i\ell}
	  \, ,
    \end{align}
    where the last equality also used that $\Ric = 0$ to get that
    \begin{align}
    	R_{ij\ell k} (B_b)_{i\ell} - R_{ij\ell k} (b^2)_{i\ell} = - 2 |\nabla b|^2 \,  R_{ij\ell k} \, g_{i \ell} = 0 \, .
    \end{align}
    On the other hand, the metric is parallel so we have
    \begin{align}
    	\Delta \, \left( 2 \, |\nabla b|^2 \, g \right) &= 2 \, g \, \Delta  |\nabla b|^2 =  \frac{g}{b^2} \, 
	\left( |B_b|^2 + 4(n-2) \, B_b(\nabla b , \nabla b) \right)
		\, .
    \end{align}
    Combining these, we see that
    \begin{align}	\label{e:318}
    	b^2 \, \Delta B_b &= 2n \, \left\{ b \, \nabla (B_b(\nabla b)) - \nabla b \otimes B_b(\nabla b) \right\} -
	\left\{  |B_b|^2 + 4(n-2) \, B_b(\nabla b , \nabla b) \right\} \, g \notag \\
	&\qquad + 2\, b^2 \,  R_{ij\ell k} (B_b)_{i\ell} \, .
    \end{align}
  Using this,  noting that $B_b$ is trace-free (so its inner product with $g$ is zero), and using that $b^2 \, R_{ij\ell k}$ is bounded by
  a constant $C$ (since we are close to a fixed cone), 
  we get the differential inequality
  \begin{align}
  	\frac{1}{2} \, b^2 \, \Delta \, \left| B_b \right|^2 &= b^2 \, \left| \nabla B_b \right|^2 +  \langle b^2 \, \Delta B_b , B_b \rangle  \notag   \\
	&\geq b^2 \,  \left| \nabla B_b \right|^2 
	- 2n \,    | B_b | \left\{  b\,| \nabla B_b | \, |\nabla b|  + |B_b| b \, |\Hess_b| + |\nabla b|^2 \,  |B_b| \right\}
	- C \, \left| B_b \right|^2 \, .
  \end{align}
Using the a priori bounds for $|\nabla b|$ and  $b \, |\Hess_b|$, and  the absorbing inequality, we get
 \begin{align}
  	\frac{1}{2} \, b^2 \, \Delta \, \left| B_b \right|^2 & \geq b^2 \,  \left| \nabla B_b \right|^2 
	- C_1    | B_b |\,  b\,| \nabla B_b |  
	- C_2 \, \left| B_b \right|^2 \geq \frac{1}{2} \, b^2 \,  \left| \nabla B_b \right|^2  - C_2' \, \left| B_b \right|^2 \, .
  \end{align}
 We will use this twice.  First,  this differential inequality allows us to use  the meanvalue inequality to get  the desired pointwise bound for $\left| B_b \right|^2$.  
 Second, using a cutoff function $\eta \geq 0 $ with support in the annular region and  arguing as in the reverse Poincar\'e inequality, we have  
 \begin{align}
 	0 &= \int \dv \, \left(  \eta^2 \, \nabla \left| B_b \right|^2 \right) \geq  
	\int  \left(\eta^2 \,   |\nabla B_b|^2 - 2 C_2'  \eta^2 \frac{|B_b|^2}{b^2}    - 4\, \eta \, |\nabla \eta| \, |B_b| \, \left|     \nabla B_b \right|  \right) \notag \\
	&\geq \int  \left(\frac{1}{2} \, \eta^2 \,   |\nabla B_b|^2 - 2 C_2' \eta^2 \frac{|B_b|^2}{b^2}    - 8 \,   |\nabla \eta|^2 \, |B_b|^2    \right) \, .
 \end{align}
 Since we are on the scale $R$, we have $|\nabla \eta| \leq \frac{C}{R}$ and $b \approx R$, so this yields
 \begin{align}
 	R^2 \, \int_{\frac{3R}{4} \leq b \leq \frac{5R}{4}}  \left| \nabla B_b \right|^2  \leq C \, 
	\int_{\frac{3R}{2} \leq b \leq \frac{3R}{2}}  \left|   B_b \right|^2
 \end{align}
 We will again use the meanvalue inequality to go from this integral bound to a pointwise bound for $|\nabla B_b|$.  We start with the ``Bochner formula'' for $\Delta \, \left| \nabla B_b \right|^2$ 
 \begin{align}	\label{e:quoteboch}
 	\Delta \, \left| \nabla B_b \right|^2 \geq 2\, \left| \nabla \, \nabla B_b \right|^2 - \frac{C}{b^2}  \, \left| \nabla B_b \right|^2
	+ 2 \, \langle \nabla B_b , \nabla \Delta B_b \rangle \, ,
 \end{align}
 where the constant $C$ comes from a scale-invariant curvature bound for $M$ which holds because it is $C^3$ close to a fixed cone on this scale.  Bringing in the formula \eqr{e:318} for $\Delta B_b$ and the a priori bounds that hold since $M$ is close to conical on this scale, we see that
 \begin{align}
 	b^2 \, \left| \nabla \Delta B_b \right| \leq C \, \left\{ \left| \nabla B_b \right| +
	b \, \left| \nabla \nabla B_b \right| +  \frac{|B_b|}{b}   \right\}  \, .
 \end{align}
 Using this in the Bochner formula \eqr{e:quoteboch} 
 and using the absorbing inequality as before, then allows us to 
  use the meanvalue inequality   to get the desired bound on $b\, \left| \nabla B_b \right|$.
    \end{proof}

\subsection{The proof of property (4)}

As in the previous section, the functional $\cF$ is given by
\begin{align}
	\cF \equiv \frac{1}{2-n} \,  \left( A - \frac{B }{(n-2)}    \right) \, .
\end{align}
 
 The next proposition verifies property (4) for the functional $\cF$.
 
\begin{Pro}	\label{p:prop4}
There exists $C$ so that
\begin{align}
	\left| \nabla_1 \cF (   R^{-2} \, g_R, |\nabla b|) \right|^2 \leq C \, \int_{\frac{R}{2} \leq b \leq \frac{3\,R}{2}}  b^{-n} \,   \left| \Hess_{b^2} - \frac{\Delta b^2}{n} \, g
	   \right|^2 \, .
\end{align}
\end{Pro}

To prove this, we will first give a pointwise bound for $\nabla_1 \cF$ for metrics $g$ that are in a fixed neighborhood of $b_{\infty}^{-2} \, g_0$.

\begin{Lem}	\label{l:boundgF}
If $(g,w)$ is in a sufficiently small neighborhood of $(b_{\infty}^{-2} \, g_0 , b_{\infty})$, then
\begin{align}
	\left| \nabla_1 \cF \right| \leq C \, \sup \left( \left| \Ric_g - (n-2) w^2 \, g \right| + 
	\left| \Hess_w \right| +|\nabla w|   \right) \, .
\end{align}
\end{Lem}

\begin{proof}
Within this proof, we will write $| \cdot |$ for pointwise norms and $\| \cdot \|$ for $L^2$ norms, while $\langle \cdot , \cdot \rangle$ 
will be the $L^2$ inner product.

The   space $\cA_1$ is a level set of   $A_1$, so  
the projection $\nabla_1 \cF$ of the gradient $\nabla \cF$ is  
\begin{align}	\label{e:na1F}
	\nabla_1 \cF = \nabla \cF - \langle \nabla \cF , \nabla A_1 \rangle \, \frac{ \nabla A_1}{\|\nabla A_1\|^2} \, ,
\end{align}
where   Corollary \ref{c:gA1} gives that
\begin{align}
	\nabla A_1 = \left( \frac{1}{2} \, \Psi ({g}) , 1 \right) \, \nu \, .
\end{align}
By 
Corollary \ref{c:gradF}, 
the gradient of $\cF$ at $(g,w)$ is given by
\begin{align}
	(2-n) \, \nabla \cF = \phi_1 \,  \nabla A_1  
	+ \left( \Psi (J) ,  0 \right) \, \nu 
	\, .
\end{align}
Here   $\nu$, $\phi_1$  and $J = J_1 + J_2  $ are given by
\begin{align}
	\nu &= \frac{ w \, \sqrt{ \det (g)}}{b_{\infty} \, \sqrt{\det (b_{\infty}^{-2} \, g_0)}} \, , \\
	\phi_1 & = 3\, w^2  - \frac{R_g}{n-2} \, , \\
	J_1 &= \frac{\Ric_g}{n-2} - w^2 \, g \, , \\
	J_2  &= \frac{1}{n-2} \, \left(  \frac{\Delta w}{w} \, g - \frac{ \Hess_w}{w} \right) 
	 \, .
\end{align}

Since $\Psi$ is a bounded operator, $w$ is bounded above and below, and $\nu$ is bounded, we get the pointwise bound
\begin{align}	\label{e:t2a}
	\left| \left( \Psi (J) ,  0 \right) \, \nu \right| \leq C \,  |J|   \leq C \, \left( \left| \Ric_g - (n-2) w^2 \, g \right| + 
	\left| \Hess_w \right|  
	\right) \,  .
\end{align}
To bound  $\nabla_1 \cF$, we   combine the above with a bound on the projection of $ \phi_1 \,  \nabla A_1  $ given by
\begin{align}		\label{e:t1a}
	 \phi_1 \,  \nabla A_1 - \langle \phi_1 \,  \nabla A_1 , \nabla A_1 \rangle \, \frac{ \nabla A_1}{\|\nabla A_1\|^2} 
	 =  \left( \phi_1  - 
	 \langle \phi_1 \,  \frac{ \nabla A_1}{\|\nabla A_1\|} , \frac{ \nabla A_1}{\|\nabla A_1\|} \rangle   \right) \, \nabla A_1 \, .
\end{align}
However, since $\nabla A_1$ is bounded, we can bound this by
\begin{align}		\label{e:t1a}
	C\,  \left| \phi_1  - 
	 \frac{ \int_N \phi_1 \left| \nabla A_1 \right|^2 }{\int_N |\nabla A_1|^2}     \right|  \leq C \, \left( \sup \phi_1 - \inf \phi_1 \right) 
	 \, .
\end{align}
Using the definition of $\phi_1$, we can bound this by a multiple of the supremum $|\nabla w| + 
\left| \Ric_g - (n-2) w^2 \, g \right|$.
\end{proof}

\begin{proof}[Proof of Proposition \ref{p:prop4}]
Set $g = R^{-2} \, g_R$, where $g_R$ is the induced metric on the level set $b = R$ and set $w = |\nabla b|$, where $\nabla$ is the gradient in $M$; $\nabla^T$ will denote the tangential gradient on
the level set.  We can assume that $g$ is close to $b_{\infty}^{-2} \, g_0$ and $w$ is close to $b_{\infty}$.  

It follows from Lemma
\ref{l:boundgF} that
\begin{align}	\label{e:p4r}
	\left| \nabla_1 \cF \right| \leq C \,  \sup  \left(   \left| \Ric_g - (n-2) w^2 \, g \right| + |\nabla_g w|_g + \left| \Hess_{w,g} \right|_g 
	 \right) \, .
\end{align}
 
To complete the proof, we will show that the right hand side of \eqr{e:p4r}
can be bounded by the scale-invariant $C^1$ norm of the trace-free Hessian $B_b$ of $b^2$ and then appeal to Theorem \ref{t:meanval}. 
The first observation is that at $b=R$
\begin{align}
	|\nabla_g w|_g^2 = R^2 \, \left| \nabla^T w \right|^2 = R^2 \, \left| \nabla^T |\nabla b| \right|^2 = \frac{1}{4} \, \left| (B_b (\nn))^T \right|^2 \, , 
\end{align}	
so we see that  $|\nabla_g w|_g$ is bounded by the $C^0$ norm of trace-free Hessian of $b^2$.  Similarly, differentiating
the equation
$2R \, \nabla^T |\nabla b|  = B_b (\nn)$ shows  that the tangential Hessian of $w$ is bounded by the $C^1$ norm of $B_b$.
Finally, Lemma \ref{l:riccitan} gives the desired bound on
$ \left| \Ric_g - (n-2) w^2 \, g \right| $.

\end{proof}

\subsection{The proof of property (5)}

We will let $g_R$ denote the induced metric on the level set $\{ b = R\}$  in the manifold $M$.  The main result in this section is the following proposition which verifies property (5):

\begin{Pro}	\label{p:p5}
There exists $C$ so that
\begin{align}
	A(R) \leq \cF (  R^{-2} \, g_R , |\nabla b| )  +  C \,
 \int_{\frac{R}{2} \leq b \leq \frac{3\,R}{2}}  b^{-n}\,  \left| \Hess_{b^2} - \frac{\Delta b^2}{n} \, g
	   \right|^2    
	 \,   .
\end{align}
\end{Pro}

The next lemma expresses 
$\cF (  R^{-2} \, g_R , |\nabla b| )  $ in terms of $A(R)$ and an integral that vanishes when $B_b$ is zero.
This must be  since $\cF$ and $A$ agree on  cones.  To prove the proposition, we must 
show that the error terms  either have the right sign or
are at least quadratic in $B_b$.

\begin{Lem}	\label{l:cFonR}
We can write $\cF (R^{-2} \, g_R , |\nabla b|)$ as
\begin{align}
	 A(R) + \frac{R^{1-n}}{n-2} \int_{b=R}  \left\{ -     \, B_b(\nn , \nn)   + 
	  \frac{ 2\, \left| B_b(\nn)  \right|^2 - \left| B_b \right|^2 }{4(n-2) \, |\nabla b|^2}
	     \right\} |\nabla b|  \, .
\end{align}
\end{Lem}

\begin{proof}
   We have 
\begin{align}
	A(R) = R^{1-n} \,  \int_{b=R} |\nabla b|^3 \, .
\end{align}
On the other hand, we have
\begin{align}
	\cF (R^{-2} \, g_R , |\nabla b|) =
	 \frac{1}{n-2} \, R^{1-n} \,  \int_{b=R} \left\{ \frac{ R^2 \, R_R }{(n-2)} -|\nabla b|^2    \right\} \, |\nabla b| \, , 
\end{align}
where   the scalar curvature $R_R$ of the level set 
is given by Lemma \ref{l:R1}
   \begin{align}
 	  b^2 \, |\nabla b|^2 \, R_R &  = 
	  (n-1)(n-2) \, |\nabla b|^4 -  (n-2) |\nabla b|^2 \, B_b(\nn , \nn)   - 
	\frac{1}{4} \,  \left| B_b \right|^2 
	  +  \frac{1}{2} \, \left| B_b(\nn)  \right|^2  
	 \, .
 \end{align}
 We see that at $b=R$
  \begin{align}
 	 \frac{ R^2 \, R_R }{(n-2)} -|\nabla b|^2 &=  (n-2) \, |\nabla b|^2 -     \, B_b(\nn , \nn)   + 
	  \frac{ 2\, \left| B_b(\nn)  \right|^2 - \left| B_b \right|^2 }{4(n-2) \, |\nabla b|^2} 
	 \, .
 \end{align}
After dividing by $(n-2)$ the first term on the right gives us $A(R)$, giving the lemma.
 
\end{proof}

\begin{proof}[Proof of Proposition \ref{p:p5}]
Using Lemma
\ref{l:cFonR}, we can write 
$\cF (R^{-2} \, g_R , |\nabla b|)$
as
\begin{align}
	 A(R) + \frac{R^{1-n}}{n-2} \int_{b=R}  \left\{ -     \, B_b(\nn , \nn)   + 
	  \frac{ 2\, \left| B_b(\nn)  \right|^2 - \left| B_b \right|^2 }{4(n-2) \, |\nabla b|^2}
	      \right\} |\nabla b|  \, .
\end{align}
Since $  |\nabla b| \, B_b(\nn , \nn) = b\, \langle \nabla |\nabla b|^2 , \nn \rangle$, 
we see that
\begin{align}
	R^{1-n} \, \int_{b=R}  \left\{ -     \, B_b(\nn , \nn)     \right\} |\nabla b| = -
	 R^{2-n} \, \int_{b=R} \langle \nabla |\nabla b|^2 , \nn \rangle = -  R\,  A'(R) \geq 0 \, .
\end{align}
Substituting this back into \eqr{e:fromcF} and throwing away the (only helpful) $\left| B_b (\nn) \right|^2$ term gives
\begin{align}	\label{e:fromcF}
	\cF (  R^{-2} \, g_R , |\nabla b|)- A(R) &\geq - \frac{R^{1-n}}{n-2} \int_{b=R}  \left\{ 
	  \frac{  \left| B_b \right|^2 }{4(n-2) \, |\nabla b|^2}
	   \right\} |\nabla b|   \, .
\end{align}
We conclude that
\begin{align}
	A(R) \leq \cF (R^{-2} \, g_R , |\nabla b| )  + C \, R^{1-n} \, \int_{b=R} \left| \Hess_{b^2} - \frac{\Delta b^2}{n-1} \, g \right|^2 
	 \,   .
\end{align}
Finally, the proposition follows by using Theorem \ref{t:meanval} to estimate the last term.
\end{proof}

\section{Second variation of $\cF$ and the linearization of the gradient of $  \cF$}

The rest of the paper will be devoted to proving the  Lojasiewicz-Simon inequality (3) for $\cF$.   We will need to understand the linearization $L_{\cF}$ of the gradient $\nabla_1 \cF$ of 
the functional $\cF$ restricted to $\cA_1$.  This is  equivalent to understanding the second variation of $\cF$.  
The operator $L_{\cF}$ will behave quite differently on different subspaces of variations, just as for the second variation of the classical Einstein-Hilbert scalar curvature functional.

\vskip2mm
  Throughout this section, we will assume that
 \begin{align}
 	(b_{\infty}^{-2} \, g_0 + t \, h, b_{\infty} \, \e^{tv_t}) \in \cA_1
\end{align}
 is a  variation.
As in the previous section, $g_0$ is an Einstein metric with $\Ric_{g_0} = (n-2) \, g_0$ and $b_{\infty}$ is a positive constant.
Where it is clear, we will omit the subscript $t$ from $g$ and $v$.

We will first compute the second variations of $A$ and $B$  and then combine these to get
 the second variation of $\cF$ on two important subspaces.  Roughly speaking, this will determine the two on-diagonal blocks of $L_{\cF}$.  In the last subsection, we will show that the remaining (off-diagonal) blocks of $L_{\cF}$ vanish.

\subsection{The second variation of $A$}

\begin{Lem}	\label{l:A2}
The second variation $A''  = \frac{d^2}{dt^2} \big|_{t=0} \, A(b_{\infty}^{-2} \, g_0 + th , b_{\infty} \, \e^{tv_t})$ is
\begin{align}
	  b_{\infty}^3 \, \int_N \left\{ 4 \,v \,  \left( \frac{1}{2} \, \Tr (h) +2 v  \right) +
	   \left( \frac{1}{2} \, \Tr (h) +v  \right)^2  + 6   v' - \frac{|h|^2}{2} + \frac{\Tr (h')}{2}
	 \right\} \,    d\mu_{ b_{\infty}^{-2} \, g_0}  
	\, .
\end{align}
\end{Lem}

\begin{proof}
To simplify notation, set $\bar{g} = b_{\infty}^{-2} \, g_0 + t h$.
Proposition \ref{p:fvar1} gives 
\begin{align}
	A'  (t)& = b_{\infty}^3 \, \int_N \left\{   \e^{2tv} \, \left( \frac{1}{2} \, \Tr (h) +v +t v' \right)  + 2   (v+t v') \, \e^{2tv}
	 \right\} \, \e^{tv} \, d\mu_{\bar{g }}
	\, .
\end{align}
At $t=0$, the term in curly brackets becomes
\begin{align}
	\left( \frac{1}{2} \, \Tr (h) +v  \right)  + 2   v \, .
\end{align}
Since we also have
\begin{align}
	 \left( \Tr (h) \right)' = \left( \bar{g}^{ij} h_{ij} \right)' = \Tr (h') - |h|^2 \, , 
\end{align}
differentiating $A$ a second time at $t=0$ gives
\begin{align}
	\frac{A''}{b_{\infty}^3} & =    \int_N \left\{   \left( \frac{1}{2} \, \Tr (h) +v  \right)^2  + 4   v    \,
	  \left( \frac{1}{2} \, \Tr (h) +v  \right) + \frac{ \Tr (h') - |h|^2}{2}     + 4v^2 + 6 v' 
	 \right\} \,    d\mu_{\bar{g }_0}   
	\, .
\end{align}
\end{proof}

\subsection{The second variation of $B$}

\begin{Lem}	\label{e:B2}
The second variation $B''  = \frac{d^2}{dt^2} \big|_{t=0} \, B(b_{\infty}^{-2} \, g_0 + th , b_{\infty} \, \e^{tv_t})$ is
\begin{align}
	B'' &= b_{\infty} \, \int_N \left\{ b_{\infty}^2 \,  (n-2) \, \left[ (n-1)  \, \left( \frac{1}{2} \, \Tr (h) + v \right) - 2\,  \Tr (h) \right] \,
	\left( \frac{1}{2} \, \Tr (h) + v \right)
	  \right.  \notag \\
	& \qquad  - \langle \nabla (\delta h) , h \rangle   
	 + \frac{1}{2} \langle \Delta h , h \rangle + \frac{1}{2} \langle \Hess_{\Tr h} , h \rangle +
	 R_{ikj\ell} h_{k\ell} h^{ij}  \notag \\
	 & \qquad +  \langle h , \Hess_v \rangle  -   \Tr (h) \, \Delta v
	 + \left(  \delta^2 h   - \Delta \Tr (h)
	\right)  \, \left( \frac{1}{2} \, \Tr (h) + v \right) \notag  \\
	&\left. \qquad + b_{\infty}^2 \, (n-2) \, \left( \frac{n-3}{2} \, (  \Tr (h') - |h|^2) +2\, (n-1)\, v'  \right)  \right\} \, d\mu_{b_{\infty}^{-2} \, g_0}
	\, .
\end{align}
\end{Lem}

\begin{proof}
To simplify notation, set $\bar{g} = b_{\infty}^{-2} \, g_0 + th$.
Proposition \ref{p:fvar1} gives that $ \frac{B' (t)}{b_{\infty}}$ is
\begin{align}	\label{e:4brs}
		  \int_N \left\{   - \langle \Ric_{\bar{g}} , h \rangle +
		  \langle h , \frac{ \Hess_{\e^{tv}}}{\e^{tv}} \rangle   - 
		 \Tr (h)\frac{\Delta \e^{tv}}{\e^{tv}}  + 
	R_{\bar{g}} \, \left( \frac{\Tr (h) }{2} + v + t v' \right) \right\} \, \e^{tv} \, d\mu_{\bar{g}} \,  	\, .
\end{align}
At $t=0$, $\Ric_{\bar{g}_0} = b_{\infty}^2 \, (n-2) \, \bar{g}_0$ and the term in curly brackets is equal to
\begin{align}
	- b_{\infty}^2 \,  (n-2) \, \Tr (h)   + b_{\infty}^2 \, 
	(n-1) \, (n-2) \, \left( \frac{1}{2} \, \Tr (h) + v \right) \, .
\end{align}
Using Lemma \ref{l:topping} and
$\Ric_{\bar{g}_0} = b_{\infty}^2 \, (n-2) \, \bar{g}_0$,   we get at $t=0$:
 \begin{align}
 	\left( \bar{g}^{ij} \right)' &= - h^{ij}  \, , \\
	R_{\bar{g}}' &=  \delta^2 \, h  - \langle \Ric_{\bar{g}_0} , h \rangle   - \Delta \, \Tr (h) = \delta^2 \, h - b_{\infty}^2 \, (n-2) \, \Tr  (h) - \Delta \, \Tr (h)
	  \, , \\
	 	\Ric_{ij} ' &= \frac{1}{2} \, 
	\left( \nabla_i  (\delta \, h)_j +   \nabla_j   (\delta \, h)_{i} +
	\Ric_{ik} h_{jk} + \Ric_{jk} h_{ik} - \Delta h_{ij} 
	- \Hess_{\Tr \, h}
	\right) - R_{ikj\ell} h_{k\ell} \notag \\
	&
	 = \frac{1}{2} \, 
	\left( \nabla_i  (\delta \, h)_j +   \nabla_j   (\delta \, h)_{i} +
	2\, b_{\infty}^2 \, (n-2) \, h_{ij}  - \Delta h_{ij}
	- \Hess_{\Tr \, h}
	\right) - R_{ikj\ell} h_{k\ell}
	 \, , \\
	\left( \Hess_{\e^{tv}} \right)_{ij}' &= \Hess_v - \frac{1}{2} \, 
	\left(\nabla_i \left( \Hess_{\e^{tv}} \right)_{jk} + 
	\nabla_j \left( \Hess_{\e^{tv}} \right)_{ik} - \nabla_k \, \left( \Hess_{\e^{tv}} \right)_{ij}
	\right) \, \nabla_k {\e^{tv}} \notag  \\
	&=\Hess_v \, .
 \end{align}
 (In the formula for $\Ric'$, we work in an othonormal frame and ignore the difference between upper and lower indices after differentiating.)
 
 We also need the formula for $(\Delta \e^{tv} )'$.  This follows from the first and last formulas above since
 $\Delta w =\bar{ g}^{ij} \, (\Hess_w)_{ij}$ so that
 \begin{align}
 	(\Delta {\e^{tv}})' &=  \bar{g}_0^{ij} \, \left( \Hess_{\e^{tv}} \right)_{ij}' 
	= \Delta v \, .
 \end{align}

We will differentiate the four terms in curly brackets in \eqr{e:4brs}  at $t=0$.  The first is 
\begin{align}
	 \langle \Ric_{\bar{g}} , h \rangle' &=  \langle \Ric_{\bar{g}} ', h \rangle +  \langle \Ric_{\bar{g}} , h' \rangle 
	 - R_{ij} h_{k\ell} h^{ik} \bar{g}^{j\ell} -  R_{ij} h_{k\ell} \bar{g}^{ik} h^{j\ell}  \notag \\
	 &= \langle \Ric_{\bar{g}} ', h \rangle + b_{\infty}^2 \,  (n-2) \, \Tr (h') - 2b_{\infty}^2 \,  (n-2) \, |h|^2 \\
	 &= b_{\infty}^2 \,  (n-2) \, \Tr (h') - 2b_{\infty}^2 \,  (n-2) \, |h|^2 + \langle \nabla (\delta h) , h \rangle + b_{\infty}^2 \,  (n-2) \, |h|^2 \notag \\
	 &\qquad 
	 - \frac{1}{2} \langle \Delta h , h \rangle - \frac{1}{2} \langle \Hess_{\Tr h} , h \rangle -
	 R_{ikj\ell} h_{k\ell} h^{ij}
	 \notag \, .
\end{align}
Simplifying this gives
\begin{align}
	 \langle \Ric_{\bar{g}} , h \rangle' & 
	 = b_{\infty}^2 \,  (n-2) \, \left[ \Tr (h') -   |h|^2 \right]  + \langle \nabla (\delta h) , h \rangle   \notag \\
	 &\qquad 
	 - \frac{1}{2} \langle \Delta h , h \rangle - \frac{1}{2} \langle \Hess_{\Tr h} , h \rangle -
	 R_{ikj\ell} h_{k\ell} h^{ij}
	  \, .
\end{align}
Since $\Hess_{\e^{tv}}$ vanishes at $t=0$, 
differentiating the second term gives
\begin{align}
	\left(  \langle h , \frac{ \Hess_{\e^{tv}}}{\e^{tv}} \rangle \right)' &= 
  \langle h ,  \Hess_{\e^{tv}} ' \rangle = \langle h , \Hess_v \rangle \, .
 \end{align}
 Similarly, the third term is
 \begin{align}
		\left(   \Tr (h)\frac{\Delta \e^{tv}}{\e^{tv}} \right)' &=  \Tr (h) \, \left(\Delta \e^{tv} \right)' 
		= \Tr (h) \, \Delta v \, .
\end{align}
Finally, the last term is
\begin{align}
	\left( R_{\bar{g}} \, \left( \frac{1}{2} \, \Tr (h) + v + t v' \right) \right)' &= R_{\bar{g}}' \, \left( \frac{1}{2} \, \Tr (h) + v \right) 
	+b_{\infty}^2 \,  (n-1)(n-2) \, \left( \frac{1}{2} \, \Tr (h) + v +t v' \right)' \notag \\
	&=  \left\{  \delta^2 h - b_{\infty}^2 \, (n-2) \, \Tr (h) - \Delta \Tr (h)
	\right\}  \, \left( \frac{1}{2} \, \Tr (h) + v \right)  \\
	&\qquad + b_{\infty}^2 \, (n-1) (n-2)\, \left( \frac{1}{2} \, (  \Tr (h') - |h|^2) + 2 v'  \right)  \, . \notag
\end{align}
Combining all of this gives
\begin{align}
	B'' &= b_{\infty} \, \int_N  \left\{ b_{\infty}^2 \, (n-2)\, \left[  (n-1)  \left( \frac{1}{2} \, \Tr (h) + v \right) -   \Tr (h) \right] \,
	\left( \frac{1}{2} \, \Tr (h) + v \right)
	  \right.  \notag \\
	&- b_{\infty}^2 \, 
	(n-2) \, \left( \Tr (h') -  |h|^2 \right)  - \langle \nabla (\delta h) , h \rangle   
	 + \frac{1}{2} \langle \Delta h , h \rangle + \frac{1}{2} \langle \Hess_{\Tr h} , h \rangle +
	 R_{ikj\ell} h_{k\ell} h^{ij}  \notag \\
	 &+  \langle h , \Hess_v \rangle  -   \Tr (h) \, \Delta v
	 + \left(  \delta^2 h - b_{\infty}^2 \,  (n-2) \, \Tr (h) - \Delta \Tr (h)
	\right)  \, \left( \frac{1}{2} \, \Tr (h) + v \right) \notag  \\
	&\left. \qquad + b_{\infty}^2 \,  (n-1)(n-2) \, \left( \frac{1}{2} \, (  \Tr (h') - |h|^2) + 2v'  \right)  \right\} \, d\mu_{\bar{g}}
	\, .
\end{align}
Simplifying this completes the proof.

\end{proof}

\subsection{The constraint on the variation}

Since the variation $ (b_{\infty}^{-2} \, g_t , b_{\infty} \, \e^{tv_t})$  is in
$ \cA_1$, there are constraints on $h, h', v$ and $v'$.  The next lemma records this.

\begin{Lem}	\label{l:constraint}
At $t=0$, we have that
\begin{align}
	\int_N \left\{ \frac{1}{2} \, \Tr (h) +v
	\right\} \, d\mu_{b_{\infty}^{-2} g_0} &= 0 \, , \\
	\int_N  \left\{ \left(\frac{1}{2} \, \Tr (h) +v \right)^2
	 + \frac{1}{2} \,   \Tr (h')  -  \frac{1}{2} |h|^2  + 2v'   
	\right\} \, d\mu_{b_{\infty}^{-2} g_0} &= 0 \, .
\end{align}
\end{Lem}

\begin{proof}
Set $\bar{g} = b_{\infty}^{-2} \, g_0 + th$.  
Since the path $(\bar{g} , b_{\infty} \, \e^{tv_t})$ is contained in $\cA_1$, the integral
 \begin{align}
 	A_1 (t) \equiv \int_N  \e^{tv} \, d\mu_{\bar{g}} 
\end{align}
is constant in $t$.  Differentiating this gives
\begin{align}	 
	0 = A_1  '(t) &=  \int_N \left( \frac{1}{2} \, \bar{g}^{ij} \, h_{ij}   + v + t v' \right) \, \e^{tv} \, d\mu_{\bar{g}} \, .
\end{align}
This gives the first claim.  Differentiating $A_1$ a second time at $t=0$ gives
\begin{align}
	0 = A_1 ''(0) = \int_N \left\{  \left(\frac{1}{2} \, \Tr (h) +v \right)^2
	 + \frac{1}{2} \,   \Tr (h')  -  \frac{1}{2} h^{ij} \, h_{ij} + 2 v'    
  \right\}   \, d\mu_{\bar{g}} \, .
\end{align}
\end{proof}

\subsection{The transverse trace-less second variation}

The functional $\cF$ is given by
\begin{align}
	\cF \equiv \frac{1}{2-n} \,  \left(  A - \frac{B }{(n-2)}   \right)  \, .
\end{align} 
 Since we have computed the second variations of $A$ and $B$, we get $\cF''$ as a consequence.  It is useful to divide this  into two cases, depending on the     variation $h$ of the metric.  In this subsection, we will consider the   case where $h$ is ``transverse-traceless'', i.e., 
when   
\begin{align}	\label{e:tt}
	\delta h = 0 {\text{ and }}  \Tr \, h = 0 \, .
\end{align}
The next proposition computes the second variation for transverse trace-less variations.{\footnote{When we apply this later, we will have $v=0$.}}

\begin{Pro}	\label{p:sv1}
If $h$ satisfies \eqr{e:tt}, then 
the second variation is  
\begin{align}
	(2-n) \, \cF''  & = - b_{\infty}  \, \int_N \left\{ \frac{1}{2(n-2)} \langle \cL \, h , h \rangle
	-6  \, b_{\infty}^2 \, v^2 \right\} \, d \mu_{b_{\infty}^{-2} {g}_0} \, , 
\end{align}
where $\cL$ is the Lichnerowicz operator
\begin{align}
	\left( \cL \, h \right)_{ij} = \left( \Delta \, h\right)_{ij}  + 2\, 
	  R_{ikj\ell}  h_{k\ell}  \, . 
\end{align}
\end{Pro}

\begin{proof}
Set $\bar{g} = b_{\infty}^{-2} \, g_0 + th$.   Since $\Tr (h) = 0$, 
Lemma \ref{l:A2}
gives  
\begin{align}
	A'' & = b_{\infty}^3 \, \int_N \left\{  
	   9 \,    v^2  + 6   v' - \frac{|h|^2}{2} + \frac{\Tr (h')}{2}
	 \right\} \,    d\mu_{\bar{g}_0}
	\, .
\end{align}
Since $\Tr (h) = 0$ and $\delta h = 0$, Lemma \ref{e:B2} gives
\begin{align}
	 B''  &= b_{\infty} \, \int_N \left\{ b_{\infty}^2 \, (n-1)(n-2) \, v^2  + \frac{1}{2} \langle \Delta h , h
	\rangle + R_{ikj\ell} h_{k\ell} h^{ij} \right. \notag \\
	&\left.   
	+ b_{\infty}^2 \, (n-2) \, \left( \frac{n-3}{2} \, (  \Tr (h') - |h|^2) + 2\, (n-1)\, v'  \right)
	\right\} d\mu_{\bar{g}_0} \, , 
\end{align}
where we have also used that $ \int \langle h , \Hess_v \rangle = - \int \langle \delta h , \nabla v \rangle = 0 $.
Combining the two formulas gives that
\begin{align}
	(2-n) \, \cF''  &=  A'' - \frac{B''}{(n-2)}      \notag \\
	&= b_{\infty}  \, \int_N \left\{ b_{\infty}^2 \,  (10-n)\,    v^2  +  b_{\infty}^2 \, (4-n)\,  \left[  2\, v' - \frac{ |h|^2}{2} + \frac{  \Tr (h')}{2}
	\right]
	      - \frac{\langle \Delta h , h
	\rangle}{2(n-2)}  \right. \notag \\
	& \left. - \frac{ R_{ikj\ell} }{n-1} \, h_{k\ell} h^{ij}    
	   \right\} \, d\mu_{\bar{g}_0} \, . 
\end{align}
We want to eliminate the $v'$ and $h'$ terms.  Lemma \ref{l:constraint}
gives that
\begin{align}
	\int_N \left\{  2v' - \frac{|h|^2}{2} + \frac{
	  \Tr (h')}{2}    
	\right\} \, d\mu_{\bar{g}_0} &= -  \int_N \left\{   v^2 
	\right\} \, d\mu_{\bar{g}_0}  \, .
\end{align}
Substituting this gives
\begin{align}
	(2-n) \, \cF''  & = b_{\infty}  \, \int_N \left\{  6 \,   b_{\infty}^2 \,  v^2  	      - \frac{\langle \Delta h , h
	\rangle}{2(n-2)}   - \frac{ R_{ikj\ell} }{n-2} \, h_{k\ell} h^{ij}    
	  \right\} \, d\mu_{\bar{g}_0} \, . 
\end{align}
\end{proof}

\subsection{The conformal second variation}

We suppose next that 
\begin{align}	\label{e:conf}
	h = \phi \, b_{\infty}^{-2} \, g_0 
\end{align}
 at $t=0$ for a function $\phi$, so that
\begin{align}
	\Tr \, h &= (n-1) \, \phi \, , \\
	(\delta \, h) &=  \nabla \phi \, , \\
	\nabla \delta h &= \Hess_{\phi} \, , \\
	\delta^2 \, h &= \Delta \phi \, , \\
	\Delta h &= (\Delta \phi) \, b_{\infty}^{-2} \, g_0 \, .
\end{align}

\begin{Thm}	\label{t:conformal}
If $h$ satisfies \eqr{e:conf}, then the second variation is
\begin{align}
	(2-n) \, \cF''  & = b_{\infty}  \, \int_N \left\{  
	 \frac{n-3}{2} \, \phi \, \left[ \Delta \phi + (n-1)b_{\infty}^2 \phi \right]
	 + 2(n-1)b_{\infty}^2 \phi v 
	 +   \phi \, \Delta v + v \Delta \phi \right. \notag \\
	 &\qquad  \qquad \left.     + 6b_{\infty}^2 \, v^2   
	 \right\} \, d\mu_{ b_{\infty}^{-2} \, g_0}
	 \, . 
\end{align}

\end{Thm}

\begin{proof}
To simplify notation, set $\psi =  \left( \frac{n-1}{2} \, \phi + v   \right)$ and $\bar{g} = b_{\infty}^{-2} \, g_0$.
Lemma \ref{l:A2}
gives  
\begin{align}
	A'' & = b_{\infty}^3 \, \int_N \left\{ 4 \,v \,  \left( \psi + v  \right) +
	  \psi^2  + 6   v' - \frac{|h|^2}{2} + \frac{\Tr (h')}{2}
	 \right\} \,    d\mu_{\bar{g}}  
	\, .
\end{align}
Lemma \ref{e:B2} gives
\begin{align}
	B'' &= b_{\infty} \, \int_N \left\{ b_{\infty}^2 \, (n-2) \, (n-1)\left[ \psi - 2\,  \phi \right] \,
	\psi
	   - \phi \, \Delta \phi   
	 + (n-1)   \, \phi \, \Delta \phi +
	\phi^2 \,  R_{ikj\ell} g_{k\ell} g^{ij} \right.  \notag \\
	 &+  \phi \, \Delta v -  (n-1) \, \phi \, \Delta v
	 + \left(  \Delta \phi   - (n-1) \, \Delta \phi
	\right)  \,\psi \notag  \\
	&\left. \qquad + b_{\infty}^2 \, (n-2) \, \left( \frac{n-3}{2} \, (  \Tr (h') - |h|^2) + 2(n-1)\, v'  \right)  \right\} \, d\mu_{\bar{g}} 
	\, .
\end{align}
Collecting terms, this becomes
\begin{align}
	B'' &= b_{\infty} \, (n-2) \, \int_N \left\{ b_{\infty}^2 \,  (n-1) \left[  \psi^2
	-2    \, \phi \, \psi \right] 
	    +     \phi \, \Delta \phi   +  b_{\infty}^2 \,( n-1)  \, 
	\phi^2  -   \phi \, \Delta v - \psi \, \Delta \phi  \right.
	   \notag  \\
	&\left. \qquad +   b_{\infty}^2 \, \left( \frac{n-3}{2} \, (  \Tr (h') - |h|^2) +2 (n-1)\, v'  \right)  \right\} \, d\mu_{\bar{g}} 
	\, .
\end{align}
Combining the two formulas gives that
\begin{align}
	(2-n) \, \cF''  &=  A'' - \frac{B'' }{(n-2)}    \notag \\
	&= b_{\infty}  \, \int_N \left\{  -  \phi \, \Delta \phi     +   \phi \, \Delta v + \psi \Delta \phi    \right. \\
	&\qquad \left. + b_{\infty}^2 \, \left[
	 4v^2 + (6-n)\, 
	   \psi^2  + (4 -n) \left[ 2 v' - \frac{|h|^2}{2} + \frac{\Tr (h')}{2}\right]    - (n-1)  \, \phi^2 \right]
	 \right\} \, d\mu_{\bar{g}}  \, , \notag
\end{align}
where the last equality also used that
\begin{align}
	4 v \, \psi + 2(n-1) \phi \, \psi = 4 \psi^2 \, .
\end{align}
We want to eliminate the $v'$ and $h'$ terms.  Lemma \ref{l:constraint}
gives that
\begin{align}
	\int_N \left\{   \frac{1}{2} \,   \Tr (h')  -  \frac{1}{2} |h|^2_g + 2 v'   
	\right\} \, d\mu_{\bar{g}}  &= - \int_N \left\{ \psi^2
	\right\} \, d\mu_{\bar{g}}  \, .
\end{align}
Putting this in gives
\begin{align}
	(2-n) \, \cF''  &= b_{\infty}  \, \int_N \left\{  -  \phi \, \Delta \phi     +   \phi \, \Delta v + \psi \Delta \phi   
	 + b_{\infty}^2 \, \left[
	 4v^2 + 2
	   \psi^2    - (n-1)  \, \phi^2 \right]
	 \right\} \, d\mu_{\bar{g}}  \, . \notag
\end{align}
Since $\psi =  \left( \frac{n-1}{2} \, \phi + v \right)$, we have
\begin{align}
	2
	   \psi^2  + 4 v^2  -	   (n-1)  \, 
	\phi^2   &= 2v^2 + 2(n-1) \phi v + \frac{(n-1)^2}{2} \phi^2 
	   + 4 v^2  -	  ( n -1) \, 
	\phi^2   \notag \\
	&= 6v^2  + \frac{(n-1)(n-3)}{2} \, \phi^2 + 2(n-1) \phi v \,  , \\
	-  \phi \, \Delta \phi   +   \phi \, \Delta v + \psi \Delta \phi &= 
	\frac{n-3}{2} \,   \phi \, \Delta \phi   +   \phi \, \Delta v + v \Delta \phi \, .
\end{align}
Substituting these two equations back in gives the claim.

\end{proof}

\subsection{The gradient of $\cF$ in the conformal directions}

The next proposition shows that the linearization of   $\nabla \cF$ maps conformal variations onto the span of conformal variations together with variations tangent to the action of the diffeomorphism group.

\begin{Pro}	\label{p:lingradF}
The first variation of $\nabla \cF$ along the path $(b_{\infty}^{-2} g_t, b_{\infty} \, \e^{tv_t})$ where
$b_{\infty}^{-2} g_t' = \phi \, b_{\infty}^{-2} \, g_0$ and $v_t' = v'$ can be written as 
\begin{align}
	 ( \nabla \cF)' =  \left( f_1 \, g_0 , f_2 \right) + \left( \Hess_{f_3} , 0 \right)  
	\, ,
\end{align}
where $f_1$, $f_2$ and $f_3$ are functions.
\end{Pro}

\begin{proof}
Set $\bar{g}_t = b_{\infty}^{-2} \, g_t$; we omit the subscript when the meaning is clear.
Corollary \ref{c:gradF} gives 
\begin{align}
	(2-n) \, ( \nabla \cF) = \phi_1 \,  \left(   \frac{1}{2}  \,   \Psi (\bar{g}) , 1\right) \, \nu + \left( \Psi (J) ,  0 \right) \, \nu 
	\, .
\end{align}
At $t=0$, we know that
\begin{align}
	\nu = 1 , \, J = 0 , \, \Psi {\text{ is the identity}},   {\text{ and }} \phi_1 = (4-n) \, b_{\infty}^2 \, .
\end{align}
Lemma \ref{l:gradadjust}
gives that if $\bar{J}$ is a family of $2$-tensors depending on $t$, then 
\begin{align}
	\frac{d}{dt} \, \big|_{t=0} \, \Psi (\bar{J})_{ij}  = \bar{J}_{ij}' - \bar{g}_{ip}' \,  \bar{g}^{pn}  \, \bar{J}_{nj}  
	-    \bar{J}_{im}   \bar{g}^{mp} \, \bar{g}_{pj}'  \, .
\end{align}
Using this, we see that
\begin{align}
	[\Psi (\bar{g})]'  &= - \bar{g}'   \, , \\
	[\Psi (J)]' &= J'
	\, .
\end{align}
Thus, we see that at $t=0$ we have
\begin{align}	\label{e:463}
	(2-n) \, ( \nabla \cF)' &=  \frac{(n-4)}{2} \, b_{\infty}^2\,  \left(    \bar{g}' , 0 \right) + 
	  \left(   \frac{1}{2}  \,  \bar{g} , 1\right) \, \left[ (4-n) \, b_{\infty}^2 \,  \nu'+ \phi_1' \right]  + \left( J' ,  0 \right)  
	\, .
\end{align}
Next, we bring in the conformal nature of the variation in order to compute $\nu'$, $J'$,  and $\phi_1'$.  
If we write the metric $\bar{g}_t$ as
\begin{align}
	\bar{g}_t = b_{\infty}^{-2} \, \e^{t \phi} \, g_0 \, , 
\end{align}
then we have at $t=0$ that $\bar{g}_0 = b_{\infty}^{-2} \, g_0$ and $\bar{g}' = \phi \, \bar{g}_0$.  
Using this variation in the formulas for $\nu$, $\phi_1$,   and $J$ from Corollary \ref{c:gradF}
  gives
\begin{align}
	\nu &=  \e^{t \, \left( v+  \frac{(n-1)}{2} \, \phi \right)}  \, , \\
	\phi_1 & = 3\,b_{\infty}^2\,  \e^{2tv}  - \frac{R_{\bar{g}_t}}{n-2} \,   , 
\end{align}
and   the $2$-tensor $J = J_1 + J_2 $ is given by
\begin{align}
	J_1 &= \frac{      \Ric_{\bar{g}_t}   }{n-2} - b_{\infty}^2\,  \e^{2tv} \, \bar{g} \, , \\
	J_2 &= \frac{1}{n-2} \, \left(  
		\frac{\Delta \e^{tv}   }  {\e^{tv}} \,\bar{ g} 
		- \frac{ \Hess_{\e^{tv}}}   {\e^{tv}} 
		\right) \,   .
\end{align}

Using Lemma \ref{l:topping} and
$\Ric_{\bar{g}} = b_{\infty}^2 \, (n-2) \, \bar{g}$ and working in an orthonormal frame (so we do not distinguish upper and lower indices),   we get at $t=0$:
 \begin{align}
 	R_{\bar{g}}' &=  \delta^2 \, \bar{g}'  - \langle \Ric_{\bar{g}_0} , \bar{g}'    \rangle   - \Delta \, \Tr (\bar{g}') = (2-n) \, \left\{
	\Delta \phi +
	b_{\infty}^2 \,   (n-1) \, \phi 	\right\}  \, , \\
	 	\Ric_{ij} ' &= \frac{1}{2} \, 
	\left( \nabla_i  (\delta \, \bar{g}')_j +   \nabla_j   (\delta \, \bar{g}')_{i} +
	\Ric_{ik} \bar{g}'_{jk} + \Ric_{jk} \bar{g}'_{ik} - \Delta \bar{g}'_{ij} 
	- \Hess_{\Tr \, \bar{g}'}
	\right) - R_{ikj\ell} \bar{g}'_{k\ell} \notag \\
	&
	 =  \Hess_{\phi}  + b_{\infty}^2 \, (n-2) \, \phi \, \bar{g}
	 -
	 \frac{1}{2} \, 
	\left\{   ( \Delta \phi) \bar{g}
	+ (n-1) \, \Hess_{\phi}
	\right\}  - b_{\infty}^2 \, (n-2) \, \phi \, \bar{g}  \notag \\
	&=   \frac{1}{2} \, 
	\left\{  (3-n) \,  \Hess_{\phi} -  ( \Delta \phi) \bar{g}
	\right\}
	 \, , \\
	\left( \Hess_{\e^{tv}} \right)_{ij}' &= \Hess_v - \frac{1}{2} \, 
	\left(\nabla_i \left( \Hess_{\e^{tv}} \right)_{jk} + 
	\nabla_j \left( \Hess_{\e^{tv}} \right)_{ik} - \nabla_k \, \left( \Hess_{\e^{tv}} \right)_{ij}
	\right) \, \nabla_k {\e^{tv}} \notag  \\
	&=\Hess_v \, .
 \end{align}
By the last formula    and the general formula
 $\Delta u =\bar{ g}^{ij} \, (\Hess_u)_{ij}$, we get
 \begin{align}
 	(\Delta {\e^{tv}})' &=  \bar{g}_0^{ij} \, \left( \Hess_{\e^{tv}} \right)_{ij}' 
	= \Delta v   \, .
 \end{align}

 Using these formulas for the derivatives   in the definitions of $\phi_1$ and $J$, we   compute
\begin{align}
	\phi_1' & = 6 \,b_{\infty}^2\,  v  + \Delta \phi +
	b_{\infty}^2 \,   (n-1) \, \phi \, , \\
			J' &= \frac{      (3-n) \,  \Hess_{\phi} -  ( \Delta \phi) \bar{g}_0  }{2(n-2)} -   b_{\infty}^2\,  (2v+\phi) \, \bar{g}_0
		+  \frac{1}{n-2} \, \left(  
		 \Delta v   \,\bar{ g}_0
		-   \Hess_{v}   
		\right)  \notag \\
		&= \frac{      (3-n)   }{2(n-2)} \,  \Hess_{\phi} - \frac{\Hess_v}{n-2} + 
		\left( \frac{\Delta v}{n-2} - \frac{\Delta \phi}{2(n-2)} -   b_{\infty}^2 \, (2v+\phi) \right) \, \bar{g}_0 \, .
\end{align}
Finally, substituting these in \eqr{e:463} gives
\begin{align}	\label{e:esl}
	(2-n) \, ( \nabla \cF)' &= 
	\left[ (4-n) \, b_{\infty}^2 \,  \left[ v + \frac{n-1}{2} \, \phi \right]+ 6 \,b_{\infty}^2\,  v  + \Delta \phi +
	b_{\infty}^2 \,   (n-1) \, \phi  \right] \,   \left(   \frac{1}{2}  \,  \bar{g}_0 , 1\right)  \notag \\
	&\qquad+
	\left( \frac{      (3-n)   }{2(n-2)} \,  \Hess_{\phi} - \frac{\Hess_v}{n-2} , 	0
  \right)  \\
	&\qquad +
	\left[ \frac{\Delta v}{n-2} - \frac{\Delta \phi}{2(n-2)} -  b_{\infty}^2 \, (2v+\phi)
		+  \frac{(n-4)}{2} \, b_{\infty}^2\,    \phi   \right] \, 
	\left(	 \bar{g}_0
	,  0
	 \right) \notag \, .
\end{align}

\end{proof}

The previous proposition linearized the full gradient $\nabla \cF$ along a conformal variation.   The next corollary   linearizes the projection $\nabla_1 \cF$ of the gradient to $\cA_1$.

\begin{Cor}	\label{c:lF}
The first variation of $\nabla_1 \cF$ along the path $(b_{\infty}^{-2} g_t, b_{\infty} \, \e^{tv_t})$ where
$b_{\infty}^{-2} g_t' = \phi \, b_{\infty}^{-2} \, g_0$ and $v_t' = v'$ can be written as 
\begin{align}
	 ( \nabla_1 \cF)' =  \left( \bar{f}_1 \, g_0 , \bar{f}_2 \right) + \left( \Hess_{\bar{f}_3} , 0 \right)  
	\, ,
\end{align}
where $\bar{f}_1$, $\bar{f}_2$ and $\bar{f}_3$ are functions.

\end{Cor} 

\begin{proof}
Set $\bar{g}_t = b_{\infty}^{-2} \, g_t$; we omit the subscript when the meaning is clear.
Within this proof,   $| \cdot |$ is the pointwise norm and $\| \cdot \|$ is the $L^2$ norm, while $\langle \cdot , \cdot \rangle$ 
is the $L^2$ inner product.

Since $\cA_1$ is a level set of the functional $A_1$,    the projection $\nabla_1 \cF$ of   $\nabla \cF$ is  
\begin{align}
	\nabla_1 \cF = \nabla \cF - \langle \nabla \cF , \nabla A_1 \rangle \, \frac{ \nabla A_1}{\|\nabla A_1\|^2} \, .
\end{align}
It follows that{\footnote{The gradients are computed with the fixed $L^2$ inner product $\langle \cdot , \cdot \rangle$ induced by the background
metric $\bar{g}_0$.}}
\begin{align}	\label{e:backii}
	(\nabla_1 \cF)'  &= (\nabla \cF )' - \langle (\nabla \cF )' , \nabla A_1 \rangle \, \frac{ \nabla A_1}{\|\nabla A_1 \|^2} 
	- \langle \nabla \cF , (\nabla A_1)' \rangle \, \frac{ \nabla A_1}{\|\nabla A_1 \|^2} 
	- \langle \nabla \cF , \nabla A_1 \rangle \, \frac{ (\nabla A_1)'}{ \|\nabla A_1 \|^2} \notag \\
	&\qquad
	+ 2\, \langle \nabla \cF , \nabla A_1 \rangle \, \langle \nabla A_1 , (\nabla A_1)' \rangle
	 \, \frac{ \nabla A_1}{ \|\nabla A_1 \|^4}  \, .
\end{align}
We next calculate $\nabla \cF$, $\nabla A_1$ and $(\nabla A_1)'$ at $t=0$.  First, 
Corollary \ref{c:gradF} gives at $t=0$
\begin{align}
	(2-n) \, ( \nabla \cF) = (4-n)  \, b_{\infty}^2 \,  \left(   \frac{1}{2}  \,   \bar{g}_0 , 1\right)  
	\, .
\end{align}
Next, Corollary \ref{c:gA1} gives that
the gradient of $A_1$ at $t$ is given by $\nabla A_1 = \left( \frac{1}{2} \, \Psi (\bar{g}) , 1 \right) \, \nu $.  In particular,  
at $t=0$, we have
    \begin{align}
	 \nabla A_1 &= \left( \frac{1}{2} \, \bar{g}_0 , 1 \right)  \, , \\
	 \left( \nabla A_1\right)' &= \left( \frac{n-1}{2} \, \phi + v \right) \, \left( \frac{1}{2} \, \bar{g}_0 , 1 \right)  - \frac{\phi}{2} \, 
	  \left(   \bar{g}_0 , 0 \right) \, ,
\end{align}
where the second equality also used Lemma \ref{l:gradadjust}
to see that
	$[\Psi (\bar{g})]'  = - \bar{g}' $.

 Observe that both $\nabla A_1$ and $(\nabla A_1)'$ give conformal variations of the metric.  The corollary now follows from this, 
 \eqr{e:backii} and 
 Proposition \ref{p:lingradF}.
 \end{proof}

\section{The   action of the diffeomorphism group}

Let $\cD$ be the space of $C^{3,\beta}$ diffeomorphisms on $N$.  The group $\cD$ acts by pull-back on both the space of metrics
 and the space functions, where the metric or function are pulled back by the diffeomorphism.  The tangent space $\cTDiff$ to this action is given by 
  \begin{align}
	\cTDiff & = \{ (\cL_V g_0 , 0) \, | \, V {\text{ is a $C^{3,\beta}$ vector field}} \} \, ,
 \end{align}
 where $\cL_V g_0$ is the Lie derivative of the metric $g_0$ with respect to $V$.  As observed by Berger and Ebin (see, e.g., (b) in corollary $32$ of the appendix of \cite{Be}), 
 it follows   that $\ca$ decomposes  as an orthogonal direct sum
\begin{align}
	\ca = \cTDiff \oplus \ca_1 \, , 
{\text{ where }}
 \ca_1 \equiv \left\{  (h,v) \in C^{2,\beta} \, | \, \delta \, h = 0 \right\} \, .
 \end{align}
  Here, the divergence $\delta$ is computed with respect to $g_0$.  
  
  We will be most interested in the subspace $\ca_1^0 \subset \ca_1$ of variations that are tangent to $\cA_1$, i.e., that preserve the weighted volume constraint
\begin{align}
	\ca^0 &= \left\{ (h,w) \, | \, \int \left( \frac{1}{2} \, \Tr (h) + w \right) \, d\mu_{g_0} = 0 \right\} \, , \\
	\ca^0_1 &= \ca_1 \cap \ca^0 \, .
\end{align}

There are two main results in this section, both related to the action of the diffeomorphism group.  The first is the use of the Ebin-Palais slice theorem to mod out by this action; this is described in subsection \ref{ss:EP}.  The second is the following theorem which shows that the linearization $L_{\cF}$ of $\nabla_1 \cF$ has finite dimensional kernel after we restrict it to 
$\ca_1^0$.  To state this precisely,   
define a bilinear form $B_{\cF}$ on $\cT^0 \times \cT^0$ by setting
 \begin{align}
 	B_{\cF}(x,y) = \langle L_{\cF} \, x , y \rangle \, .
 \end{align}

 \begin{Thm}	\label{p:fker}
 The restriction of  $B_{\cF}$ to $\ca_1^0$  is Fredholm.
 \end{Thm}
 
Here, we  identify the quadratic form with the associated linear operator; it is really the associated linear operator that is Fredholm.  The theorem says  there is a finite dimensional kernel $K \subset \ca_1^0$, so that if $x \in \ca_1^0 \cap K^{\perp}$, then there is a unique $y_x \in \ca_1^0 \cap K^{\perp}$ so that 
\begin{align}
	\langle L_{\cF} \, y_x , z \rangle \equiv B_{\cF}(y_x,z) = \langle  x , z \rangle
	{\text{ for every }} z \in \ca_1^0 \, .
\end{align}
We will prove Theorem \ref{p:fker} at the end of this section.

\subsection{The action of $\cD$}

 Given $\eta$ in the diffeomorphism group $ \cD$,  $(g,w) \in \cA$, 
 and tangent vectors $X,Y$ at a point $p \in M$, then the action of $\eta$  is given by
 \begin{align}
	\eta^{\star} (g)_p (X,Y) &\equiv g_{\eta (p)} ( d\eta (X) , d\eta (Y)) \, , \\
	\eta^{\star} (w) (p) &= w (\eta (p)) \, .
\end{align}
 This action gives a map
 \begin{align}
 	\rho: \cD \times \cA \to \cA \, ,
 \end{align}
 where $\rho (\eta , (g,w) ) \equiv (\eta^{\star} (g) , \eta^{\star} (w) )$.  We will need three elementary 
 properties of this action:
 \begin{itemize}
 \item The action preserves $\cA_1$, i.e., if $\eta \in \cD$ and $\gamma \in \cA_1$, then $\rho (\eta , \gamma) \in \cA_1$.
 \item The action fixes the functional $\cF$.  
 \item The action is isometric with respect to the metric on $\cA$.
 \end{itemize}
 
 Given $\gamma \in \cA$, let $I_{\gamma}$ and $O_{\gamma}$ denote its isotropy group and orbit, respectively
 \begin{align}
 	I_{\gamma} &= \{ \eta \in \cD \, | \,  \rho (\eta , \gamma ) = \gamma \} \, , \\
	O_{\gamma} &= \{ \rho (\eta , \gamma )  \, | \, \eta \in \cD  \} \, .
\end{align}

\subsection{The slice theorem} 		\label{ss:EP}

The Ebin-Palais slice theorem, \cite{E}, gives a way to mod out by the action of the diffeomorphism group $\cD$.  In particular, the version due to Palais (which uses $C^{\beta}$ spaces, rather than Sobolev spaces as in Ebin) gives:
\begin{itemize}
\item A neighborhood $\tilde{\cU}_{1}$ of $0$ in the space of divergence-free symmetric $2$-tensors.
\item A neighborhood $\tilde{\cU}$ of $b_{\infty}^{-2} \, g_0$ in the space of metrics.
\item  A neighborhood $\tilde{\cU}_O$ of $b_{\infty}^{-2} \, g_0$ in the orbit of   $b_{\infty}^{-2} \, g_0$ under $\cD$.
\item A map $\chi :  \tilde{\cU}_O \to \cD$ to a neighborhood of the identity $\Id$  with
 $\chi (b_{\infty}^{-2} \, g_0) = \Id$.
\end{itemize}
so that the mapping  
\begin{align}
	F(u,h) \equiv \rho ( \chi (u) , b_{\infty}^{-2} \, g_0 + h) 
\end{align}
is a diffeomorphism from
 $  \tilde{\cU}_O \times \tilde{\cU}_1$ to $\tilde{\cU}$.  Here we are using a slight abuse of notation, as the action $\rho$ is actually on pairs of metrics and functions, but the meaning is clear.

This slice theorem allows us to mod out by the action of $\cD$ on the space of metrics, but it does not incorporate the second part of the action where the diffeomorphism acts on the function by composition.  When we incorporate the full action, we get  neighborhoods $\cU_1 \subset \ca_1$ of $(0,0)$ and
$\cU \subset \cA$ of $(b_{\infty}^{-2} \, g_0,b_{\infty})$, so that  
\begin{align}
	 F:  \tilde{\cU}_O \times {\cU}_1 \to {\cU} {\text{ is onto}}.
\end{align}
The slice theorem guarantees that this map  hits all of the metrics, so the point here is that it also covers a neighborhood of the constant function $b_{\infty}$ in the space of functions.  To see this, given a diffeomorphism $\eta$, note that push forward by $\eta$ takes $w \circ \eta^{-1}$ to $w$.

The last thing that we need to do here is to restrict to the space $\cA_1$ of normalized pairs of metrics and functions, i.e., to the subset of $\cA$ where $A_1 = \Vol (\partial B_1 (0))$.  

\begin{Lem}	\label{l:exp}
The analytic map $\exp$ on $ \ca^0_1$ given by
\begin{align}
	\exp (h,w) = \left(b_{\infty}^{-2} \, g_0 + h , \, \frac{\Vol (\partial B_1(0))}{A_1 (b_{\infty}^{-2} \, g_0 + h , b_{\infty} \, \e^w)} \, b_{\infty} \, \e^w 
	\right) 
\end{align}
   is a diffeomorphism from a neighborhood of $0$ to a neighborhood of
$(b_{\infty}^{-2} \, g_0 , b_{\infty})$ in $\cA_1$.
\end{Lem}

\begin{proof}
Analyticity follows since linear maps and exponentials are analytic and the functional $A_1$ is analytic since it is given as an integral where the integrand depends analytically.  The is defined so that $A_1 \circ \exp \equiv \Vol (\partial B_1 (0))$, so it automatically lands in $\cA_1$.  Furthermore, $\exp$ takes the origin to $(b_{\infty}^{-2} \, g_0 , b_{\infty})$.

Finally, we will show that $\exp$ is a 
  local diffeomorphism by using the implicit function theorem, \cite{N}.  To do this, first observe that the linearization at the origin is given by
  \begin{align}
  	\frac{d}{dt} \big|_{t=0} \, \exp (t h , tw) = \left(  h ,  b_{\infty} \, w \right) \, , 
  \end{align}
  where we used that the variation is tangent to $\cA_1$ so that the  derivative of $A_1$ vanished.  In particular, the linearization is the identity{\footnote{Recall our convention on the tangent space where we exponentiate the second factor.}}
  and the inverse function theorem applies.
\end{proof}

Combining all of this, we get the following slice theorem:
  
\begin{Cor}	\label{l:slice}
There is a neighborhood $\cU_1' $ of $(b_{\infty}^{-2} \, g_0,b_{\infty})$ in $\cA_1$ and a constant $C$, so that for each
 $y \in \cU_1' $,  there is $y_0 \in \ca^0_1$ and $\eta \in \cD$ so that $y = \rho (\eta , \exp (y_0))$ and $\| \eta \|_{C^{3,\beta}}
 \leq C$.
\end{Cor}

 \subsection{The linearized operator}
 
 We need a little notation.  We will let $\cT_c$ denote the variations corresponding to the conformal directions and $\cT_{tt}$ denote the space of transverse traceless variations, so that
 \begin{align}
 	\cT_{tt} &= \{ (h,0) \in C^{2,\beta} \, | \delta h = 0 {\text{ and }} \Tr (h) =0 \} \, , \\
	\cT_c &= \{ (\phi \, g_0 , v )  \in C^{2,\beta} \} \, , \\
	 \cTDiff & = \{ (\cL_V g_0 , 0) \, | \, V \in C^{3,\beta} {\text{ is a vector field}} \} \, .
 \end{align}
  We add a superscript $0$ to denote the intersection  with $\cT^0$, so that
  $\cT_c^0 \equiv \cT_c \cap \cT^0$ consists of the conformal variations that are tangent to $\cA_1$.
  
  It will be useful to define two additional spaces.  The first is the space $\cT_{c\cD}$ 
  of variations coming from conformal diffeomorphisms
  \begin{align}
  	\cT_{c\cD} \equiv \cT_c \cap \cTDiff \, .
\end{align}
The last space that we will need are the variations $\cTperp^0$ in $\cT_1^0$ that can be generated from conformal variations and diffeomorphisms
\begin{align}
	\cTperp^0 = \cT_1 \cap \left( \cT_c^0 + \cTDiff \right) \, .
\end{align}
Note that $\cTperp^0$ is orthogonal to $\cT_{tt}$ since both $\cT_c$ and $\cTDiff$ are.  The next lemma shows that 
\begin{align}
	\cT_1^0 = \cTperp^0 \oplus \cT_{tt} \, .
\end{align}

\begin{Lem}	\label{l:spaces}
Given any $x \in\cT_1^0$, there exist $x_{tt} \in \cT_{tt}$,  $x_c \in \cT_c^0$,  and $x_{\cD} \in  \cTDiff$ so 
\begin{align}
	x= x_{tt} + x_c + x_{\cD} \,  .
\end{align}
Conversely, given any $x_c \in  \cT_c^0$,  there exists $x_{\cD} \in \cTDiff$ so that 
$x_c + x_{\cD} \in \cT_1^0$.  
\end{Lem}

\begin{proof}
Suppose that $x = (g,v)$.  York's decomposition of Riemannian metrics (see \cite{Y} or theorem $1.4$ in \cite{FM}) gives a transverse traceless
metric
$g_{tt}$, a conformal metric $g_c $, and a $C^{3,\beta}$ vector field $V$  so that
\begin{align}
	g = g_{tt} + g_c + \cL_{V} g_0 \, .
\end{align}
The first claim follows with 
  $x_{tt} = (g_{tt} , 0) \in \cT_{tt}$,  $x_c = (g_c , v) \in \cT_c$,  and $x_{\cD} = (\cL_{V} g_0 , 0) \in  \cTDiff$.  
  To see that $x_c \in \cT_c^0$ (and not just $\cT_c$), note that the spaces
   $\cT_{tt}$ and $\cTDiff$ are   tangent to $\cA_1$.
   
   For the second part, we need to find a vector field $V$ so that
   \begin{align}	\label{e:fredh}
   	\delta \cL_V g_0 = - \delta x_c \, .
   \end{align}
However, $\delta$ is (a multiple of) the adjoint of $\cL_{\left( \cdot \right)} g_0$, so the operator $V \to \delta \cL_V g_0$ is elliptic and, thus, Fredholm, and its kernel consists of Killing vector fields.    In particular, the kernel is orthogonal to the image of $\delta$, so we can solve \eqr{e:fredh} as claimed.
\end{proof}

We will  need the following   standard property of the linearized operator $L_{\cF}$.

\begin{Lem}	\label{l:linearized}
The operator $L_{\cF}$ is symmetric. 
\end{Lem}

\begin{proof}
Let $x(s,t) \in \cA_1$  be  a $2$-parameter variation depending on $s$ and $t$. We have
\begin{align}
	\frac{\partial^2}{\partial s \partial t} \cF (x) = \frac{\partial}{\partial s} \, \langle \nabla_1 \cF(x) , x_s \rangle
	=   \langle L_{\cF} \, x_t , x_s \rangle \, .
\end{align}
Since mixed partials commute, we get that $L_{\cF}$ is symmetric as claimed.
\end{proof}

The next proposition  describes $L_{\cF}$ on the   subspaces $\cT_c^0$, $\cT_{tt}$, $\cTDiff$
and $\cTperp^0$.  Part (D) says that the off-diagonal blocks of $L_{\cF}$ are zero.  The reader should keep in mind that 
$\cT_{tt}$ and $\cT_{\perp}$ are orthogonal and span $\cT_1^0$, but $\cT_{tt}^{\perp}$ is larger than $\cT_{\perp}$.  Namely, this orthogonal complement is done relative to the $L^2$ inner product, so it includes things with lower regularity.

 \begin{Pro}	\label{l:parts}
 The linearization $L_{\cF}$ has the following properties:
   \begin{enumerate}
   \item[(A)] The restriction of $B_{\cF}$  to $\cT_c^0$ is Fredholm.
   \item[(B)] The restriction of $B_{\cF}$  to $\cT_{tt}$ is Fredholm.
\item[(C)] $L_{\cF}$ is identically zero on $  \cTDiff $ and maps  to $ \cTDiff^{\perp}$.
\item[(D)] $L_{\cF}:    \cTperp^0 \to \cT_{tt}^{\perp} $ and
$L_{\cF}:  \cT_{tt} \to \left[   \cTperp^0 \right]^{\perp} $.
\end{enumerate}
 \end{Pro}

 \begin{proof}
 {\bf{Proof of (A)}}:
 To prove this, define the quadratic form $Q_{c}: \cT_c^0 \to \RR$  by
 \begin{align}
 	Q_c (h,v) = \langle L_{\cF} (h,v) , (h,v) \rangle \, .
 \end{align}
 The claim is that the linear operator $L_c$ associated to $Q_c$ is Fredholm.
 
It follows from Theorem  \ref{t:conformal} that if $h = \phi \, b_{\infty}^{-2} \,  g_0$, then
\begin{align}
	Q_c (h,v) = \frac{1}{2-n} \, \langle L_c (\phi , v) , (\phi , v) \rangle \, , 
\end{align}
where the linear operator $L_c$ maps the pair of 
functions
$(\phi , v)$ to the pair of functions
\begin{align}
	  \left(  \frac{n-3}{2}\, \Delta \phi + b_{\infty}^2 \,  \frac{(n-1)(n-3)}{2}\, \phi  + b_{\infty}^2 \,  (n-1) \, v + \Delta v , \, 
	6b_{\infty}^2 \,  v + b_{\infty}^2 \, ( n-1) \phi + \Delta \phi   \right) \, .  \notag
\end{align}
In block form, we can write this as the symmetric linear operator
\begin{equation}	\label{e:matrix}
\left( \begin{array}{cc}
\frac{n-3}{2}\, \left(  \Delta  + b_{\infty}^2 \,  ( n-1) \right)   &
\Delta + b_{\infty}^2 \,  (n -1)
  \\
   \Delta + b_{\infty}^2 \,  (n-1)   & 
      6 b_{\infty}^2 \,   \end{array} \right)  \, .
\end{equation}
It suffices to show that this linear second order operator is   elliptic.  For this, we need only consider the second order part which can be written as
\begin{equation}	 
\left( \begin{array}{cc}
\frac{n-3}{2}  &
1
  \\
  1  & 
     0  \end{array} \right) \, \Delta \, .
\end{equation}
Since $\Delta$ is elliptic, it suffices to show that the matrix in front of $\Delta$ is non-degenerate.{\footnote{There are several different notions of ellipticity for systems.  Weak ellipticity requires only non degeneracy of the matrix and is sufficient to imply elliptic estimates and that the map is Fredholm.  Strong ellipticity requires that the matrix is positive definite; this gives additional properties like the maximum principle.}}
     This follows since the determinant of this matrix is
$-1$.

 {\bf{Proof of (B)}}:
 Define a quadratic form $Q_{tt}: \cT_{tt} \to \RR$ by
 \begin{align}
 	Q_{tt} (h,0) = \langle L_{\cF} (h,0) , (h,0) \rangle \, .
 \end{align}
 It follows from Proposition \ref{p:sv1} that $Q_{tt}$ is given by
\begin{align}
	Q_{tt} (h,0) = \frac{1}{2(n-2)^2} \,  \langle 
	  ( \cL  h , 0) , (h, 0) \rangle   \, , 
\end{align}
where $\cL$ is the Lichnerowicz operator
\begin{align}
	\left( \cL \, h \right)_{ij} = \left( \Delta \, h\right)_{ij}  + 2\, 
	  R_{ikj\ell}  h_{k\ell}  \, . 
\end{align}
Since $\cL$ is elliptic,  the linear operator associated to $Q_{tt}$ is Fredholm, giving (B).

 {\bf{Proof of (C)}}:
Since the diffeomorphism group preserves $\cF$ and, thus, maps critical points to critical points,
it follows that $L_{\cF} :  \cTDiff \to 0$.  Since   $L_{\cF}$ is symmetric
by Lemma \ref{l:linearized}, it follows that
$L_{\cF}$ maps to $  \cTDiff^{\perp}$.

 {\bf{Proof of (D)}}:
 Since $\cT_{tt}$ is perpendicular to both 
  Hessians (these are tangent to $\cTDiff$) and to conformal variations, 
Proposition \ref{p:lingradF}  implies that 
\begin{align} 
	L_{\cF}: \cT_{c} \cap \cT^0 \to \cT_{tt}^{\perp} \, .  
\end{align}
Combining this with (C), we conclude that
\begin{align} 
	L_{\cF}: \cTperp^0 \equiv  \left(  \cTDiff + \cT_{c} \right) \cap \cT^0 \to \cT_{tt}^{\perp} \, .  
\end{align}
The last claim follows from this and the symmetry of $L_{\cF}$.
\end{proof}

 We are now ready to prove Theorem \ref{p:fker}.
 
 \begin{proof}[Proof of Theorem \ref{p:fker}]
 Let $L$ denote the linear operator associated to the restriction of $\cB_{\cF}$ to $\cT_1^0$, so that
 \begin{align}
 	\langle L \, x  , y \rangle = \cB_{\cF} (x,y) \equiv \langle \cL_{\cF} x , y \rangle
 \end{align}
 for $x,y \in \cT_1^0$.    $L$ is  symmetric since $\cL_{\cF}$ is.  Moreover, $L$ maps $\cT_1^0$ to the $C^{\alpha}$ closure of
 $\cT_1^0$.

 To prove the theorem, we will show that:
 \begin{itemize}
 \item $L$ has a finite dimensional kernel $K$.
 \item Given  $x$ in (the $C^{\alpha}$ closure of) $ \cT_1^0 \cap K^{\perp}$, there is a unique
  $y \in   \cT_1^0 \cap K^{\perp}$ so that 
  \begin{align}
  	L \, y = x \, .
\end{align}
 \end{itemize}
We will decompose the map $L$ into blocks according to the orthogonal decomposition
 \begin{align}
 	\cT_1^0 = \cT_{tt} \oplus \cTperp^0  
 \end{align}
 given by Lemma \ref{l:spaces}.
Namely,  (D) in Proposition \ref{l:parts} implies that $L$ ``preserves'' this splitting.{\footnote{The spaces are defined to be in $C^{2,\alpha}$, so the image of $L$ is merely in $C^{\alpha}$; cf. (D)  in Proposition \ref{l:parts}.}}    Let $L_{tt}$ and $L_{\perp}$ denote the restrictions of $L$ to $\cT_{tt}$ and $\cTperp$, respectively.
  Let  $K_{\perp}$ and $K_{tt}$ be the kernels of 
 $K_{\perp}$ and $K_{tt}$, respectively.   By (D)  in Proposition \ref{l:parts}, we have
  \begin{align}
  	K =  K_{\perp} \oplus K_{tt} \, .
  \end{align}
 Since the off-diagonal blocks vanish, we need only show that $L_{tt}$ and $L_{\perp}$ have the two desired properties.
  This is immediate for $L_{tt}$ by (B) in Proposition \ref{l:parts}.  The rest of the proof will be to show that $L_{\perp}$ also has these properties.
  
  We will need a few preliminaries.  Define the map $\Pi_c: \cT^0 \to \cT_c^0$   by
  \begin{align}
  	\Pi_c (g,v) = \left( \frac{\Tr (g)}{n-1} \, \bar{g}_0 , v \right) \, ,
  \end{align}
  where $\bar{g}_0 = b_{\infty}^{-2} \, g_0$ is the background metric and the trace is computed relative to $\bar{g}_0$.
  The map $\Pi_c$ projects the two-tensor to a diagonal two-tensor with the same trace; it is easy to see that this  preserves $\cT^0$.
 Let $L_c$ be the linear map associated to the restriction of $\cB_{\cF}$ to $\cT_c^0$.  If $x_c \in \cT_c^0$, then it is easy to see that
 \begin{align}
 	L_c \, x_c = \Pi_c \, \left( L_{\cF} x_c \right) \, .
 \end{align}
   The map $L_c$ is Fredholm by 
 (A)   in Proposition \ref{l:parts}, so the kernel $K_c$   of $L_c$ is finite dimensional and $L_c$ is invertible on (the $C^{\alpha}$ closure of)
 $K_c^{\perp}$.
 
  Suppose now that $x , y \in \cTperp^0$.  Lemma \ref{l:spaces}
gives   $x_c , y_c \in \cT_c^0$  and $x_{\cD},y_{\cD} \in  \cTDiff$ so that
\begin{align}
	x=  x_c + x_{\cD} {\text{ and }} y=   y_c + y_{\cD} \,  .
\end{align}
Furthermore, $x_c$ and $y_c$ are unique up to elements of $\cT_{c\cD}$.
Part (C)  in Proposition \ref{l:parts} gives that $L_{\cF} x_{\cD} = 0$ and $L_{\cF} x_c$ is orthogonal to $\cTDiff$,  so we get
\begin{align}
	\langle L_{\perp} x , y \rangle =  \langle L (x_c + x_{\cD}) , (y_c + y_{\cD}) \rangle 
	= \langle L_{\cF} x_c , y_c \rangle = \langle L_c x_c , y_c \rangle \, .
\end{align}
Thus, if $ x \in K_{\perp}$, then $x_c$ is in the finite dimensional space $K_c$ (by (A) in 
Proposition \ref{l:parts}).  It follows that $K_{\perp}$ is also finite dimensional.

Next, suppose that $y$ is orthogonal to $K_{\perp}$.  Given any $x \in K_{\perp}$, then since $\cTDiff$ is orthogonal to $\cTperp^0$, we 
 get 
\begin{align}
	0 = \langle x_c + x_{\cD} , y \rangle = \langle x_c , y \rangle = \langle x_c , \Pi_c (y) \rangle \, .
\end{align}
In particular, $\Pi_c (y)$ is orthogonal to $K_c$.  Since $L_c$ is Fredholm ((A) in Proposition \ref{l:parts}), we get
$z_c$ so that $L_c z_c = \Pi_c (y)$.  The second part of Lemma \ref{l:parts} then gives $z_{\cD}$ so that
\begin{align}
	z = z_c + z_{\cD} \in \cT_1^0 \, .
\end{align}
Since $L_{\cF} z_{\cD} = 0$, we have $\Pi_c \left( L  z \right) = L_c z_c = \Pi_c (y)$.  In particular, 
\begin{align}
	(y - L z ) \in \cTperp^0 \subset  \cT_1^0
\end{align}
 is trace-free and transverse, so it belongs to $\cT_{tt}$.  But $\cTperp^0$ is perpendicular to $\cT_{tt}$, so we 
 conclude that $L z = y$ as desired.
 \end{proof}

 \section{A general Lojasiewicz-Simon inequality}

The Lojasiewicz-Simon inequality of \cite{S1} is set up for analytic   functionals that are uniformly convex in the gradient, 
such as the area or energy functionals.  Our functional does not quite fit into this framework since it depends on second derivatives and is not convex, 
so we will need a generalization.  Suppose therefore that we have:
\begin{enumerate}
\item A closed subspace $E$ of $L^2$ maps to a finite dimensional vector space and an analytic functional $G$ defined on a neighborhood $\cO_E$  of $0$ in  $C^{2,\beta} \cap E$.
\item The gradient of $G$ is a $C^1$ map  $\nabla G : \cO_E \to C^{\beta} \cap E$ with $\nabla G (0) = 0$ and 
\begin{align}
	\left\| \nabla G (x)  - \nabla G (y) \right\|_{L^2} \leq C \, \|x - y\|_{W^{2,2}} \, .
\end{align}
  
\item  The linearization $L$ of $\nabla G$ at $0$ is symmetric, bounded from $C^{2,\beta} \cap E$ to $C^{\beta} \cap E$ and from $W^{2,2} \cap E$ to $L^2 \cap E$,  and is Fredholm from $C^{2,\beta} \cap E$ to $C^{\beta} \cap E$.
\end{enumerate} 

One consequence of (3) is that $L$ has finite dimensional kernel $K \subset C^{2,\beta} \cap E$.

\vskip2mm
In (2), $C^1$ means that there is a Frechet derivative at each point and this varies continuously.  Recall that if $V$ is a map from a Banach space $X$ to another Banach space $Y$ and $x \in X$, then a linear map $V_x:X \to Y$ is the Frechet derivative of $V$ at $x$ if
\begin{align}
	\frac{ \| V(x+u) - V(x) - V_x (u) \|_{Y} }{\| u \|_X} \to 0 {\text{ as }} \| u \|_X \to 0 \, .
\end{align}

\vskip2mm
The main result of this section is the following
Lojasiewicz-Simon inequality.

\begin{Thm}	\label{t:LS}
If $G$ satisfies  (1), (2) and (3), 
there exists $\alpha \in (0,1)$ so that for all $x \in E$ sufficiently small
\begin{align}
	|G (x) - G (0)|^{2- \alpha} \leq \| \nabla G (x) \|_{L^2}^2 \, .
\end{align}
\end{Thm}

\vskip2mm
  Let $\Pi_K$ be projection onto $K$ and 
define the mapping $\cN$ by $\cN = \nabla G + \Pi_K$.  The next lemma is Lyapunov-Schmidt reduction.

\begin{Lem}	\label{l:reduction}
There is an open set 
$\cO \subset C^{\beta} \cap E$ about $0$ and a  map $\Phi : \cO \to C^{2,\beta} \cap E$ with $\Phi (0) =0$ so that 
\begin{itemize}
\item $\Phi \circ \cN (x) = x $ and $ \cN \circ \Phi (x) = x$.
\item   $\| \Phi (x) \|_{C^{2,\beta}} \leq C \| x \|_{C^{\beta}}$ and
  $\|\Phi (x) - \Phi (y) \|_{W^{2,2}} \leq C \, \|x-y\|_{L^2}$.
\item The function $f= G \circ \Phi$ is analytic.
\end{itemize}
\end{Lem}

\begin{proof}
Following \cite{S1},   the mapping     $\cN = \nabla G + \Pi_K$ is $C^1$ from $C^{2,\beta} \cap E$ to $C^{\beta} \cap E$ 
and  the Frechet derivative at $0$  is
\begin{align}
	d\cN_0 = L + \Pi_K  \, .
\end{align}
We will show that $d\cN_0 = L + \Pi_K$ is an isomorphism.  First, 
since $L$ is Fredholm and $\Pi_K$ is compact (it has finite rank), the sum
$L + \Pi_K$ is also Fredholm.   Since both $L$ and $\Pi_K$ are symmetric, so is
$L + \Pi_K$ and, thus, 
  it is an isomorphism if and only if it is injective.  
 Finally, since $K$ is the kernel of the symmetric operator
  $L$, we see that $L$ maps to $K^{\perp}$ and, thus,   $L + \Pi_K$ is injective.
  We conclude that $d\cN_0$ is an isomorphism from  $C^{2,\beta} \cap E$ onto  $C^{\beta} \cap E$  
 and the inverse $\left[ d \cN_0 \right]^{-1}$ is a bounded linear mapping from $C^{\beta} \cap E$ to $C^{2,\beta} \cap E$. 
 
 The implicit function theorem (theorem $2.7.2$ in \cite{N}) gives an open set 
$\cO \subset C^{\beta} \cap E$ about $0$ and a $C^1$ inverse map 
$\Phi : \cO  \to C^{2,\beta} \cap E$ with $\Phi (0) = 0$ and  
\begin{align}
	 \Phi \circ \cN (x) = x {\text{ and }}  \cN \circ \Phi (x) = x \, .
\end{align}
The Frechet derivative of $\Phi$ is continuous and is given by
\begin{align}
	d \Phi_y =   \left[ d \cN_{\Phi (y)}
	\right]^{-1} 
	\, .
\end{align}
Since $\Phi$ is $C^1$, the integral mean value theorem  on Banach spaces (see page $34$ in \cite{N}) gives a constant $C$ so that for $x, y \in \cO$
\begin{align}
	\|\Phi (x) - \Phi (y) \|_{C^{2,\alpha}} \leq C \, \|x-y\|_{C^{\beta}} \, .
\end{align}
Using this with $y = \Phi (y) = 0$ gives $\|\Phi (x)   \|_{C^{2,\beta}} \leq C \, \|x\|_{C^{\beta}}$.  The Lipschitz bound for $\Phi$ as a map from 
 $L^2$ to $W^{2,2}$ follows in the same way using the $W^{2,2}$ estimate  for $\nabla G$ and the trivial boundedness of $\Pi_K$ on $L^2$.

Finally, 
  by the remark on page $36$ of \cite{N}, the map $\Phi$ is analytic.
\end{proof}
 
 The next lemma gives a lower bound for   $\nabla G (x)$
 in terms of   $\nabla f$ at  $\Pi_K (x)$.

\begin{Lem}	\label{l:grad}
There exists $C$ so that for every sufficiently small $x \in C^{2,\beta} \cap  E$
\begin{equation}
	\|\nabla f   (\Pi_K (x))  \|_{L^2}^2 \leq C \, \| \nabla G (x)  \|_{L^2}^2 \, .
\end{equation}
\end{Lem}

\begin{proof}
Suppose first that $y \in K$.  Since $f = G \circ \Phi$, 
 it follows from the chain rule and the Lipschitz bound for $\Phi$ that 
 \begin{align}
 	\|\nabla f  (y)   \|_{L^2}^2 \leq C_2 \, \|\nabla G \circ \Phi (y)  \|_{L^2}^2  \, .
 \end{align}
Thus, given any $x$ (not necessarily in $K$), applying this with $y= \Pi_K (x)$ gives
\begin{align}	\label{e:step1of1}
	\|\nabla f (\Pi_K (x))  \|_{L^2}^2  \leq C_2 \, \| \nabla G \circ \Phi \circ \Pi_K (x) \|_{L^2}^2  \, .
\end{align}
This is close to what we want, except that $\nabla G$ is evaluated at $\Phi \circ \Pi_K (x)$ instead of at $x$.

Since $x = \Phi \circ \left( \Pi_K (x) + \nabla G (x) \right)$,   the Lipschitz bounds for $\nabla G$ and $\Phi$ give 
\begin{align}	\label{e:step2of1}
	  \| \nabla G \left(  \Phi \circ \Pi_K (x)\right) - \nabla G (x)  \|_{L^2}   &=  
	   \| \nabla G \left(  \Phi ( \Pi_K (x) )\right) - \nabla G \left(  \Phi (  \Pi_K (x) + \nabla G (x) )\right)  \|_{L^2} \notag \\
	   &\leq C \, 
	   \|    \Phi ( \Pi_K (x) ) -   \Phi (  \Pi_K (x) + \nabla G (x) )    \|_{W^{2,2}} \\
	   &\leq C \,     \| \nabla G  (x)  \|_{L^2}     \, , \notag
\end{align}
completing the proof.
\end{proof}

We next bound  the difference between $G$ and $G \circ \Phi \circ \Pi_K$.

\begin{Lem}	\label{l:F}
There exists $C$ so that for every sufficiently small $x \in C^{2,\beta} \cap  E$
\begin{equation}
	 \left| G (x) - f(\Pi_K (x)) \right| \leq C \, \| \nabla G (x)  \|_{L^2}^2 \, .
\end{equation}
\end{Lem}

\begin{proof}
 Define 
 the one-parameter family   $t \to y_t$  by
\begin{equation}
	y_t = \Pi_K (x) +t \,   \nabla G (x) \, ,
\end{equation}
so that $\Phi (y_1) = x$,  $ y_0  = \Pi_K (x)$, and $\frac{d}{dt} \, y_t = \nabla G (x)$.

Combining the definition of $f$ and  the fundamental theorem of calculus gives
\begin{align}
	 G (x) - f(\Pi_K (x)) &= G (\Phi (y_1)) - f(y_0)=  f(y_1) - f(y_0)
	  =  \int_0^1 \frac{d}{dt} \, f (y_t) \, dt  \notag \\
	 &=  \int_0^1 \langle \nabla f (y_t) ,   \nabla G (x)  \rangle
	  \, dt  \, .
\end{align}
Hence, 
the lemma follows from Cauchy-Schwarz  once we show that
\begin{align}	\label{e:gootcs}
	  \| \nabla f (y_t) \|_{L^2} \leq C \, \| \nabla G (x) \|_{L^2} \, .
\end{align}
To show this, note first that $\nabla f$ is Lipschitz from $L^2$ to $L^2$ by the chain rule (since $\Phi$ is Lipschitz from $L^2$ to $W^{2,2}$ and $\nabla G$ is from $W^{2,2}$ to $L^2$).  In particular, we have
\begin{align}
	\| \nabla f (y_t) - \nabla f (y_1) \|_{L^2} \leq C \, \| y_t - y_1 \|_{L^2} \leq C \, \| \nabla G (x) \|_{L^2} \, .
\end{align}
Finally, \eqr{e:gootcs} follows from this and the fact that $\| \nabla f (y_1) \|_{L^2} \leq C \, \| \nabla G (x) \|_{L^2}$
 which we already established using the chain rule in the proof of the last lemma.
 \end{proof}

 We will now prove the  Lojasiewicz-Simon  inequality using the two lemmas and
  the finite dimensional Lojasiewicz inequality  applied to the
  restriction $f_K \equiv f \big|_K$ of the analytic function $f$ to the finite dimensional vector space $K$.  
  
%

\begin{proof}[Proof of Theorem \ref{t:LS}]
Let $x \in E$ be sufficiently small.  Applying   Lemma \ref{l:grad} and
  the finite dimensional Lojasiewicz inequality (which applies to $f_K$) gives
\begin{align}	\label{e:step1}
	C \, \| \nabla G (x)  \|_{L^2}^2  &\geq \|\nabla f   (\Pi_K (x))  \|_{L^2}^2 \geq
	\left| \nabla f_K (\Pi_K (x)) \right|^2 \geq \left| f_K (\Pi_K (x)) - f_K(0) \right|^{2 -  \alpha  } \notag \\
	& =
	\left| f (\Pi_K (x)) - G(0) \right|^{2 -  \alpha  }
	\, .
\end{align}
The estimate now follows from the triangle inequality and  Lemma \ref{l:F} which gives
\begin{align}
	 \left| f (\Pi_K (x)) - G (x) \right|  \leq	  C \, \| \nabla G (x)  \|_{L^2}^2  \, .
\end{align}
\end{proof}

 \section{The Lojasiewicz-Simon inequality for $\cF$}
 
 Finally, in this section, we will prove that $\cF$ satisfies a Lojasiewicz-Simon inequality.  We cannot argue directly 
 on $\cF$ since the diffeomorphism group creates an infinite dimensional kernel for the linearized operator.  However, the slice theorem of Ebin allows us to mod out by this action and then 
  prove such an inequality which will in turn imply one for $\cF$.

 \subsection{Modding out by the group action}
 
We will  prove a Lojasiewicz-Simon inequality for  $G: \ca_1^0 \to \RR$ given by
\begin{align}
	G (x) = \cF \circ \exp (x) \, ,
\end{align}
where $\exp:  \ca_1^0 \to \cA_1$ is given by  Lemma \ref{l:exp}.
Since $\cF$ and $\exp$ are both analytic, so is $G$.  

By definition, the gradient $\nabla G$ of $G$ is given by
\begin{align}	\label{e:nablaG}
	\langle \nabla G (x) , y \rangle &= \frac{d}{dt} \big|_{t=0} \,  \cF \circ \exp (x+ty) 
	= \langle \nabla_1 \cF (\exp (x) ) , d\exp_{x} (y) \rangle \notag \\
	& = \langle (  d\exp_{x})^t \, \nabla_1 \cF (\exp (x) ) , y \rangle \, ,
\end{align}
where $(d\exp_{x})^t $ is the transpose of $d\exp_{x}$.

\begin{Pro}	\label{p:LS}
A  Lojasiewicz-Simon inequality for $G$ implies one for $\cF$ on $\cA_1$. 
\end{Pro}

\begin{proof}

Corollary \ref{l:slice}
gives a neighborhood $\cU_1' $ of $(b_{\infty}^{-2} \, g_0,b_{\infty})$ in $\cA_1$ and a constant $C$, so that for each
 $y \in \cU_1' $,  there is $y_0 \in \ca^0_1$ and $\eta \in \cD$ so that $y = \rho (\eta , \exp (y_0))$ and $\| \eta \|_{C^{3,\beta}}
 \leq C$.  In particular, 
  the invariance of $\cF$ under the group action
gives that
\begin{align}
	\cF (y) = G (y_0)  \, .
\end{align}
Therefore,  the  Lojasiewicz-Simon inequality for $G$ and \eqr{e:nablaG}   give
\begin{align}
	\left| \cF (y) - \cF (b_{\infty}^{-2} \, g_0,b_{\infty}) \right|^{2-\alpha} &= \left| G ( y_0 ) -  G (0) \right|^{2-\alpha} \leq 
	\| \nabla G (y_0)  \|_{L^2}^2 \notag \\
	& \leq C_{\exp}  \, \| \nabla_1 \cF (\exp (y_0)) \|_{L^2}^2 \, ,
\end{align}
where $C_{\exp}$ comes from the bound for the differential of $\exp$.  

Finally, we need to bound $ \nabla_1 \cF$ at  $\exp (y_0)$
by the value at $y$.  To do this,  let $x$ be tangent to $\cA_1$ at $\exp (y_0)$ and use 
  the invariance of $\cF$ under the action to get that
\begin{align}	 
	\langle \nabla_1 \cF  (\exp (y_0)) , x \rangle &= \frac{d}{dt} \big|_{t=0} \,  \cF  ( \exp (y_0) + t x)
	= \frac{d}{dt} \big|_{t=0} \,  \cF  ( \rho (\eta , \exp (y_0) + t x) )
	  \notag \\
	& = \langle \nabla_1 \cF ( \rho (\eta , \exp (y_0) ) ) , d \rho( \eta , \cdot)_{\exp (y_0)} (x)  \rangle \\
	&= \langle \left(  d \rho( \eta , \cdot)_{\exp (y_0)}  \right)^t \, \nabla_1 \cF (y) ,  x  \rangle \, , \notag 
\end{align}
where the third equality used that the action preserves $\cA_1$ to get $\nabla_1 \cF$ instead of $\nabla \cF$.  
Since $\| \eta \|_{C^{3,\beta}}
 \leq C$, the differential $d \rho( \eta , \cdot)_{\exp (y_0)}$ is bounded independent of $x$ and we conclude that
 \begin{align}	 
	\| \nabla_1 \cF  (\exp (y_0)) \|_{L^2} \leq C' \, \| \nabla_1 \cF (  y ) \|_{L^2} \, ,
\end{align}
completing the proof.
\end{proof}

\subsection{Verifying the properties}

We now need to verify that 
\begin{align}	\label{e:defGc}
	G = \cF \circ \exp: \ca_1^0 \to \RR
\end{align}
   has the  properties needed for Theorem \ref{t:LS}.   Recall that we need $3$ properties:
\begin{enumerate}
\item $G$ is analytic on an open neighborhood $\cO_E$ of $0$ in $C^{2,\beta} \cap \ca_1^0$.
\item   $\nabla G$ is $C^1$ from  $ \cO_E $ to $ C^{\beta}$ with
 $\nabla G (0) = 0$ and 
\begin{align}
	\left\| \nabla G (x)  - \nabla G (y) \right\|_{L^2} \leq C \, \|x - y\|_{W^{2,2}} \, .
\end{align}
\item  The linearization $L_G$ of $\nabla G$ at $0$ is symmetric,
 bounded from $C^{2,\beta} \cap \ca_1^0$ to $C^{\beta}$ and from $W^{2,2} \cap \ca_1^0$ to $L^2$, and
is Fredholm.
 \end{enumerate} 

\begin{Lem}	\label{l:G3}
$G$ defined in \eqr{e:defGc} satisfies (1), (2) and (3).
\end{Lem}

\begin{proof}
We deal with these in order.

\vskip1mm
\noindent
{\bf{Proof of (1)}}:
Property (1) is automatic since   $\exp$ is analytic from $C^{2,\beta}$ to $C^{2,\beta}$ and $\cF$ is analytic 
from $C^{2,\beta}$ to $\RR$.  The analyticity of $\cF$ follows since it is given as an integral of an analytic (in fact algebraic) function of
the weight and the metric, as well as their first and second derivatives (the second derivatives come in from the scalar curvature), cf. \cite{S1}.

\vskip1mm
\noindent
{\bf{Proof of (2)}}:
 Since $\exp (0) = (b_{\infty}^{-2} \, g_0 , b_{\infty})$ is a critical point for $\cF$,     $\nabla G (0)=0$.
By \eqr{e:nablaG},  
\begin{align}	\label{e:todo1}
	\nabla G (x)  = (  d\exp_{x})^t \, \nabla_1 \cF (\exp (x) ) \, .
\end{align}
It follows from the formula \eqr{e:na1F} for $\nabla_1 \cF$ and Corollaries \ref{c:gradF}
and
\ref{c:gA1} that $\nabla_1 \cF$ is $C^1$ from  a neighborhood of $0$ in $C^{2,\beta}$ to $C^{\beta}$ and also Lipschitz (in this neighborhood) from $W^{2,2}$ to $L^2$.
Since $\exp$ is smooth, the formula \eqr{e:todo1} implies that $\nabla G$ has the same properties.

\vskip1mm
\noindent
{\bf{Proof of (3)}}:
The Lipschitz bounds on $\nabla G$ from (2) imply the boundedness of $L_G$ 
from $C^{2,\beta} \cap \ca_1^0$ to $C^{\beta} \cap \ca_1^0$ and from $W^{2,2} \cap \ca_1^0$.
Using \eqr{e:nablaG}, plus the fact that $\exp(0)$ is a critical point for $\cF$, 
we can calculate the linearization $L_G$ of $\nabla G$ at $0$ by
\begin{align}
	\langle L_G (x) , y \rangle &= \frac{d}{dt} \big|_{t=0} \, \langle \nabla G (tx) , y \rangle 
	= \frac{d}{dt} \big|_{t=0} \, \langle \nabla_1 \cF (\exp (tx) ) , d\exp_{tx} (y) \rangle \notag  \\
	&= \langle L_{\cF} (d\exp_0 (x) ) , d\exp_{0} (y) \rangle = \langle L_{\cF} (x) , y \rangle \equiv B_{\cF} (x,y)
	\, ,
\end{align}
where the first equality in the second line used that $d\exp_0$ is the identity on $\ca_1^0 $.    Since $L_{\cF}$ maps to $\ca_1^0$, we conclude that $L_G$ is just the restriction of $L_{\cF}$  to $\ca_1^0$.  Thus, 
$L_G$ is symmetric
since $L_{\cF}$ is and   $L_{G}$ is Fredholm  by   Theorem \ref{p:fker}.

\end{proof}

\appendix

 \section{The weighted total scalar curvature functional}
 
 We will need the following calculations from \cite{Tp} for the changes of geometric quantities under deformation of a metric.
 The derivative at $t=0$ will be denoted by a prime; for example,  $R'$ denotes the derivative of the scalar curvature $R$ at $t=0$.  
 
 \begin{Lem}	\label{l:topping}
 Let $g + t \, h$ be a one-parameter family of metrics on a closed manifold and $u+tv$ a one-parameter family of functions.  Then 
 \begin{align}
 	\left( (g+t\,h)^{ij} \right)' &= - h^{ij}  \, , \\
	\left( \left| \nabla (u+tv) \right|^2 \right)' &= - h (\nabla u , \nabla u) + 2 \, \langle \nabla u , \nabla v \rangle \, , \\
	d\mu' &= \frac{1}{2} \, \Tr (h) \, d \mu \, , \\
	R' &= - \langle \Ric , h \rangle + \delta^2 \, h - \Delta \, \Tr (h) \, , 
 \end{align}
 where $\delta$ is the divergence operator and $\delta^2$ comes from applying it twice.  These will suffice for first variation formulas.  
 
We will need the following additional formulas for the second variation; to simplify notation, we compute these at an orthonormal frame so that we do not need to keep track of upper or lower indices:
 \begin{align}
 	\Ric_{ij} ' &= \frac{1}{2} \, 
	\left( \nabla_i  (\delta \, h)_j +   \nabla_j   (\delta \, h_{i}) +
	\Ric_{ik} h_{jk} + \Ric_{jk} h_{ik} - \Delta h_{ij} 
	- \Hess_{\Tr \, h}
	\right) - R_{ikj\ell} h_{k\ell}  \, , \\
	&\left( \Hess_{u+tv} \right)_{ij}' = \Hess_v - \frac{1}{2} \, 
	\left(\nabla_i \left( \Hess_u \right)_{jk} + 
	\nabla_j \left( \Hess_u \right)_{ik} - \nabla_k \, \left( \Hess_u \right)_{ij}
	\right) \, \nabla_k u
 \end{align}
 \end{Lem}
 
  Note that $h^{ij}$ is given by using the background metric $g$ to raise the indices on the tensor $h$, i.e., $h^{ij} = g^{ik} g^{j\ell} h_{k\ell}$.

   \section{Some computations and identities for the trace free Hessian}
  
  In this appendix, we collect some calculations and identities for the trace free Hessian $B_b$ of $b^2$ where $b^2$ satisfies $\Delta b^2 = 2n \, |\nabla b|^2$ on an $n$-dimensional Ricci flat manifold $(M,g)$.
  
  \subsection{The trace-free Hessian}
Throughout this section,  the function $b$ satisfies 
\begin{equation}
	\Delta b^2 = 2n \, |\nabla b|^2
\end{equation}
 and we 
define the tensor $B_b$ to be the trace-free part of the Hessian of $b^2$, i.e., 
\begin{align}
	B_b =   \Hess_{b^2}   - 2\, |\nabla b|^2 \, g  \, .
\end{align}
We will use that $ \Hess_{b^2} = 2 \, b \, \Hess_b + 2 \, \nabla b \otimes \nabla b$, so that
\begin{align}	\label{e:uhessu}
	2 \, b \, \Hess_b =  \Hess_{b^2} - 2 \, \nabla b \otimes \nabla b = B_b + 2\, \left(   |\nabla b|^2 \, g  
	-   \nabla b \otimes \nabla b  \right) \, .
\end{align}

The next lemma computes the gradient of $|\nabla b|^2$ in terms of $B_b$.

\begin{Lem}	\label{l:k1}
We have   $b \, \nabla |\nabla b|^2 = B_b(\nabla b)$, where $B_b(\nabla b)$ is   given by
$\langle B_b(\nabla b) , v \rangle \equiv B_b(\nabla b , v)$. 
\end{Lem}

\begin{proof}
Since $\nabla |\nabla b|^2 = 2\, \Hess_b (\nabla b , \cdot )$, equation \eqr{e:uhessu} gives
\begin{align}
	b \, \nabla |\nabla b|^2 = 2\, b \, \Hess_b (\nabla b , \cdot ) = 
	 B_b (\nabla b , \cdot) + 2\, \left(   |\nabla b|^2 \, \nabla b  
	-  |\nabla b|^2 \,  \nabla b  \right) = B_b(\nabla b , \cdot) \, .
\end{align}
\end{proof}

\begin{Cor}	\label{c:k1}
We have    $2\, b \, \nabla |\nabla b| =   B_b(\nn)$ where $\nn = \frac{\nabla b}{|\nabla b|}$
 and $4 \, b^2 \, \left| \nabla |\nabla b| \right|^2  = \left| B_b(\nn)  \right|^2$.
\end{Cor}

\begin{proof}
Since $b \, \nabla |\nabla b|^2 = 2 \, b \, |\nabla b| \, \nabla |\nabla b|$, this
  follows from Lemma \ref{l:k1}.
\end{proof}

The next lemma computes the  divergence of $B_b$.

\begin{Lem}	\label{l:k2}
The divergence of $B_b$ is 
\begin{equation}
	\delta B_b =   (2n-2) \, \nabla |\nabla b|^2 = 
	  (2n-2) \,  b^{-1} \, B_b( \nabla b) \, .
\end{equation}
\end{Lem}

\begin{proof}
Fix a point $p \in M$ and let $e_i$ be an orthonormal frame at $p$ with $\nabla_{e_i} e_j (p) = 0$.  
Since $M$ is Ricci flat, we get for any function $w$ that
\begin{align}	\label{e:ricciw}
	\nabla \Delta w = \Delta \nabla w \, .
\end{align}
Using the definition of $B_b$, the fact that $g$ is parallel, and \eqr{e:ricciw} with $w=b^2$ gives
\begin{align}
	\left( \delta B_b \right)_i \equiv (B_b)_{ij,j} = (b^2)_{ijj} - 2\left( |\nabla b|^2 \right)_i =   \left( \Delta b^2 \right)_i 
	- 2\left( |\nabla b|^2 \right)_i \, .
\end{align}
Thus,   $\delta B_b =   \nabla ( \Delta b^2 - 2 \, |\nabla b|^2)$.
The lemma follows since $\Delta b^2 = 2n \, |\nabla b|^2$.
\end{proof}

Using this, we can compute the Laplacian of $|\nabla b|^2$.

\begin{Lem}		\label{l:deltanu}
We have 
\begin{align}
	b^2 \, \Delta |\nabla b|^2  &= \frac{1}{2} \, \left| B_b \right|^2 + (2n-4) \,    B_b (\nabla b , \nabla b)
	  \notag \\
	&= \frac{1}{2} \, \left| B_b \right|^2 + (n-2) \,   \langle \nabla |\nabla b|^2 , \nabla b^2 \rangle
	\, .
\end{align}
\end{Lem}

\begin{proof}
Using the definition of the Laplacian, then Lemma \ref{l:k1}, and then Lemma \ref{l:k2} gives
\begin{align}
	b^2 \, \Delta |\nabla b|^2 &= b^2 \, \dv  \, \nabla |\nabla b|^2 = b^2\,  \dv \, \left( b^{-1} \, B_b(\nabla b) \right)   = b\, \langle \delta B_b , \nabla b \rangle + 
	   \langle B_b , b\, \Hess_b \rangle -   B_b (\nabla b , \nabla b) \notag \\
	 &= 
	 (2n-2) \,   B_b(\nabla b , \nabla b)   +
	\langle B_b , \left\{ \frac{1}{2} \, B_b  +  \left(   |\nabla b|^2 \, g  
	-   \nabla b \otimes \nabla b  \right)  \right\} \rangle -   B_b (\nabla b , \nabla b) \,  .  
\end{align}
Using   $\langle B_b , g \rangle = 0$ since $B_b$ is trace-free, and noting that
$\langle B_b , \nabla b \otimes \nabla b \rangle = B_b(\nabla b , \nabla b)$ gives
\begin{align}
	b^2 \, \Delta |\nabla b|^2  =  	 (2n-4) \, B_b(\nabla b , \nabla b) + 
	\frac{1}{2} \, \left| B_b \right|^2  
	\, .  \notag
\end{align}
 This gives the first equality.  To get the second equality, use that  $b \, \nabla |\nabla b|^2 = B_b(\nabla b)$ by Lemma \ref{l:k1} to write
\begin{align}
	2\, B_b(\nabla b , \nabla b) = 2\, \langle B_b(\nabla b) , \nabla b \rangle = 2\, b \, \langle  \nabla |\nabla b|^2 , \nabla b \rangle
	=  \langle \nabla |\nabla b|^2 , \nabla b^2 \rangle  \, .
\end{align}
\end{proof}

\subsection{The trace-free second fundamental form}

The second fundamental form $\II$ of the level sets of $b$ is given by
\begin{align}
	\II (e_i , e_j) \equiv \langle \nabla_{e_i} \nn , e_j \rangle \, , 
\end{align}
where $e_i$ is a tangent frame and $\nn = \frac{ \nabla b}{|\nabla b|}$ is the unit normal.   It follows that
\begin{align}
	2 \, b \, |\nabla b| \,  \II (e_i , e_j) = \langle \nabla_{e_i}   \nabla b^2   , e_j \rangle
	= \Hess_{b^2} (e_i , e_j)  \, .
\end{align}

\begin{Lem}	\label{l:tracefree}
 The trace-free second fundamental form $\II_0$ and mean curvature $H$ are  
\begin{align}
	2 \, b \, |\nabla b| \,  \II_0  &= B_b  + \frac{B_b(\nn , \nn)}{n-1}     \, g^T \, , \\
	2 \, b \, |\nabla b| \, H & = 2(n-1) \, |\nabla b|^2  - B_b(\nn , \nn) \, , 
\end{align}
where  $\Hess_{b^2}$ and $B_b$ are restricted to   tangent vectors and $g^T$ is the metric on the level set.
\end{Lem}

\begin{proof}
The mean curvature $H$ is the trace of $\II$ over the $e_i$'s.  We have
\begin{align}
	2 \, b \, |\nabla b| \, H &= \Delta b^2 - \Hess_{b^2} (\nn  , \nn) = 2n \, |\nabla b|^2 - \Hess_{b^2} (\nn  , \nn) \notag \\
	&= 2(n-1) \, |\nabla b|^2 + \left( 2\, |\nabla b|^2 -  \Hess_{b^2} (\nn  , \nn)  \right) \\
	&= 2(n-1) \, |\nabla b|^2  - B_b(\nn , \nn) \, , \notag
\end{align}
giving the first claim.
The trace-free second fundamental form $\II_0$ is  
\begin{align}
	2 \, b \, |\nabla b| \,  \II_0 &= 2 \, b \, |\nabla b| \,  \left(  \II - \frac{H}{n-1} \, g^T\right)  = 
	\Hess_{b^2}   -   2\, |\nabla b|^2 \, g^T + \frac{B(\nn , \nn)}{n-1}     \, g^T \notag \\
	&= B_b  + \frac{B_b(\nn , \nn)}{n-1}     \, g^T \, ,
\end{align}
where  $\Hess_{b^2}$ and $B_b$ are restricted to   tangent vectors.
 \end{proof}
 
 \begin{Lem}	\label{l:decompB}
 If $B_0$ denotes the restriction of the tensor $B_b$ to tangent vectors, then 
 \begin{equation}	
 	\left| B_b \right|^2 = \left| B_0 \right|^2 + 2\, \left| B_b(\nn)^T \right|^2 + (B_b(\nn , \nn ))^2 \, .
\end{equation}
 \end{Lem}
 

 \begin{Lem}	\label{l:B0}
If we let $B_0$ denote the restriction of $B_b$ to tangent vectors, then 
 \begin{align}
 	4 \, b^2 \, |\nabla b|^2 \, \left| \II_0 \right|^2 &=   |B_0|^2 - \frac{(B_b(\nn , \nn))^2}{n-1}
	=   \left| B_b \right|^2 
	  -   2\, \left| B_b(\nn)^T \right|^2 - \frac{n}{n-1} \, (B_b(\nn , \nn))^2   \notag \\
	  &=  \left| B_b \right|^2 
	  -   \frac{n}{n-1} \, \left| B_b(\nn)  \right|^2  -   \frac{n-2}{n-1} \, \left| B_b(\nn)^T \right|^2 \, .
 \end{align}
 \end{Lem}
 
 
 The next lemma computes the scalar curvature $R_{g^T}$ where $g^T$ is the induced metric on the level sets of $b$.
 
 \begin{Lem}	\label{l:R1}
 The scalar curvature $R_{g^T}$ is given by
  \begin{align}
 	4 \, b^2 \, |\nabla b|^2 \, R_{g^T} &  = 
	4 (n-1)(n-2) \, |\nabla b|^4 - 4(n-2) |\nabla b|^2 \, B_b(\nn , \nn)  \notag \\
	&\qquad - 
	 \left| B_b \right|^2 
	  +  2 \, \left| B_b(\nn)  \right|^2  
	 \, .
 \end{align}
 \end{Lem}
 
 \begin{proof}
Using that $\II_0$ and $g^T$ are pointwise orthogonal and $\left| g^T \right|^2 = (n-1)$, we get
 \begin{align}
 	\left| \II \right|^2 &= \left| \II_0 + \frac{H}{n-1} \, g^T \right|^2 = 
	\left| \II_0   \right|^2 +  \frac{H^2}{n-1}  \, .
 \end{align}
 Since $M$ is Ricci flat, the Gauss equation gives  
 \begin{align}	\label{e:R1}
 	R_{g^T} = H^2 - \left| \II \right|^2 = H^2 - \left| \II_0   \right|^2 -  \frac{H^2}{n-1} = \frac{n-2}{n-1} \, H^2 - \left| \II_0   \right|^2 \, .
 \end{align}
 To handle this, we first compute $H^2$
 \begin{align}	\label{e:weave}
 	4 \, b^2 \, |\nabla b|^2 \, H^2 & = \left[ 2(n-1) \, |\nabla b|^2  - B_b(\nn , \nn) \right]^2 \notag \\
	&= 4 (n-1)^2 \, |\nabla b|^4 - 4(n-1) |\nabla b|^2 \, B_b(\nn , \nn) + \left( B_b(\nn , \nn) \right)^2 \, .
 \end{align}
 Combining this with the calculation of $| \II_0 |^2$ from Lemma \ref{l:B0} gives
  \begin{align}
 	4 \, b^2 \, |\nabla b|^2 \, R_{g^T} &  = 
	4 (n-1)(n-2) \, |\nabla b|^4 - 4(n-2) |\nabla b|^2 \, B_b(\nn , \nn) + \frac{n-2}{n-1} \, 
	\left( B_b(\nn , \nn) \right)^2 \notag \\
	&\qquad - 
	 \left| B_b \right|^2 
	  +   \frac{n}{n-1} \, \left| B_b(\nn)  \right|^2  +   \frac{n-2}{n-1} \, \left| B_b(\nn)^T \right|^2
	 \, .
 \end{align}
 Finally, simplifying this gives
    \begin{align}
 	4 \, b^2 \, |\nabla b|^2 \, R_{g^T} &  = 
	4 (n-1)(n-2) \, |\nabla b|^4 - 4(n-2) |\nabla b|^2 \, B_b(\nn , \nn)  \notag \\
	&\qquad - 
	 \left| B_b \right|^2 
	  +  2 \, \left| B_b(\nn)  \right|^2  
	 \, .
 \end{align}
\end{proof}

We will also need to compute the Ricci curvature $\Ric^T$ of the level sets.   

\begin{Lem}	\label{l:riccitan}
The Ricci curvature $\Ric^T$ of the level sets is given by
\begin{align}
	b^2 \, \Ric^T = (n-2) \,  |\nabla b|^2 \, g^T + \cE \, , 
\end{align}
where the error term $\cE$ is bounded by a constant times $|B_b| + b\, \left| \nabla B_b \right|$.
\end{Lem}

\begin{proof}
Let ${R}$ and $R^T$ denote the curvature tensor of $M$ and the level set of $b$, respectively.  Choose an orthonormal frame $e_i$ where $e_n = \frac{\nabla b}{|\nabla b|}$ is the unit normal  and $e_1 , \dots , e_{n-1}$ diagonalize the second fundamental form $\II$; let $\lambda_i$ be the eigenvalue corresponding to $e_i$.

 For $i \ne j$ (and $i,j < n$), the Gauss equation gives
\begin{equation}
	R^T_{ijij} = {R}_{ijij} + \lambda_i \, \lambda_j \, .
\end{equation}
Summing over $j<n$ gives the Ricci curvature of the level set in the $e_i , e_i$ direction
\begin{equation}
	\Ric^T_{ii} = \sum_{i \ne j < n} \, \left({R}_{ijij} + \lambda_i \, \lambda_j \right)
		= {\Ric}_{ii} - {R}_{inin} + \lambda_i \, \left( H - \lambda_i \right) \, .
\end{equation}
Using that $M$ is Ricci flat, this becomes
\begin{equation}
	\Ric^T_{ii}  =  - {R}_{inin} + \lambda_i \,  H - \lambda_i^2 
	   \, ,
\end{equation}
where we used that $H= \sum_{i<n} \lambda_i$.  Using that $\lambda_i = \II_0 (e_i , e_i) + \frac{H}{n-1}$, we get 
\begin{align}	\label{e:ricT}
	\Ric^T_{ii}  &=  - {R}_{inin} + H \,  \II_0 (e_i , e_i) + \frac{H^2}{n-1}
	  - \left( \II_0 (e_i , e_i) + \frac{H}{n-1} \right)^2 \, .
\end{align}
Lemma \ref{l:tracefree} gives that
\begin{align}
	\left| H - \frac{(n-1)|\nabla b|}{b} \right| + \left| \II_0 \right|  \leq C \, \frac{\left| B_b \right|}{b} \, .
\end{align}
Using this in \eqr{e:ricT} and noting that both $\left| B_b \right|$ and $b \, |H|$ are uniformly bounded gives
\begin{align}
	 \Ric^T_{ii}    &=  - {R}_{inin} + (n-2) \, \frac{|\nabla b|^2}{b^2}  + \frac{\cE}{b^2}  \, ,
\end{align}
where the error term  $\cE$ is   bounded by a constant times $B_b$.

To complete the proof, we will bound
 the ``radial'' extrinsic curvature term $R_{inin}$  in terms of the trace-free Hessian $B_b$.
Let $e$ be a tangent vector to the level set $b=R$; we can assume that $\nabla_{\nabla b} e = 0$.  The definition of the curvature tensor gives 
\begin{align}
	4\, b^2 \, \langle R(\nabla b , e) \nabla b , e \rangle &= 
	\langle R(\nabla b^2 , e) \nabla b^2 , e \rangle  \notag \\
	&=  \langle \nabla_e \nabla_{\nabla b^2} \nabla b^2 , e \rangle
	- \langle \nabla_{\nabla b^2}  \nabla_e \nabla b^2 , e \rangle
	+ \langle \nabla_{[\nabla b^2 , e]}    \nabla b^2 , e \rangle \\
	&= \langle \nabla_e \Hess_{b^2} (\nabla b^2) , e \rangle
	- \langle \nabla_{\nabla b^2}  \Hess_{b^2}(e) , e \rangle
	-  \Hess_{b^2} \left( \Hess_{  b^2}(e)    , e \right)
	\notag \, .
\end{align}
Next, we use metric compatibility (and the fact that $\nabla_{\nabla b} e = 0$) to get
\begin{align}
	4\, b^2 \, \langle R(\nabla b , e) \nabla b , e \rangle &= 
	 \nabla_e \left(  \Hess_{b^2} (\nabla b^2 , e ) \right) -
	   \Hess_{b^2} (\nabla b^2  , \nabla_e e )
	-   \nabla_{\nabla b^2}  \left( \Hess_{b^2}(e , e ) \right) \\
	&\qquad
	-  \Hess_{b^2} \left( \Hess_{  b^2}(e)    , e \right)
	\notag \, .
\end{align}
Bringing in that $\Hess_{b^2} = B_b + 2 |\nabla b|^2 \, g$, we can write this as
\begin{align}
	4\, b^2 \, \langle R(\nabla b , e) \nabla b , e \rangle &= 
	 \nabla_e \left(  B_{b} (\nabla b^2 , e ) \right) -
	   B_b  (\nabla b^2  , \nabla_e e ) - 2|\nabla b|^2 \, \langle \nabla b^2 , \nabla_e e \rangle 
	   \notag \\
	&\qquad -   \nabla_{\nabla b^2}  \left( B_{b}(e , e ) \right) -  2\, \nabla_{\nabla b^2}    |\nabla b|^2   
		-    B_{b} \left( B_b (e) + 2|\nabla b|^2 e    , e \right) \\
		&\qquad 
		- 2 |\nabla b|^2 \,     B_b (e,e) -  4 |\nabla b|^4
	\notag \, .
\end{align}
The right-hand side has eight terms.  Terms 1, 2, 4, 5, 6 and 7 are all bounded by 
$C \, \left( \left| B_b \right| + b \, \left| \nabla B_b \right| \right)$ (here we also used that
  $\nabla |\nabla b|$ can also be bounded in terms of $B_b$).    Thus, we get that
  \begin{align}
	4\, b^2 \, \langle R(\nabla b , e) \nabla b , e \rangle &= 
	 - 2|\nabla b|^2 \, \langle \nabla b^2 , \nabla_e e \rangle 
	    -  4 |\nabla b|^4 + \cE_0 
	  \, ,
\end{align}
where $\cE_0 \leq C \, \left( \left| B_b \right| + b \, \left| \nabla B_b \right| \right)$.
  Using that $\nabla b$ and $e$ are orthogonal, we get
\begin{align}
	\langle \nabla b^2 , \nabla_e e \rangle = - \langle \nabla_e \nabla b^2 ,   e \rangle 
	= - \Hess_{b^2} (e,e) = - B_{b} (e,e) - 2|\nabla b|^2
	\, , 
\end{align}
 and plugging this in completes the proof.
\end{proof}

\end{document}